\documentclass[12pt]{article}
 \usepackage{amsmath,amssymb,amsthm, mathrsfs}
\usepackage{mathtools}
\usepackage{listings}
\usepackage{longtable}
\usepackage{tikz}
\usepackage{graphicx}
\usepackage{epstopdf}
\usepackage{pdflscape}
\usepackage{multirow}
\usepackage{stmaryrd}
\usepackage{cleveref}
\usepackage{subcaption}
\usetikzlibrary{decorations.pathreplacing}

\renewcommand{\b}{\backslash}
\renewcommand{\to}{\longrightarrow}
\newcommand{\bnu}{{\boldsymbol{\nu}}}
\newcommand{\bgamma}{{\boldsymbol{\gamma}}}

\newcommand{\HH}{\mathbb H_{\mathbf e}}

\theoremstyle{definition}

\theoremstyle{plain}
\newtheorem{theorem}{Theorem}[section]

\newtheorem{lemma}{Lemma}[section]

\title{Evolution of a semidiscrete system modeling the scattering of acoustic waves by a piezoelectric solid}
\author{Thomas S. Brown\footnote{Department of Mathematical Sciences, University of Delaware. {\tt tsbrown@udel.edu}}, 
		Tonatiuh S\'{a}nchez-Vizuet\footnote{Courant Institute of Mathematical Sciences, 
		New York University. {\tt tonatiuh@cims.nyu.edu}}, 
		Francisco-Javier Sayas\footnote{Department of Mathematical Sciences, University of Delaware. {\tt fjsayas@udel.edu}}}
\date{\today}

\begin{document}
\maketitle

\begin{abstract}
We consider a model problem of the scattering of linear acoustic waves in free homogeneous space by an elastic solid. The stress tensor in the solid combines the effect of a linear dependence of strains with the influence of an existing electric field. The system is closed using Gauss's law for the associated electric displacement. Well-posedness of the system is studied by its reformulation as a first order in space and time differential system with help of an elliptic lifting operator. We then proceed to studying a semidiscrete formulation, corresponding to an abstract Finite Element discretization in the electric and elastic fields, combined with an abstract Boundary Element approximation of a retarded potential representation of the acoustic field. The results obtained with this approach improve estimates obtained with Laplace domain techniques. While numerical experiments illustrating convergence of a fully discrete version of this problem had already been published, we demonstrate some properties of the full model with some simulations for the two dimensional case.
\end{abstract}

\section{Introduction}

With a wide range of applications, such as the design of active and passive materials for noise and vibration control and the construction of sensors for non-destructive ultra-sonic testing, the study of the interaction between acoustic waves and solids with piezoelectric properties has been of great interest to researchers in mathematics, physics, and engineering in recent years.  The study of the mechanics of the piezoelectric solid date back to the late 19th century and some examples of work done on these problems can be found in \cite{AkNa2002, Cimatti2004, TaCh2013} to name just a few.  For our purposes, we will use the model of \cite{DeLaOh2009}, which also presents the variational form of the problem. The mathematical justification for the use of this quasi-static approximation in the solid where the electric potential satisfies a time-independent equation has been tackled by \cite{ImJo2012}.  

Introducing an acoustic wave which scatters off of the solid results in the wave-structure interaction problem which is the subject of this article. For the sake of the analysis we will use a first order in space and time formulation as in \cite{HaQiSaSa2015}, whereas for the numerical experiments we will use a surface integral potential representation for the scattered acoustic field as seen in \cite{HsSaWe2015, Sayas2016, ChNa2015}.  This treatment of the acoustics leads to a boundary element formulation for those unknowns and we will use finite elements for the semidiscretization of the piezoelectric solid similar to \cite{BiMa1991}, where the formulation considered is for a purely elastic solid.  While we will not use mixed methods, we would like to remark that it has been shown by \cite{FlKaTrWo2010} that their use for the treatment of the solid results in an equivalent problem.   

Let us now introduce the basic idea of the model we are working on, using some informal notation. (A detailed rigorous description of the model equations is given in Section \ref{ContSect}.) We consider a solid occupying a bounded region of the space $\Omega_-$ and surrounded by an irrotational fluid in its unbounded exterior $\Omega_+$. An acoustic incident wave in free space hits the solid at positive time. The acoustic field caused by scattering and by the effects of wave propagation through the obstacle will be represented by the scattered acoustic potential $u$. As the incident wave hits the obstacle, an elastic wave is triggered in the solid. The piezoelectric behavior is expressed in two ways: stress is the combination of an instantaneous linear operator acting on strain (Hooke's Law) and the effect of electric fields in the solid; Gauss's Law is then imposed for the electric displacement, which combines the electric field  and the elastic strain $\varepsilon$. The electric field will be expressed through an electric potential $\psi$ and the elastic effects will be described by the displacement field $\mathbf u$. Formally we have three equations, namely, a scalar wave equation in the unbounded region, and a vector wave equation as well as an elliptic equation in the bounded domain:
\begin{alignat*}{6}
u_{tt}=\Delta u & \qquad & \Omega_+\times [0,\infty),\\
\rho\mathbf u_{tt}=\mathrm{div}\, (\mathcal C\varepsilon+\mathbf e\nabla\psi) & & \Omega_-\times [0,\infty),\\
\mathrm{div}(\mathbf e^\top \varepsilon-\kappa\nabla \psi)=0 & & \Omega_-\times [0,\infty).
\end{alignat*}
The coupling of the elastic and the electrostatic fields happens through the piezoelectric tensor $\mathbf e$. (When $\mathbf e=0$ the model will reduce to wave-structure interaction.) The coupling of acoustic and piezoelectric dynamics takes place through two transmission conditions involving the normal components of the elastic stress $\boldsymbol\sigma\boldsymbol\nu$ and the acoustic pressure $\rho_f\partial_\nu\dot u$ as well as the normal components of the acoustic and elastic velocities ($\partial_\nu u$ and  $\dot{\mathbf u}\cdot\boldsymbol\nu$ respectively) at the interface. The goal of Section \ref{ContWP}  is the mathematical analysis of this model, which is accomplished by recasting the system of PDE's into a first order system in the spirit of \cite{HaQiSaSa2015}
\begin{equation}\label{eq:0}
\dot U(t) =A_\star U(t)+F(t), \qquad B U(t)=\xi(t), \qquad U(0)=0,
\end{equation}
for a certain operator $A_\star$ that involves a first order in space differential operator and the inverse of an elliptic operator, and boundary operator $B$ that accounts for the conditions at the interface. The unknowns collected in $U$ are related to acoustic pressure, velocity, displacement, purely elastic stress and electric field, while right-hand sides correspond to the influence of the incident wave in the system.

In the next step (Section  \ref{Semidiscrete}) we rewrite the system using a variational formulation for the interior problems and a retarded potential representation for the exterior fields. The interior equations are then approximated in the space variables using a generic Galerkin (FEM) discretization, while the boundary unknowns are approximated with an independent Galerkin (BEM) scheme. We then follow the `template' of the continuous problem to describe the semidiscrete problem in the form \eqref{eq:0}, by redefining all the operators, and to prove discrete well-posedness (Section \ref{SDWP}). The technique is reminiscent of \cite{HaSa2016b} (coupling of the same exterior problem with a scalar wave equation in the interior domain), although we need to overcome several difficulties, arising from the fact that we have an elliptic equation coupled with the system with possible non-homogeneous Dirichlet boundary conditions, as well as from the substitution of a simple scalar equation in the interior domain by a coupled system. While \cite{SaSa2016, HsSaSa2016, HsSaWe2015} had already dealt with different versions of transient wave-structure interaction problems, using Laplace transform techniques based on \cite{BaHa1986, Lubich1994, LaSa2009} , the current approach has several advantages: (a) it presents all situations in a unified form; (b) it provides sharper estimates (lower regularity required and bounds that do not grow with time); (c) it reveals the underlying stability of the semidiscrete system, as reflected by the fact that the solution is the convolution of the data with a group of isometries in a certain Hilbert space. We next show (Section \ref{EE}) that a slight variation of the problem describes the evolution of the error due to semidiscretization.  In Section \ref{AddIss} we include some easy extensions and compare with existing results. Finally, in Section \ref{NumExp}, we illustrate the behavior of the system with some numerical experiments performed by using a multistep Convolution Quadrature \cite{Lubich1994, BaSa2012, HaSa2016} technique applied to the semidiscrete problem. We note that the papers \cite{HsSaSa2016, SaSa2016} already contained numerical experiments on this problem (with and without the piezoelectric coupling) aimed at visualizing the convergence properties of the fully discrete problem. The goal of the numerical simulations here is the illustration of the behavior of solutions to the coupled system.

Instead of presenting the work as a piece of numerical analysis (model equations, discretization, error estimates, numerical experiments), we emphasize the evolutionary equation structure of the three associated problems (continuous problem, semidiscrete problem, error equations) and display all results in full generality, giving conditions on the approximation spaces (there are very few prerequisites in them). Using ideas that go back to \cite{LaSa2009}, we take advantage of the fact that Galerkin semidiscretization of boundary integral equations can be described using exotic transmission conditions (with the exterior acoustic fields partially invading the interior domain) reflecting Galerkin orthogonality and the demand that the unknowns are in discrete spaces.

\textbf{Remark.} The symbol $\lesssim$ will be used to avoid repeated occurrences of $C$, denoting a constant arising from different types of inequalities and whose value is not relevant for the argument.  Whenever discretization is taken into account (this will be made clear by the appearance of an $h$ index, not necessarily related to a discretization mesh) and time is a variable, we will assume that $a \lesssim b$ means that $ a \leq C b$ where $C$ is independent of $h$ and $t$.  Independence of $h$ will mean specifically that the constant does not depend on the particular choice of discrete spaces.    

\section{The continuous problem} \label{ContSect}

\paragraph{Sobolev space preliminaries.}
On an open set $\mathcal O$ we consider the space $L^2(\mathcal O)$ and its vector valued counterpart $\mathbf L^2(\mathcal O):=L^2(\mathcal O)^d$. The symbols $(\cdot,\cdot)_{\mathcal O}$ and $\|\cdot\|_{\mathcal O}$ will be indistinctly used for the inner product and norm of scalar, vector, or matrix valued functions with components in $L^2(\mathcal O)$. In the Sobolev space $H^1(\mathcal O)$ we consider the standard norm $\|u\|_{1,\mathcal O}^2:=\|u\|_{\mathcal O}^2+\|\nabla u \|_{\mathcal O}^2$. We will write $\mathbf H^1(\mathcal O):=H^1(\mathcal O)^d$. Finally, in 
\[
\mathbf{H}(\mathrm{div}, \mathcal{O}):= \{ \mathbf{v} \in \mathbf{L}^2(\mathcal{O}) : \nabla \cdot \mathbf{v} \in L^2(\mathcal{O})\}
\]
we consider the norm $\|\mathbf v\|_{\mathrm{div},\mathcal O}^2:=\|\mathbf v\|_{\mathcal O}^2+\|\nabla\cdot\mathbf v\|_{\mathcal O}^2$. For matrix-valued functions, we consider the spaces
\begin{alignat*}{3}
\mathrm L^2_{\mathrm{sym}}(\mathcal O) & :=\{ \mathrm A \in L^2(\mathcal O)^{d\times d}\,:\, \mathrm A^\top=\mathrm A\quad a.e.\},\\
\mathrm H_{\mathrm{sym}}(\mathrm{div},\mathcal O)
	&:=\{\mathrm A\in \mathrm L^2_{\mathrm{sym}}(\mathcal O) \,:\, \mathrm{div}\,\mathrm A \in \mathbf L^2(\mathcal O)\},
\end{alignat*}
where the divergence operator is applied to the rows of a matrix-valued function, outputting a vector-valued function. All vectors will be taken to be column vectors. 

The geometric setting for this article consists of a bounded
open domain $\Omega_- \subset \mathbb{R}^d$ with Lipschitz boundary $\Gamma$ and exterior $\Omega_+ := \mathbb{R}^d \b \overline{\Omega_-}$.
There is no need for $\Gamma$ to be connected, i.e., we admit $\Omega_-$ to contain cavities.
The setting can be extended to $\Omega_-$ being the finite union of Lipschitz domains with non-intersecting closures. 

The Sobolev spaces $H^{1/2}(\Gamma)$ and the surjective trace operators $\gamma^\pm : H^1(\Omega_\pm)\to H^{1/2}(\Gamma)$ are defined as usual
\cite{AdFo2003, McLean2000}. We will denote $H^{-1/2}(\Gamma)$ to the dual space of $H^{1/2}(\Gamma)$ and $\langle\cdot,\cdot\rangle$ will be used to represent the duality products of $H^{-1/2}(\Gamma)\times H^{1/2}(\Gamma)$ as well as that of the product spaces $\mathbf H^{\pm 1/2}(\Gamma):=H^{\pm 1/2}(\Gamma)^d$. The unit normal vector field on $\Gamma$ will point from $\Omega_-$ to $\Omega_+$ and will be denoted $\bnu$. The exterior-interior weak normal component operators $\gamma^\pm_\bnu:\mathbf H(\mathrm{div},\Omega_\pm)\to H^{-1/2}(\Gamma)$ are given by the standard definitions
\[
\langle \gamma^\pm_\bnu \mathbf v,\gamma^\pm w\rangle
	:= \mp(\nabla\cdot\mathbf v,w)_{\Omega_\pm}\mp (\mathbf v,\nabla w)_{\Omega_\pm}
		\quad\forall w\in H^1(\Omega_\pm).
\]
Both of them are surjective. We can also define $\bgamma_{\bnu}^\pm : \mathrm{H}_{\mathrm{sym}}(\mathrm{div}, \Omega_\pm) \to \mathbf{H}^{-1/2}(\Gamma)$,  given by
\[
\langle \bgamma_{\bnu}^\pm \mathrm{S}, \gamma^\pm \mathbf{w} \rangle 
	:= \mp (\mathrm{S}, \varepsilon(\mathbf{w}))_{\Omega_\pm} 
	 	\mp  (\mathrm{div} \; \mathrm{S}, \mathbf{w})_{\Omega_\pm} \qquad \forall \mathbf w\in \mathbf H^1(\Omega_\pm),
\]
where $\varepsilon(\mathbf{w}) := \frac12(\nabla \mathbf{w} + (\nabla \mathbf{w})^\top)$ is the symmetric gradient.  Note that we can also define the action of the $\bgamma_{\bnu}^\pm$ above using $\nabla \mathbf{w}$ rather than $\varepsilon(\mathbf{w})$ and obtain an equivalent definition. These operators are also surjective. 

 We assume that $\Gamma$ is partitioned into two non-overlapping relatively open sets $\Gamma_N$ and $\Gamma_D$ with non-trivial $\Gamma_D$ with the intention of imposing Neumann and Dirichlet boundary conditions respectively on these two subsets of $\Gamma = \overline{\Gamma}_N \cup \overline{\Gamma}_D$. To this end we also need to introduce the appropriate Sobolev spaces in which these conditions will live.  For $u \in H^1(\Omega_-)$ we define $\gamma_D u:= \gamma u \big|_{\Gamma_D}$ and the following spaces \cite{McLean2000}:
\begin{alignat*}{4}
H^{1/2}(\Gamma_D) &:= \{ \gamma_D u: u \in H^1(\Omega_-) \}, &\qquad H^1_D(\Omega_-) &:= \{u \in H^1(\Omega_-) : \gamma_D u = 0\}, \\
\widetilde{H}^{1/2}(\Gamma_N) &:= \{ \gamma u \big|_{\Gamma_N} : u \in H^1_D(\Omega_-)\}, &\qquad H^{-1/2} (\Gamma_N) &:=  (\widetilde{H}^{1/2}(\Gamma_N))^*.
\end{alignat*}
Above and in the sequel, the notation $X^*$ will be used to denote the dual of the space $X$. The duality pairing of $H^{-1/2}(\Gamma_N)$ and $\widetilde{H}^{1/2}(\Gamma_N)$ will be denoted explicitly by $\langle \cdot, \cdot \rangle_{\Gamma_N}$.

\paragraph{Physical coefficients.}
We will consider that the unbounded domain $\Omega_+$ is occupied by a compressible, inviscid, and irrotational fluid and we will deal with acoustic equations in it. The bounded domain $\Omega_-$ describes the equilibrium state of an elastic solid with piezoelectric properties. (Note that for the first set of results, with variable coefficients on both domains, the situation of the solid and the fluid can be reversed. Once we move to boundary-field formulations and their semidiscrete BEM-FEM discretization, the elastic domain will need to be bounded and the acoustic coefficients will be expected to be constant.)

The material properties on the acoustic domain will be determined by the functions
\begin{alignat*}{4}
c &\in L^\infty (\Omega_+) &\qquad & c \geq c_0 >  0 \quad a.e.,\\
\rho_f &\in L^\infty(\Omega_+) &\qquad & \rho_f \geq \rho_{0,f} >0 \quad a.e.,
\end{alignat*}
although we will use the combined coefficients $\kappa_0:=1 / \rho_f$ and $\kappa_1:= 1/ (c^2 \rho_f)$.  From the theoretical point of view, there is no restriction in $\kappa_0$ taking values on the space of symmetric matrices, as long as $\kappa_0$ is almost everywhere uniformly positive definite. 
The mass density in the elastic domain is given by a function
\[
\rho \in L^\infty(\Omega_-) \qquad \rho \geq \rho_0 >0 \quad a.e.,
\]
and elasticity is described by means of the symmetric stiffness tensor \; $\mathcal C: \Omega_- \to \mathcal{B}(\mathbb{R}^{d \times d}, \mathbb{R}^{d \times d})$, where for all $A,B \in \mathbb{R}^{d \times d}$ and almost everywhere in $\Omega_-$ we have
\begin{alignat*}{6}
\|\mathcal C \mathrm{A} \| &\leq C_\Sigma \|\mathrm{A} \|, &\qquad & \\
\mathcal C \mathrm{A} &= 0 && \text{if } \mathrm{A} = -\mathrm{A}^\top,\\
\mathcal C \mathrm{A}:\mathrm{B} &= \mathrm{A}:\mathcal C\mathrm{B}, && \\
 C_{0, \Sigma} \|\mathrm{A}\|^2 &\leq \mathcal C \mathrm{A}: \mathrm{A} && \text{if } \mathrm{A} = \mathrm{A}^\top.
\end{alignat*}
Here $\mathrm A:\mathrm B$ denotes the Frobenius inner product between the matrices $\mathrm A$ and $\mathrm B$,  and all of the matrix norms are those induced by this inner product in $\mathbb{R}^{d \times d}$. The above properties imply that $\mathcal C\mathrm{A} \in \mathrm L^2_{\mathrm{sym}}(\Omega_-)$  for all $\mathrm{A} \in \mathrm L^2_{\mathrm{sym}}(\Omega_-)$.  In some theoretical arguments, we will briefly need to make use of $\mathcal C^{-1}$, so at this time we will define what we mean by this symbol.  Commonly referred to as the compliance tensor, we define $\mathcal C^{-1} : \Omega_- \to \mathcal B(\mathbb R^{d \times d}_{\mathrm{sym}}, \mathbb R^{d \times d}_{\mathrm{sym}})$, where we say that $\mathcal C^{-1} \mathrm A = \mathrm B$ if  
$\mathcal C \mathrm B:\mathrm M = \mathrm A : \mathrm M,$ for all $\mathrm M \in \mathbb R^{d \times d}.$

To incorporate the piezoelectric properties of the solid occupying $\Omega_-$, we need to define the piezoelectric tensor $\mathbf e : \Omega_- \to \mathcal B(\mathbb R^d, \mathbb R^{d \times d})$ and the dielectric tensor $\kappa_\psi: \Omega_- \to \mathcal B(\mathbb R^d, \mathbb R^d)$.  We will also make use of $\mathbf e^\top : \Omega_- \to \mathcal B(\mathbb R^{d\times d}, \mathbb R^d)$.  The entries of these tensors will be functions in $L^\infty (\Omega_-)$ and exhibit the following symmetries
\[
\mathbf e_{kij} = \mathbf e_{kji}, \qquad \qquad \kappa_{\psi,ij} = \kappa_{\psi,ji}.
\]
We also assume that there is a constant $d_0 >0$ such that 
\[
\kappa_\psi \mathbf d \cdot \mathbf d \geq d_0 \|\mathbf d\|^2, \quad \forall \mathbf d \in \mathbb R^d, \quad a.e.
\]
For $\mathbf d \in \mathbb R^d$ and $\mathrm M \in \mathbb R_\mathrm{sym}^{d \times d}$ we define the action of these tensors as
\[
(\mathbf e \mathbf d)_{ij} = \sum_{k=1}^d e_{kij} d_k, \qquad (\mathbf e^\top  \mathrm M)_k = \sum_{i,j=1}^d e_{kij} M_{ij}, \qquad (\kappa_\psi  \mathbf d)_i = \sum_{j=1}^d \kappa_{\psi,ij} d_j.
\]

\paragraph{A strong second order formulation.} For the moment being, we are going to consider classical solutions of a wave-structure interaction problem. We will use notation of the abstract theory of evolution equations where only the time variable is displayed and functions take values on appropriate Sobolev spaces. All differential operators will be applied in the space variables and the upper dot will denote classical differentiation with respect to the time variable. 

We look for a scalar (acoustic) field and another scalar (piezoelectric) field coupled with a vector (elastic) field
\begin{alignat*}{4}
u:[0, \infty) &\to D(\Omega_+)
	:= &&\{ w \in H^1(\Omega_+) : \kappa_0 \nabla w \in \mathbf{H}(\mathrm{div}, \Omega_+)\},\\
(\psi, \mathbf{u}):[0, \infty) &\to \mathbf{D}(\Omega_-)
	:= &&\{(\phi, \mathbf{w}) \in H^1(\Omega_-) \times \mathbf{H}^1(\Omega_-) :\\
		&&&\qquad \mathcal C\varepsilon(\mathbf{w}) + \mathbf e \nabla \phi \in 					\mathrm{H}_\mathrm{sym}(\mathrm{div}, \Omega_-),\\
		&&&\qquad \mathbf e^\top \varepsilon(\mathbf w) - \kappa_\psi \nabla \phi 				
		\in \mathbf H(\mathrm{div},\Omega_-) \},
\end{alignat*}
satisfying the wave equations
\begin{subequations} \label{eq:SecOrdForm1}
\begin{alignat}{6}
\kappa_1 \ddot{u}(t)&= \nabla \cdot \left(\kappa_0 \nabla u \right)(t)  &\qquad&
	\mbox{in $L^2(\Omega_+)$} \quad \forall t\ge 0, \label{eq:SecOrdForm1a}\\
\rho\, \ddot{\mathbf{u}}(t) &= \mathrm{div} \;\left(\mathcal C\varepsilon(\mathbf{u})(t) + \mathbf e \nabla \psi(t)\right) && 
	\mbox{in $\mathbf L^2(\Omega_-)$} \quad \forall t\ge 0,						
	\label{eq:SecOrdForm1b}
\end{alignat}
a divergence-free condition (Gauss' law for the electric displacement)
\begin{equation} \label{eq:SecOrdForm1c}
\nabla \cdot \left(\mathbf e^\top \varepsilon(\mathbf u)(t) - \kappa_\psi \nabla \psi(t) \right) = 0 \qquad \mbox{in $L^2(\Omega_-)$} \quad \forall t \geq 0, 
\end{equation}
the transmission conditions (continuity of velocity and normal stress)
\begin{alignat}{6}
\gamma_{\bnu}^+(\kappa_0 \nabla u)(t) + \beta_1(t) + \gamma^- \dot{\mathbf{u}}(t) \cdot \boldsymbol{\nu} & = 0
&\qquad&
	\mbox{in $H^{-1/2}(\Gamma)$} \quad \forall t\ge 0,
\label{eq:SecOrdForm1d}\\
(\gamma^+\dot{u}(t) + \dot{\beta}_0(t))\boldsymbol{\nu} + \bgamma_{\bnu}^- \left(\mathcal C\varepsilon(\mathbf{u})(t) + \mathbf e \nabla \psi(t)\right) & =
	\mathbf{0} && 			\mbox{in $\mathbf H^{-1/2}(\Gamma)$} \quad \forall t
	\ge 0,\label{eq:SecOrdForm1e}
\end{alignat}
the mixed boundary conditions
\begin{alignat}{6}
\gamma^- \left(\mathbf e^\top \varepsilon(\mathbf u)(t) - \kappa_\psi \nabla
	\psi(t)\right) \cdot \bnu - \eta(t) &= 0  &\qquad& \mbox{in $H^{-1/2}
	(\Gamma_N)$} &\quad &\forall t \ge 0,\label{eq:SecOrdForm1f}\\
\gamma_D \psi(t) - \mu(t) &= 0 && \mbox{in $H^{1/2}(\Gamma_D)$} &&
	\forall t \ge 0,\label{eq:SecOrdForm1g} 
\end{alignat}
and vanishing initial conditions
\begin{equation}
u(0) = 0, \quad \dot{u}(0) = 0, \quad \mathbf{u}(0) = \mathbf{0}, \quad \dot{\mathbf{u}}(0) = \mathbf{0}. \label{eq:SecOrdForm1h}
\end{equation}
\end{subequations}
In \eqref{eq:SecOrdForm1d} and \eqref{eq:SecOrdForm1e},
\[
\beta_0: [0, \infty) \to H^{1/2}(\Gamma), 
\qquad
\beta_1: [0, \infty) \to H^{-1/2}(\Gamma)
\]
are boundary data, representing the trace and normal flux of a known incident wave corresponding to the given physical parameters, whereas in \eqref{eq:SecOrdForm1f} and \eqref{eq:SecOrdForm1g}, 
\[
\eta : [0, \infty) \to H^{-1/2}(\Gamma_N), \qquad \mu: [0, \infty) \to H^{1/2}(\Gamma_D)
\]
are boundary data for the electric displacement and potential. 

\section{Well-posedness} \label{ContWP}

Before stating our result for the well posedness of the above problem, we need the following definitions. For a Banach space $X$ and $k \geq 1$, we define
\[
W^k(X) := \{ f \in \mathcal C^{k-1}([0, \infty), X) : f^{(k)} \in L^1((0,\infty),X), f^{(\ell)}(0) = 0, \quad 0 \leq \ell \leq k-1\},
\]
and
\begin{equation} \label{eq:H}
H_k(f, t| X):= \sum_{\ell = 0}^k  \int_0^t \|f^{(\ell)} (\tau)\|_X \; \mathrm d \tau.
\end{equation}
With the additional definition of 
\[
(\partial^{-1} f)(t):= \int_0^t f(\tau) \; \mathrm{d}\tau,
\]
we are ready to present the following stability bounds. 

\begin{theorem} \label{results1}
For $(\beta_1, \beta_0, \eta, \mu) \in W^1(H^{-1/2}(\Gamma)) \times W^2(H^{1/2}(\Gamma)) \times W^1(H^{-1/2}(\Gamma_N)) \times W^1(H^{1/2}(\Gamma_D))$ problem \eqref{eq:SecOrdForm1} is uniquely solvable and its solution satisfies for all $t\ge 0$
\begin{alignat*}{6}
\|u(t)\|_{1,\Omega_+} + \|\psi(t)\|_{1, \Omega_-} + \|\mathbf u(t)\|_{1, \Omega_-}\lesssim 
	&\;  H_2((\partial^{-1} \beta_1,\beta_0), t |H^{-1/2}(\Gamma)\times H^{1/2}(\Gamma))\\
	&\; + H_1((\eta,\mu), t| H^{-1/2}(\Gamma_N)\times H^{1/2}(\Gamma_D)).
\end{alignat*}
\end{theorem}

\paragraph{On our way to the proof.}  We begin by writing \eqref{eq:SecOrdForm1} as a first order problem, defining the variables
\[
\mathbf{v}:= \partial^{-1} (\kappa_0 \nabla u), \qquad \qquad \mathrm{S}:= \partial^{-1} (\mathcal{C}\varepsilon(\mathbf{u})), \qquad \qquad \mathbf r: = \partial^{-1} \nabla \psi.
\]
The steady-state differential equation \eqref{eq:SecOrdForm1c} and the associated boundary conditions \eqref{eq:SecOrdForm1f}-\eqref{eq:SecOrdForm1g} will be treated as functions $(\mathbf u,\eta,\mu)\mapsto \nabla\psi$ in the following form: consider the operators
\[
L_\Omega : \mathbf H^1(\Omega_-) \longrightarrow \mathbf L^2(\Omega_-), \quad
L_N : H^{-1/2}(\Gamma_N) \longrightarrow \mathbf L^2(\Omega_-), \quad
L_D : H^{1/2}(\Gamma_D) \longrightarrow \mathbf L^2(\Omega_-),
\]
defined by
\begin{subequations} \label{psiProb}
\begin{equation}
L_\Omega\mathbf u+ L_N\eta+L_D\mu =\nabla \psi,
\end{equation}
where $\psi$ satisfies
\begin{alignat}{3}
& \psi \in H^1(\Omega_-),\qquad 
\gamma_D \psi = \mu,\\
& (\kappa_\psi \nabla \psi, \nabla \varphi)_{\Omega_-} = (\varepsilon (\mathbf u), \mathbf e \nabla \varphi)_{\Omega_-} - \langle \eta, \gamma \varphi\rangle_{\Gamma_N} \quad\forall \varphi \in H_D^1(\Omega_-).
\end{alignat}
\end{subequations}
The Sobolev space for the stress
\[
\HH (\mathrm{div}, \Omega_-) := \{(\mathrm S, \mathbf r) \in \mathrm L_\mathrm{sym}^2 (\Omega_-) \times \mathbf L^2(\Omega_-) : \mathrm{div}(\mathrm S + \mathbf e \mathbf r) \in \mathbf L^2(\Omega_-)\},
\] 
is endowed with its natural norm.
With all of the above, we can state the first order formulation of the problem as looking for
\begin{alignat*}{3}
u:[0,\infty) &\to H^1(\Omega_+),\\
\mathbf{v}:[0,\infty) &\to \mathbf{H}(\mathrm{div}, \Omega_+),\\
\mathbf{u}:[0,\infty) &\to \mathbf{H}^1(\Omega_-),\\
(\mathrm{S},  \mathbf r) :[0,\infty) &\to \HH(\mathrm{div},\Omega_-),
\end{alignat*}
satisfying the differential equations
\begin{subequations} \label{eq:FirstOrdForm}
\begin{alignat}{6}
\dot{u}(t) &= \kappa_1^{-1} \nabla \cdot \mathbf v(t) &\quad &\mbox{in $L^2(\Omega_+)$} &\quad &\forall t \ge 0,\\
\dot{\mathbf v}(t) &= \kappa_0 \nabla u(t) && \mbox{in $\mathbf L^2(\Omega_+)$} && \forall t \ge 0,\\
\dot{\mathbf u}(t) &= \rho^{-1} \mathrm{div}(\mathrm S(t) + \mathbf e \mathbf r(t)) && \mbox{in $\mathbf L^2(\Omega_-)$} && \forall t \ge 0,\\
\dot{\mathrm S}(t) &= \mathcal C\varepsilon(\mathbf u)(t) && \mbox{in $\mathrm L_\mathrm{sym}^2(\Omega_-)$} && \forall t \ge 0,\\
\dot{\mathbf r}(t) &= L_\Omega \mathbf u(t) + L_N \eta (t) + L_D \mu(t) && \mbox{in $\mathbf L^2(\Omega_-)$} && \forall t \ge 0,
\end{alignat}
the transmission conditions
\begin{alignat}{6}
\gamma_\bnu^+ \mathbf v(t) + \partial^{-1} \beta_1(t) + \gamma^- \mathbf u(t) \cdot \bnu &= 0 &\qquad & \mbox{in $H^{-1/2}
	(\Gamma)$} &\quad &\forall t \ge 0,\\
(\gamma^+ u(t) + \beta_0(t))\bnu + \bgamma_\bnu^-(\mathrm S(t) + \mathbf e \mathbf r(t)) &= \mathbf 0 && \mbox{in $\mathbf H^{-1/2}
	(\Gamma)$} &&\forall t \ge 0,
\end{alignat}
and homogeneous initial conditions
\begin{equation}
u(0) = 0, \quad \mathbf v (0) = \mathbf 0, \quad \mathbf u (0) = \mathbf 0, \quad \mathrm S(0) = 0, \quad \mathbf r(0) = \mathbf 0 .
\end{equation}
\end{subequations}

\paragraph{Fitting this into the abstract framework.} We now show how this problem can be fit into the framework outlined in \Cref{framework}, and distilled from \cite{HaQiSaSa2015}.  The desired spaces for this problem are
\begin{alignat*}{2}
\mathbb H &:= L^2(\Omega_+) \times \mathbf L^2(\Omega_+) \times \mathbf L^2(\Omega_-) \times (\mathrm L_\mathrm{sym}^2 (\Omega_-) \times \mathbf L^2(\Omega_-)),\\
\mathbb V &:= H^1(\Omega_+) \times \mathbf H(\mathrm{div}, \Omega_-) \times \mathbf H^1(\Omega_-) \times \HH (\mathrm{div}, \Omega_-),\\
\mathbb M &= \mathbb M_2:=  H^{-1/2} (\Gamma) \times \mathbf H^{-1/2}(\Gamma).
\end{alignat*}
For $U:= (u, \mathbf v, \mathbf u, (\mathrm S, \mathbf r)) \in \mathbb H$ we define
\[
\|U\|_{\mathbb H}^2 := (\kappa_1 u, u)_{\Omega_+} + (\kappa_0^{-1} \mathbf v, \mathbf v)_{\Omega_+} + (\rho \mathbf u, \mathbf u)_{\Omega_-} + (\mathcal C^{-1} \mathrm S, \mathrm S)_{\Omega_-} + (\kappa_\psi \mathbf r, \mathbf r)_{\Omega_-},
\]
whereas in $\mathbb V$ and $\mathbb M$ we will use the natural product norms.
The operators $A_\star:\mathbb V\to \mathbb H$ and $B:\mathbb V \to \mathbb M$ are given by
\begin{alignat}{3}
A_\star U & := (\kappa_1^{-1} \nabla \cdot \mathbf v,\; \kappa_0 \nabla u, \; \rho^{-1} \mathrm{div}(\mathrm S + \mathbf e \mathbf r), \; (\mathcal C \varepsilon(\mathbf u), L_\Omega \mathbf u)), \label{eq:A}\\
B U &:= (\gamma_\bnu^+ \mathbf v + \gamma^- \mathbf u \cdot \bnu, \; \gamma^+ u \bnu + \bgamma_\bnu^-(\mathrm S + \mathbf e \mathbf r)), 
\end{alignat}
and make the equivalence $\|A_\star U\|_{\mathbb H} + \|U\|_{\mathbb H} \approx \|U\|_{\mathbb V}$ hold. 
These operators and spaces allow us to write \eqref{eq:FirstOrdForm} in the abstract form 
\[
\dot{U}(t) = A_\star U(t) + F(t), \qquad BU(t)= \Xi (t), \qquad  U(0) = 0, 
\]
with 
\begin{equation}\label{eq:data}
 F = (0, \mathbf 0, \mathbf 0,(0, L_N \eta + L_D \mu)),
 \qquad \Xi = (0, (-\partial^{-1} \beta_1, -\beta_0 \bnu)).
\end{equation}
With the notation $A := A_\star \big|_{\mathrm{Ker} \, B}$ and $D(A): = \mathrm{Ker}\, B$, we go through the process of verifying the hypotheses of the framework. 
\begin{lemma} \label{lemmaDiss1}
For every $U \in D(A)$, we have $(AU,U)_{\mathbb H} = 0$. 
\end{lemma}
\begin{proof}
For each $U \in D(A)$ we have
\begin{alignat*}{2}
(AU, U)_{\mathbb H} &= (\nabla \cdot \mathbf v, u)_{\Omega_+} + (\nabla u, \mathbf v)_{\Omega_+} + (\mathrm{div}(\mathrm S + \mathbf e \mathbf r), \mathbf u)_{\Omega_-} + (\varepsilon(\mathbf u), \mathrm S)_{\Omega_-} + (\kappa_\psi L_\Omega \mathbf u, \mathbf r)_{\Omega_-}\\
&=  -\langle \gamma_\bnu^+ \mathbf v, \gamma^+ u \rangle + (\mathrm{div}(\mathrm S + \mathbf e \mathbf r), \mathbf u)_{\Omega_-} + (\varepsilon(\mathbf u), \mathrm S + \mathbf e \mathbf r)_{\Omega_-}\\
& = -\langle \gamma_\bnu^+ \mathbf v, \gamma^+ u \rangle + \langle \bgamma_\bnu^-(\mathrm S + \mathbf e \mathbf r), \gamma^- \mathbf u \rangle = 0,
\end{alignat*}
which proves the result.
\end{proof}

\begin{lemma} \label{lemmaT1}
The operator $T(u, \mathbf v, \mathbf u, (\mathrm S, \mathbf r)) := (u, -\mathbf v, -\mathbf u, (\mathrm S, \mathbf r))$ is an isometric involution in $\mathbb H$, bijective in $D(A)$, and satisfies $TA = -AT$.
\end{lemma}
\begin{proof}It is straightforward.
\end{proof}

\begin{lemma} \label{lemmaSolve1}
Let $F = (f, \mathbf f, \mathbf g, (\mathrm G, \mathbf h)) \in \mathbb H$ and $\Xi = (\xi, \boldsymbol{\xi}) \in \mathbb M$.  Then there exists a unique $U \in \mathbb V$ such that 
\begin{equation}\label{eq:UAstarU}
U = A_\star U + F, \qquad BU = \Xi,
\end{equation}
and 
\[
\|U\|_\mathbb{V} \le C(\|F\|_{\mathbb H} + \|\Xi\|_{\mathbb M}).
\]
\end{lemma}

\begin{proof}
Uniqueness follows from the linearity of $A_\star$ and $B$ and \Cref{lemmaDiss1}.  We introduce an equivalent variational problem to show existence of solutions to the problem. 
We consider the space $\mathbb U:= H^1(\Omega_+) \times \mathbf H^1(\Omega_-)$, as well as the bounded bilinear form $a: \mathbb{U} \times \mathbb{U} \to \mathbb{R}$ and bounded linear functional $\ell: \mathbb{U} \to \mathbb{R}$  given by
\begin{alignat*}{3}
a((u,\mathbf{u}), (w, \mathbf{w}))&:= (\kappa_1 u,w)_{\Omega_+} + (\kappa_0 \nabla u, \nabla w)_{\Omega_+} + (\rho \mathbf{u}, \mathbf{w})_{\Omega_-} + (\mathcal C \varepsilon(\mathbf{u}), \varepsilon(\mathbf{w}))_{\Omega_-}\\
& \qquad + (\mathbf e L_\Omega \mathbf u, \varepsilon(\mathbf w))_{\Omega_-} - \langle \gamma^- \mathbf{u} \cdot \bnu, \gamma^+ w \rangle + \langle  \gamma^+ u \bnu , \gamma^- \mathbf{w}\rangle,\\
\ell(w, \mathbf{w}) &:= (\kappa_1 f, w)_{\Omega_+} - (\mathbf{f}, \nabla w)_{\Omega_+} + (\rho \mathbf{g}, \mathbf{w})_{\Omega_-} - (\mathrm{G}, \varepsilon(\mathbf{w}))_{\Omega_-} - (\mathbf e \mathbf h, \varepsilon(\mathbf w))_{\Omega_-} \\
&\qquad  - \langle \xi, \gamma^+ w \rangle + \langle \boldsymbol{\xi}, \gamma^- \mathbf{w}\rangle.
\end{alignat*}
Since from the definition of the operator $L_\Omega$ in \eqref{psiProb} we have
\[
(\mathbf e L_\Omega \mathbf u,\varepsilon(\mathbf u))_{\Omega_-}
	=(\kappa_\psi L_\Omega\mathbf u, L_\Omega\mathbf u)_{\Omega_-},
\]
it follows readily that $a$ is coercive in $\mathbb U$.
Therefore, the problem 
\begin{equation}\label{eq:VP}
(u, \mathbf{u}) \in \mathbb{U}, \qquad
a((u, \mathbf{u}),(w, \mathbf{w})) = \ell((w, \mathbf{w})) \qquad \forall (w, \mathbf{w}) \in \mathbb{U},
\end{equation}
is well posed by the Lax-Milgram Lemma and we can easily bound its solution by 
\begin{equation} \label{eq:UBounds1}
\|(u, \mathbf{u})\|_{\mathbb{U}} \leq C(\|(f,\mathbf{f}, \mathbf{g}, \mathrm{G}, \mathbf h)\|_\mathbb{H} + \|(\xi, \boldsymbol{\xi})\|_\mathbb{M}).
\end{equation}
To show that the variational problem is equivalent to \eqref{eq:UAstarU}, we define the quantities
\begin{equation} \label{eq:SecQuant1}
\mathbf{v} := \kappa_0 \nabla u + \mathbf{f} \in \mathbf{L}^2(\Omega_-), \qquad \mathrm{S}:= \mathcal C \varepsilon(\mathbf{u}) + \mathrm{G} \in \mathrm{L}_\mathrm{sym}^2(\Omega_+), \qquad \mathbf r := L_\Omega \mathbf u + \mathbf h,
\end{equation}
and substitute them into the variational problem: 
\begin{alignat}{6} 
\label{eq:VertDot1}
\nonumber
(\kappa_1 u, w)_{\Omega_+} + (\mathbf{v}, \nabla w)_{\Omega_+} + (\rho \mathbf{u}, \mathbf{w})_{\Omega_-} + (\mathrm{S} + \mathbf e \mathbf r, \varepsilon(\mathbf{w}))_{\Omega_-} & \\
- \langle \gamma^- \mathbf{u} \cdot \bnu - \xi, \gamma^+ w\rangle + \langle \gamma^+ u \bnu - \boldsymbol{\xi}, \gamma^- \mathbf{w} \rangle 
& = (\kappa_1 f,w )_{\Omega_+} + (\rho \mathbf{g}, \mathbf{w})_{\Omega_-} \\
\nonumber
& \qquad\qquad   \forall (w, \mathbf{w}) \in \mathbb{U}. 
\end{alignat}
Testing this equation with $(w, \mathbf{w}) \in \mathcal{D}(\Omega_+) \times \mathcal{D}(\Omega_-)^d$ (here $\mathcal D(\mathcal O)$ is the space of infinitely differentiable functions with compact support in $\mathcal O$) and applying elementary theory of distributions \cite{Schwartz1966}, we can show that 
\begin{equation} \label{eq:PrimQuant1}
u = \kappa_1^{-1} \nabla \cdot \mathbf{v} + f, \qquad \qquad \mathbf{u} = \rho^{-1} \mathrm{div} \left(\mathrm{S} + \mathbf e \mathbf r \right) + \mathbf{g},
\end{equation}
hence $\mathbf{v} \in \mathbf{H}(\mathrm{div}, \Omega_+)$ and $(\mathrm{S},\mathbf r) \in \HH(\mathrm{div}, \Omega_-)$.  Substituting \eqref{eq:PrimQuant1} into \eqref{eq:VertDot1}, we obtain 
\begin{alignat*}{6}
(\nabla \cdot \mathbf{v}, w)_{\Omega_+} + (\mathbf{v}, \nabla w)_{\Omega_+} + (\mathrm{div} \left( \mathrm{S} + \mathbf e \mathbf r\right), \mathbf{w})_{\Omega_-} + (\mathrm{S} + \mathbf e \mathbf r, \varepsilon(\mathbf{w}))_{\Omega_-} & \\
- \langle \gamma^- \mathbf{u} \cdot \bnu - \xi, \gamma^+ w\rangle + \langle \gamma^+ u \bnu - \boldsymbol{\xi}, \gamma^- \mathbf{w} \rangle  & = 0  \qquad \forall (w, \mathbf{w}) \in \mathbb{U},
\end{alignat*}
which in turn leads us to 
\[
- \langle \gamma_{\bnu}^+ \mathbf{v} + \gamma^- \mathbf{u} \cdot \bnu - \xi, \gamma^+ w \rangle + \langle \bgamma_{\bnu}^- \left(\mathrm{S} + \mathbf e \mathbf r\right)+ \gamma^+ u \bnu - \boldsymbol{\xi}, \gamma^- \mathbf{w} \rangle = 0 \qquad \forall (w, \mathbf{w}) \in \mathbb{U}.
\]
Since $(\gamma^+,\gamma^-): \mathbb U \to H^{1/2}(\Gamma)\times \mathbf H^{1/2}(\Gamma)$ is surjective, it follows that the latter identity is equivalent to 
\[
\gamma_{\bnu}^+ \mathbf{v} + \gamma^- \mathbf{u} \cdot \bnu = \xi, \qquad \qquad \gamma^+ u \bnu + \bgamma_{\bnu}^- \left(\mathrm{S} + \mathbf e \mathbf r \right) = \boldsymbol{\xi} \qquad \text{on } \Gamma.
\]
The process leading from the variational equation \eqref{eq:VP} and the introduction of additional variables \eqref{eq:SecQuant1} to arrive at \eqref{eq:UAstarU} can be easily reversed, showing that the problems are equivalent. From \eqref{eq:UBounds1}, we have bounds on $\|u\|_{1, \Omega_+}$ and $\|\mathbf{u}\|_{1, \Omega_-}$.  An elementary calculation using \eqref{eq:SecQuant1} and \eqref{eq:PrimQuant1} gives us bounds on $\|\mathbf{v}\|_{\mathrm{div}, \Omega_+}$ and $\|(\mathrm{S}, \mathbf r)\|_{\HH}$ in terms of $u, \mathbf{u}$ and the data, hence the promised bound is valid. 
\end{proof}

With the hypotheses verified, we can immediately use the result from \Cref{bigTheorem} in the context of our problem.  We also point out that the hidden constant from our result in \Cref{results1} depends only on the constant $C$ in \Cref{lemmaSolve1}.   

\begin{proof}[Proof of \Cref{results1}]
Lemmas \ref{lemmaDiss1}, \ref{lemmaT1}, and \ref{lemmaSolve1} are verifications of the hypotheses of \Cref{bigTheorem}. Applying then \Cref{bigTheorem} with data given by \eqref{eq:data}, we have a unique solution with bounds:
\begin{alignat*}{3}
\|u(t)\|_{1, \Omega_+} + \|\mathbf u(t)\|_{1,\Omega_-} &\leq \|U(t)\|_{\mathbb V}\\
&\lesssim \|U(t)\|_{\mathbb H} + \|A_\star U(t)\|_{\mathbb H}\\
&\lesssim \|U(t)\|_{\mathbb H} + \|\dot{U}(t)\|_{\mathbb H} + \|L_N \eta(t)\|_{\Omega_-} + \|L_D \mu(t)\|_{\Omega_-}\\
& \lesssim H_2((\partial^{-1} \beta_1, \beta_0),t |H^{-1/2}(\Gamma) \times H^{1/2}(\Gamma))\\
&\qquad + H_1((L_N \eta, L_D \mu), t |  H^{-1/2}(\Gamma_N) \times H^{1/2}(\Gamma_D))\\
&\qquad + \|L_N \eta(t)\|_{\Omega_-} + \|L_D \mu(t)\|_{\Omega_-}.
\end{alignat*}
Recalling that both $L_N$ and $L_D$ are bounded and that we can estimate 
\[
\|\eta(t)\|_{-1/2, \Gamma_N} \leq \int_0^t \|\dot{\eta}(\tau)\|_{-1/2, \Gamma_N} \; \mathrm d\tau,
\]
with similar results for $\|\mu(t)\|_{1/2, \Gamma_D}$, and observing that from \eqref{psiProb}, we can bound 
\begin{alignat*}{3}
\|\psi(t)\|_{1, \Omega_-} &\lesssim \|\varepsilon (\mathbf u)(t)\|_{\Omega_-} + \|\eta(t)\|_{-1/2, \Gamma_N} + \|\mu(t)\|_{1/2, \Gamma_D}\\
&\lesssim \|\mathcal C^{-1} \dot{\mathrm S}(t)\|_{\Omega_-} + \|\eta(t)\|_{-1/2, \Gamma_N} + \|\mu(t)\|_{1/2, \Gamma_D}\\
&\lesssim \|\dot{U}(t)\|_{\mathbb H} + \|\eta(t)\|_{-1/2, \Gamma_N} + \|\mu(t)\|_{1/2, \Gamma_D}, 
\end{alignat*}
we arrive at the desired estimate. 
\end{proof}

\section{A semidiscrete problem} \label{Semidiscrete}
We now look at a semidiscrete version of the problem where we require the parameters $\kappa_0$ and $\kappa_1$ to be strictly positive constants.  We do this with the intention of using an integral formulation on $\Gamma$ to represent the acoustic field and a volume-variational formulation in $\Omega_-$ to represent the elastic and electric fields. These formulations will then be discretized following a Galerkin approach, which will require allowing $u$ to take values in $H^1(\mathbb R^d \b \Gamma)$ rather than just $H^1(\Omega_+)$, as known since \cite{LaSa2009}.  We introduce the finite dimensional subspaces $Y_h \subseteq H^{1/2} (\Gamma)$ and $X_h \subseteq H^{-1/2} (\Gamma)$, which will be used to implement the discretized transmission conditions. The polar sets
\begin{align*}
Y_h^\circ &:= \{ \eta \in H^{-1/2} (\Gamma) : \langle \eta, \zeta^h \rangle = 0 \quad \forall \zeta^h \in Y_h\},\\
X_h^\circ &:= \{\varphi \in H^{1/2}(\Gamma): \langle \mu^h, \varphi\rangle = 0 \quad \forall \mu^h \in X_h \},
\end{align*}
will be used for shorthand notation of Galerkin testing. 
The notation $\llbracket \gamma \cdot \rrbracket := \gamma^- - \gamma^+ $ represents the jump of the various trace operators across $\Gamma$.
The finite-dimensional space for the electric potential will be denoted $V_h \subset H^1(\Omega_-)$ and we define the spaces 
\[
V_{h,D} := V_h \cap H_D^1(\Omega_-), \qquad \gamma_D V_h:=\{ \gamma_D \varphi : \varphi \in V_h\}.
\]

For technical reasons, we will need to assume that there exists an $h$-uniformly bounded right inverse of the operator $\gamma_D:V_h \to V_{h,D}$. This hypothesis is met by the traditional finite element spaces on arbitrary shape-regular meshes of a polygon/polyhedron.

We define $\mathbf V_h \subset \mathbf H^1(\Omega_-)$ as the finite dimensional approximation space for the elastic displacement.  We keep $\beta_0, \beta_1,$ and $\eta$ as in \Cref{ContSect}, but now we need $\mu^h :[0, \infty) \to V_{h,D}$ such that we approximate $\mu$ with $\mu^h$.  The way in which  this approximation is chosen will affect the final convergence estimates, but it will not change the analysis.  We now state the second order formulation of the problem as trying to find the semidiscrete quantities 
\begin{alignat*}{4}
u^h:[0, \infty) &\to \{w^h \in H^1(\mathbb R^d \b \Gamma): \Delta w^h \in L^2(\mathbb R^d \b \Gamma)\},\\
(\psi^h, \mathbf u^h):[0, \infty) &\to V_h \times \mathbf V_h,  
\end{alignat*}
which for all $ t \geq 0$ satisfy 
\begin{subequations}\label{eq:12}
\begin{alignat}{6}
\kappa_1 \ddot{u}^h(t) &= \kappa_0 \Delta u^h (t), \label{eq:12a}\\
(\rho \ddot{\mathbf u}^h(t),\mathbf w)_{\Omega_-}\!\!
	+ (\mathcal C\varepsilon(\mathbf u^h)(t)+\mathbf e\nabla\psi^h(t),\varepsilon(\mathbf w))_{\Omega_-}\!\!
	& = \langle \llbracket\gamma \dot u^h\rrbracket(t)\!-\!\dot\beta_0(t), \gamma \mathbf w \cdot \bnu \rangle
	 \,\,\, \forall \mathbf w\in \mathbf V_h,  \label{eq:12b}\\
(-\mathbf e^\top\varepsilon(\mathbf u^h)(t)+\kappa_\psi\nabla \psi^h(t),\nabla\phi)_{\Omega_-}
	& = - \langle \eta(t),\gamma \phi\rangle_{\Gamma_N}
	\quad \forall \phi\in V_{h,D}, \label{eq:12c}\\
\gamma_D \psi^h(t) &= \mu^h(t),\label{eq:12d}
\end{alignat}
as well as
\begin{alignat}{5}
(\llbracket \gamma u^h \rrbracket (t),\llbracket \partial_\bnu u^h \rrbracket(t))  &\in Y_h \times X_h, \label{eq:12e}\\
(\gamma \dot{\mathbf{u}}^h(t) \cdot \bnu + \kappa_0 \partial_{\bnu}^+ u^h(t) + \beta_1(t), \gamma^- u^h(t)) &\in Y_h^\circ \times X_h^\circ, \label{eq:12f}
\end{alignat}
with vanishing initial conditions
\begin{equation}
u^h(0) = 0, \quad \dot{u}^h(0) = 0, \quad \mathbf{u}^h(0) = \mathbf{0}, \quad \dot{\mathbf{u}}^h(0) = \mathbf{0}.
\end{equation}
\end{subequations}

We next interpret the exotic transmission conditions \eqref{eq:12e} and \eqref{eq:12f} as a shorthand form for a Galerkin semidiscretization of a retarded potential representation of the acoustic field $u^h$. 

\paragraph{Retarded Potentials and Boundary Integral Operators.}  Following \cite{Sayas2016}, we introduce the retarded potentials and time domain boundary integral operators associated to the wave equation in a weak form.  This can easily be done through the Laplace transform. \\
Given $s \in \mathbb C_+ := \{s : \mathrm{Re }\; s >0\}$ and $(\lambda, \varphi) \in 
H^{-1/2}(\Gamma) \times H^{1/2}(\Gamma)$, the problem
\begin{subequations} \label{BIE1}
\begin{alignat}{4} 
\mathrm U \in H_\Delta^1(\mathbb R \b \Gamma)&, & \qquad  \Delta \mathrm U - s^2 \mathrm U &=0, \\
\llbracket \gamma \mathrm U \rrbracket = \varphi&, &  \llbracket \partial_\bnu \mathrm U \rrbracket &= \lambda,  
\end{alignat}
\end{subequations}
has a unique solution (see \cite[Chapters 2 \& 4]{Sayas2016} and the references therein)
\[
\mathrm U = \mathrm S(s) \lambda - \mathrm D(s) \varphi, 
\]
where $\mathrm S(s)$ and $\mathrm D(s)$ are the single layer and double layer potentials and $H_\Delta^1(\mathbb R \b \Gamma) := \{ 
\xi \in H^1(\mathbb R \b \Gamma) : \Delta \xi \in L^2(\mathbb R \b \Gamma)\}$.  We define the four boundary integral operators
\begin{alignat*}{5}
\mathrm V(s) &:= \gamma \mathrm S(s), &\qquad \mathrm K(s) &:= \tfrac12 (\gamma^+ \mathrm D(s) + \gamma^- \mathrm D(s)),\\
\mathrm K^t (s) &:= \tfrac12 (\partial_\bnu^+ \mathrm S(s) + \partial_\bnu^- \mathrm S(s)), & \mathrm W(s) &:= - \partial_\bnu \mathrm D(s).
\end{alignat*}
We can then represent the solution to  
\[
\kappa_1 \ddot{u} = \kappa_0 \Delta u, \quad \llbracket \gamma u \rrbracket = \varphi, \quad \llbracket \partial_\bnu u \rrbracket = \lambda, 
\]
as 
\[
u = \mathcal S \ast \lambda - \mathcal D \ast \varphi,
\]
where we think of $\lambda$ and $\varphi$ as causal distributions, and we have denoted the Laplace transforms of $\mathcal S$ and $\mathcal D$ respectively by $\mathcal L \{\mathcal S\}: = \mathrm S \left( s /c \right)$, $\mathcal L \{\mathcal D\}: = \mathrm D \left( s/c\right)$ (recall that $c=\sqrt{\kappa_0/\kappa_1}$ is constant now). The following identities hold 
\begin{subequations} \label{eq:jump}
\begin{alignat}{3}
\gamma^\pm u &= \mathcal V \ast \lambda - \left(\pm \tfrac12 \varphi + \mathcal K \ast \varphi \right),\label{eq:jumpA}\\
\partial_\bnu^\pm u &= \mp \tfrac12 \lambda + \mathcal K^t \ast \lambda + \mathcal W \ast \varphi,\label{eq:jumpB}
\end{alignat}
\end{subequations}
where $\mathcal V, \mathcal K, \mathcal K^t$, and $\mathcal W$ are similarly defined through the Laplace transform. 

\paragraph{A boundary-field formulation.} Consider now a solution of \eqref{eq:SecOrdForm1} extended by zero to negative times.  We can then write $ u = \mathcal S \ast \lambda - \mathcal D \ast \varphi$, where $\lambda = - \partial_\bnu^+ u$ and $\varphi = - \gamma^+ u$.  This gives an automatic trivial extension of $u$ to $\Omega_-$ for all times.  The elastic wave equation \eqref{eq:SecOrdForm1b} together with the transmission condition \eqref{eq:SecOrdForm1e}, and the divergence-free law \eqref{eq:SecOrdForm1c} together with the boundary condition \eqref{eq:SecOrdForm1f} lead to the equations
\begin{subequations} \label{eq:TDBIE1} 
\begin{alignat}{4}
(\rho \ddot{\mathbf u},\mathbf w)_{\Omega_-}\!\!
	+ (\mathcal C\varepsilon(\mathbf u)+\mathbf e\nabla\psi,\varepsilon(\mathbf w))_{\Omega_-}\!\!
	& = \langle \dot{\varphi} -\dot\beta_0, \gamma \mathbf w \cdot \bnu\rangle
	 \,\,\, \forall \mathbf w\in \mathbf H^1(\Omega_-),
\label{eq:15a} \\
(-\mathbf e^\top\varepsilon(\mathbf u)+\kappa_\psi\nabla \psi,\nabla\phi)_{\Omega_-}
	& = - \langle \eta,\gamma \phi\rangle_{\Gamma_N}
	\quad \forall \phi\in H^1_D(\Omega_-),
\label{eq:15b}
\end{alignat}
where time differentiation is now in the sense of vector-valued distributions of the time variable. The Dirichlet boundary condition \eqref{eq:SecOrdForm1g} is imposed as an essential condition 
\begin{equation}
\label{eq:15d}
\gamma_D \psi = \mu.
\end{equation}
The transmission condition \eqref{eq:SecOrdForm1d} can be written as 
\begin{equation}\label{eq:15c}
\kappa_0 \left(-\tfrac12 \lambda + \mathcal K^t \ast \lambda + \mathcal W \ast \varphi\right) + \beta_1 + \gamma^- \dot{\mathbf u} \cdot \bnu = 0,
\end{equation}
by using that $u = \mathcal S \ast \lambda - \mathcal D \ast \varphi$ and the jump condition \eqref{eq:jumpB}.  Finally we add the equation 
\begin{equation}\label{eq:15e}
\mathcal V \ast \lambda - \tfrac12 \varphi - \mathcal K \ast \varphi  = 0,
\end{equation}
\end{subequations}
which follows from imposing that $\gamma^- u = 0$ and \eqref{eq:jumpA} to ensure that the potential representation of $u = \mathcal S \ast \lambda - \mathcal D \ast \varphi$ vanishes in $\Omega_-$. The equations \eqref{eq:TDBIE1} have to be understood as a coupled system of equations whose solution $((\psi, \mathbf u), \lambda, \varphi)$ is a causal $\mathbf D(\Omega_-) \times H^{-1/2}(\Gamma) \times H^{1/2}(\Gamma)$-valued distribution. 

\paragraph{The semidiscrete coupled system.}
Our claim is that \eqref{eq:12} is equivalent to a Galerkin semidiscretization in space of \eqref{eq:TDBIE1}. Four finite dimensional spaces
\[
\mathbf V_h\subset \mathbf H^1(\Omega_-),
	\quad
V_h\subset H^1(\Omega_-),
	\quad
X_h\subset H^{-1/2}(\Gamma),
	\quad
Y_h\subset H^{1/2}(\Gamma)
\]
have been chosen, and the unknowns $(\mathbf u^h,\psi^h,\lambda^h,\varphi^h)$ are functions of the time variable with values in the respective four discrete spaces. The time domain (retarded) boundary integral equations \eqref{eq:15c} and \eqref{eq:15e} are substituted by the Galerkin semidiscretizations
\begin{subequations}\label{eq:16}
\begin{alignat}{3}
\langle \gamma \dot{\mathbf u}^h \cdot \bnu + \kappa_0 (\tfrac12 \lambda^h + \mathcal K^t \ast \lambda^h + \mathcal W \ast \varphi^h), \zeta^h \rangle &= -\langle \beta_1, \zeta^h \rangle &\qquad &\forall \zeta^h \in Y_h,\\
\langle \mu^h, \mathcal V \ast \lambda^h + \tfrac12 \varphi^h - \mathcal K \ast \varphi^h\rangle &= 0 & & \forall \mu^h \in X_h,
\end{alignat}
and $(\lambda^h,\varphi^h)$ are then used as input of a retarded potential representation of the acoustic field (that  as a result of discretization now lives on both sides of $\Gamma$)
\begin{equation}\label{eq:16c}
u^h = \mathcal S \ast \lambda^h - \mathcal D \ast \varphi^h.
\end{equation} 
Note that now $\lambda^h = \llbracket \partial_\bnu u^h \rrbracket$ and $\varphi^h = \llbracket \gamma u^h \rrbracket$ and that the Galerkin semidiscrete integral equations can be short-hand-written with the help of the polar sets:
\begin{alignat*}{3}
\gamma \dot{\mathbf u}^h \cdot \bnu + \kappa_0 (\tfrac12 \lambda^h + \mathcal K^t \ast \lambda^h + \mathcal W \ast \varphi^h) + \beta_1 &\in Y_h^\circ,\\
\mathcal V \ast \lambda^h + \tfrac12 \varphi^h - \mathcal K \ast \varphi^h &\in X_h^\circ.
\end{alignat*}
What is left is the Galerkin approximation of equations \eqref{eq:15a} and \eqref{eq:15b} (with the side restriction for the Dirichlet condition \eqref{eq:15d}), which become
\begin{alignat}{6}
(\rho \ddot{\mathbf u}^h,\mathbf w)_{\Omega_-}\!\!
	+ (\mathcal C\varepsilon(\mathbf u^h)+\mathbf e\nabla\psi^h,\varepsilon(\mathbf w))_{\Omega_-}\!\!
	& = \langle \dot\varphi^h\!-\!\dot\beta_0, \gamma \mathbf w \cdot \bnu\rangle
	 \,\,\, \forall \mathbf w\in \mathbf V_h,  \\
(-\mathbf e^\top\varepsilon(\mathbf u^h)+\kappa_\psi\nabla \psi^h,\nabla\phi)_{\Omega_-}
	& = -\langle \eta,\gamma \phi\rangle_{\Gamma_N}
	\quad \forall \phi\in V_{h,D}, \\
\gamma_D \psi^h &= \mu^h.
\end{alignat}
\end{subequations}
The system \eqref{eq:16} is a semidiscrete system that combines: two Galerkin-semidiscrete retarded integral equations, the Galerkin semidiscretization of a second order hyperbolic PDE, and the Galerkin discretization of an elliptic PDE, with the potential post processing \eqref{eq:16c} for the boundary fields to recover the acoustic field in the exterior domain. 

\section{Discrete well-posedness} \label{SDWP}

Similar to what was done in \Cref{ContWP}, we first present our stability (or discrete well-posedness) result and then we will work through the details leading to a proof. We begin by stating the dependence of the semidiscrete solution with respect to the data as a function of time with constants that are independent of $h$. By this we mean that the bounds will be independent on the particular choice of discrete spaces $\mathbf V_h, V_h, X_h$, and $Y_h$. In fact, the only place where the constants in the estimates will even notice the presence of the discretization is through the assumption that $\gamma_D:V_h \to \gamma_DV_h$ has an $h$-uniformly bounded right-inverse.
\begin{theorem} \label{Results3}
For $(\beta_0, \beta_1, \eta, \mu^h) \in W^{\ell+1}(H^{1/2}(\Gamma)) \times W^\ell(H^{-1/2}(\Gamma)) \times W^\ell(H^{-1/2}(\Gamma_N)) \times W^\ell(\gamma_D V_h)$, with $\ell=1$, problem \eqref{eq:12} is uniquely solvable and its solution satisfies
\begin{alignat*}{4}
\|u^h(t)\|_{1,\mathbb R^d \b \Gamma} + \|\psi^h(t)\|_{1, \Omega_-}\\
 \, + \|\mathbf u^h(t)\|_{1, \Omega_-} + \|\varphi^h(t)\|_{1/2, \Gamma} &\lesssim && H_2((\partial^{-1} \beta_1,\beta_0), t | H^{-1/2}(\Gamma) \times H^{1/2}(\Gamma))\\
& &&+ H_1((\eta, \mu^h), t | H^{-1/2}(\Gamma_N) \times H^{1/2}(\Gamma_D))
\end{alignat*} 
for all $t\geq 0$. If $\ell=2$, then additionally it holds
\begin{alignat*}{4}
\|\lambda^h(t) \|_{-1/2, \Gamma} \lesssim & H_3((\partial^{-1} \beta_1,\beta_0), t | H^{-1/2}(\Gamma) \times H^{1/2}(\Gamma))\\
& + H_2((\eta, \mu^h), t | H^{-1/2}(\Gamma_N) \times H^{1/2}(\Gamma_D)).
\end{alignat*}
\end{theorem}

\paragraph{The first order formulation.} To be able to write \eqref{eq:12} in first order form we need to introduce the discrete versions of the divergence, the transpose of $\gamma_\bnu$, and of the operators defined in \eqref{psiProb}.  The discrete divergence $\mathrm{div}_h : \mathrm L_\mathrm{sym}^2 (\Omega_-) \to \mathbf V_h$ is given by 
\[ 
\mathrm{div}_h \mathrm M \in \mathbf V_h, \qquad
(\rho \; \mathrm{div}_h \mathrm M, \mathbf w)_{\Omega_-} = - (\mathrm M, \varepsilon (\mathbf w))_{\Omega_-} \qquad \forall \mathbf w \in \mathbf V_h.
\]
The discrete transposed normal trace $\gamma_{\bnu, h}^t : H^{1/2}(\Gamma) \to \mathbf V_h$ is defined by 
\[
\gamma_{\bnu, h}^t \xi\in \mathbf V_h, \qquad 
(\rho \; \gamma_{\bnu, h}^t \xi, \mathbf w)_{\Omega_-} = \langle \gamma \mathbf w \cdot \bnu, \xi \rangle \qquad \forall \mathbf w \in \mathbf V_h.
\]
Introducing the space $\mathbf W_h:=\nabla V_h=\{\nabla u^h\,:\, u^h\in V_h\}$, we consider the operators 
\[
L_\Omega^h : \mathbf H^1(\Omega_-) \longrightarrow \mathbf W_h,
\quad
L_N^h : H^{-1/2}(\Gamma_N) \longrightarrow \mathbf W_h,
\quad
L_D^h : \gamma_D V_h \longrightarrow \mathbf W_h,
\]
defined by
\begin{subequations} 
\begin{equation}
L_\Omega^h\mathbf u+ L_N^h\eta+L_D^h\mu^h =\nabla \psi^h,
\end{equation}
where $\psi^h$ satisfies
\begin{alignat}{3}
& \psi^h \in V_h,\qquad 
\gamma_D \psi^h = \mu^h,\\
& (\kappa_\psi \nabla \psi^h, \nabla \phi)_{\Omega_-} = (\varepsilon (\mathbf u), \mathbf e \nabla \phi)_{\Omega_-} - \langle \eta, \gamma \phi\rangle_{\Gamma_N} \quad\forall \phi \in V_{h,D}.
\end{alignat}
\end{subequations}
From the above we can see that by taking different combinations of homogeneous data, it is possible to bound each of these separately 
\[
\|L_\Omega^h \mathbf u \|_{\Omega_-} \leq C_1 \|\varepsilon (\mathbf u)\|_{\Omega_-}, \qquad 
\|L_N^h \eta\|_{\Omega_-} \leq C_2 \|\eta\|_{-1/2, \Gamma_N}, \qquad
\|L_D^h \mu^h\|_{\Omega_-} \leq C_3 \|\mu^h\|_{1/2, \Gamma_D},
\]
where all constants are independent of $h$ and the bound on $L_D \mu^h$ uses the hypothesis that $\gamma_D$ has a uniformly bounded right inverse.  Furthermore
\[
(\kappa_\psi L_\Omega^h \mathbf u, L_\Omega^h \mathbf u)_{\Omega_-} = (\varepsilon(\mathbf u), \mathbf e L_\Omega^h \mathbf u)_{\Omega_-}.
\] 
With the variables 
\[
\mathbf v^h:=  \partial^{-1}(\kappa_0 \nabla u^h), \qquad \mathrm S^h := \partial^{-1} \mathcal C \varepsilon(\mathbf u^h), \qquad \mathbf r^h:= \partial^{-1} \nabla \psi^h,
\]
and using the space $\mathrm W_h :=	\{ \mathcal C\varepsilon(\mathbf u^h)\,:\, \mathbf u^h \in \mathbf V_h\}$ for ease of notation, we can write a first order form of the semidiscrete problem \eqref{eq:12} in the following way: we look for 
\[
(u^h, \mathbf v^h, \mathbf u^h, \mathrm S^h, \mathbf r^h):[0, \infty) \to H^1(\mathbb R^d \b \Gamma) \times \mathbf H(\mathrm{div}, \mathbb R^d \setminus \Gamma) \times \mathbf V_h \times \mathrm W_h \times \mathbf W_h,
\]
which for all $t \geq 0$ satisfies
\begin{subequations} \label{eq:FirstOrdForm3} 
\begin{alignat}{7}
\dot{u}^h(t) &= \kappa_1^{-1} \nabla \cdot \mathbf{v}^h(t),\\ 
\dot{\mathbf{v}}^h(t) &= \kappa_0 \nabla u^h(t),\\ 
\dot{\mathbf{u}}^h(t) &= \mathrm{div}_h \left(\mathrm S^h (t) + \mathbf e \mathbf r^h (t) \right) + \gamma_{\bnu,h}^t(\llbracket \gamma u^h \rrbracket (t) - \beta_0 (t)),\\ 
\dot{\mathrm{S}}^h(t) &= \mathcal C \varepsilon (\mathbf{u}^h)(t),\\
\dot{\mathbf r}^h(t) &= L_\Omega^h \mathbf u^h (t) + L_N^h \eta(t) + L_D^h \mu^h (t), 
\end{alignat}
as well as
\begin{alignat}{4}
(\llbracket \gamma u^h \rrbracket (t),\; \llbracket \gamma_\bnu \mathbf v^h \rrbracket(t)) & \in Y_h \times X_h, \label{eq:FirstOrdForm3f}\\
(\gamma \mathbf{u}^h \cdot \bnu (t) + \gamma_{\bnu}^+ \mathbf{v}^h (t)+ \partial^{-1} \beta_1 (t),\;  \gamma^- u^h (t)) &\in Y_h^\circ\times  X_h^\circ\label{eq:FirstOrdForm3i},
\end{alignat}
and have vanishing initial values
\begin{equation}
u^h(0) = 0, \quad \mathbf{v}^h(0) = \mathbf{0}, \quad \mathbf{u}^h(0) = \mathbf{0}, \quad \mathrm{S}^h(0) = 0, \quad \mathbf r(0) = \mathbf 0. 
\end{equation}
\end{subequations}

\paragraph{The abstract framework.}  As with the continuous problem, we now show that \eqref{eq:FirstOrdForm3} has a unique solution by way of fitting this problem into the abstract framework of \Cref{framework}.  First we define the Hilbert spaces 
\begin{alignat*}{3}
\mathbb H &:= L^2 (\mathbb R^d \setminus\Gamma) \times \mathbf L^2 (\mathbb R^2 \setminus\Gamma) \times \mathbf V_h \times \mathrm W_h \times \mathbf W_h,\\ 
\mathbb V &:= H^1(\mathbb R^d \setminus\Gamma) \times \mathbf H(\mathrm{div}, \mathbb R^d \setminus\Gamma) \times \mathbf V_h \times \mathrm W_h \times \mathbf W_h,\\
\mathbb M &= \mathbb M_1 \times \mathbb M_2 := H^{1/2}(\Gamma) \times \Big((Y_h^\circ)^* \times (X_h^\circ)^* \times Y_h^* \times X_h^* \Big), 
\end{alignat*}
where the asterisks denote dual spaces.  In $\mathbb H$ we take the norm
\begin{multline*}
\|U^h\|_\mathbb{H}^2 = \|(u^h, \mathbf v^h, \mathbf u^h, \mathrm S^h, \mathbf r^h) \|_\mathbb{H}^2 : = (\kappa_1 u^h, u^h)_{\mathbb R^d \setminus\Gamma} + (\kappa_0^{-1} \mathbf v^h, \mathbf v^h)_{\mathbb R^d \setminus\Gamma}\\ + (\rho \mathbf u^h, \mathbf u^h)_{\Omega_-} + (\mathcal C^{-1} \mathrm S^h, \mathrm S^h)_{\Omega_-} + (\kappa_\psi \mathbf r^h, \mathbf r^h)_{\Omega_-},
\end{multline*}
and in $\mathbb M$ we take the natural product norm. 

For this problem we have $A_\star: \mathbb V \to \mathbb H$ and $G:\mathbb M_1 \to \mathbb H$  where the actions of these operators on $U^h = (u^h, \mathbf v^h, \mathbf u^h, \mathrm S^h, \mathbf r^h) \in \mathbb V$ and $\xi \in \mathbb M_1$ are defined as
\begin{alignat*}{2}
A_\star U^h &:= \Big(\kappa_1^{-1} \nabla \cdot \mathbf v^h,\;  \kappa_0 \nabla u^h, \; \mathrm{div}_h (\mathrm S^h + \mathbf e \mathbf r^h) + \gamma_{\bnu, h}^t \llbracket \gamma u^h \rrbracket, \; \mathcal C \varepsilon(\mathbf u^h),\;  L_\Omega^h \mathbf u^h\Big),\\
G\xi &:= (0,\; \mathbf 0, \; \gamma_{\bnu, h}^t \xi, \; 0, \; \mathbf 0 ).
\end{alignat*} 
With this definition of $A_\star$, we set the norm on $\mathbb V$ to be given by 
\[
\|U^h\|_{\mathbb V}^2 :=\|U^h\|_{\mathbb H}^2 + \|A_\star U^h\|_{\mathbb H}^2.
\]
Before we introduce the operator $B$, let us explain some notation.  We can consider $\llbracket \gamma u^h \rrbracket \in H^{1/2}(\Gamma) \equiv (H^{-1/2}(\Gamma))^*$, as a functional acting on $H^{-1/2}(\Gamma)$, and therefore by $\llbracket \gamma u^h \rrbracket \big|_{Y_h^\circ}$ we mean the restriction of $\llbracket \gamma u^h \rrbracket :H^{-1/2}(\Gamma) \to \mathbb C$ to $Y_h^\circ\subset H^{-1/2}(\Gamma)$.  In other words, $\llbracket \gamma u^h \rrbracket \big|_{Y_h^\circ}: Y_h^\circ \to \mathbb C$ is an element of $(Y_h^\circ)^*$.  Note that 
\[
\|\llbracket \gamma u^h \rrbracket \big|_{Y_h^\circ}\|_{(Y_h^\circ)^*} \leq \|\llbracket \gamma u^h \rrbracket \|_{H^{-1/2}(\Gamma)^\ast} = \|\llbracket \gamma u^h \rrbracket \|_{1/2,\Gamma},
\]
and that 
\[
\llbracket \gamma u^h \rrbracket \big|_{Y_h^\circ} = 0
\qquad\Longleftrightarrow\qquad
\llbracket \gamma u^h \rrbracket \in Y_h
\] 
since $Y_h$ is closed.  This should give a better idea of the meaning of the first component of $B: \mathbb V \to \mathbb M_2$ whose action on $U^h$ is given by
\[
BU^h := \Big(\llbracket \gamma u^h \rrbracket \big|_{Y_h^\circ}, \; \llbracket \gamma_\bnu \mathbf v^h \rrbracket \big|_{X_h^\circ}, \; (\gamma \mathbf u^h \cdot \bnu + \gamma_\bnu^+ \mathbf v^h)\big|_{Y_h}, \; \gamma^- u^h\big|_{X_h}\Big).
\]
The remaining components of $B$ can be thought of similarly.  As before, we define $D(A) = \mathrm{Ker} \; B$ and $A = A_\star\big|_{D(A)}$.  Next we verify the hypotheses of the abstract framework. 

\begin{lemma} \label{LemmaFEM1}
For each $U^h \in D(A)$, we have $(AU^h, U^h)_{\mathbb H} = 0$.
\end{lemma}
\begin{proof}
A simple calculation gives
\begin{alignat*}{2}
(AU^h, U^h)_{\mathbb H} &= (\nabla \cdot \mathbf v^h, u^h)_{\mathbb R^d \setminus\Gamma} + (\nabla u^h, \mathbf v^h)_{\mathbb R^d \setminus\Gamma}\\
&\qquad  + (\rho (\mathrm{div}_h (\mathrm S^h + \mathbf e \mathbf r^h) + \gamma_{\bnu, h}^t \llbracket \gamma u^h \rrbracket), \mathbf u^h)_{\Omega_-} + (\varepsilon(\mathbf u^h), \mathrm S^h)_{\Omega_-} + (\kappa_\psi L_\Omega^h \mathbf u^h, \mathbf r^h)\\
& = \langle \gamma_\bnu^- \mathbf v^h, \gamma^- u^h \rangle - \langle \gamma_\bnu^+ \mathbf v^h, \gamma^+ u^h \rangle - (\mathrm S^h + \mathbf e \mathbf r^h, \varepsilon(\mathbf u^h))_{\Omega_-}\\
&\qquad + \langle \gamma \mathbf u^h \cdot \bnu, \llbracket \gamma u^h \rrbracket \rangle + (\varepsilon(\mathbf u^h), \mathrm S^h)_{\Omega_-} + (\varepsilon(\mathbf u^h), \mathbf e \mathbf r^h)_{\Omega_-}\\
& = \langle \llbracket \gamma_\bnu \mathbf v^h \rrbracket, \gamma^- u \rangle + \langle \gamma_\bnu^+ \mathbf v^h + \gamma \mathbf u^h \cdot \bnu, \llbracket \gamma u^h \rrbracket \rangle = 0 \qquad \forall U^h \in D(A).
\end{alignat*}
Note that the final equality is a result of the duality pairings of elements of closed spaces with elements of their polar sets.  
\end{proof}

\begin{lemma} \label{LemmaT3}
The operator $T(u^h, \mathbf v^h, \mathbf u^h, \mathrm S^h, \mathbf r^h) := (u^h, -\mathbf v^h, -\mathbf u^h, \mathrm S^h, \mathbf r^h)$ is an isometry from $\mathbb H$ to $\mathbb H$, a bijection from $D(A)$ to $D(A)$, and satisfies $TA = -TA$.  
\end{lemma}
\begin{proof}
It is straightforward.
\end{proof}

\begin{lemma} \label{LemmaSolve3}
For $F = (f, \mathbf f, \mathbf g, \mathrm G, \mathbf h) \in \mathbb H$ and $\Xi = (\xi, \chi) = (\xi, (\psi_1, \psi_2, \eta_1, \eta_2)) \in \mathbb M$, the problem of finding $U^h \in \mathbb V$ such that
\[
U^h = A_\star U^h + G \xi + F, \qquad BU^h =  \chi,
\]
has a unique solution and 
\[
\|U^h\|_{\mathbb V} \leq C(\|F\|_{\mathbb H} + \|\Xi\|_{\mathbb M}),
\]
where $C$ is independent of $h$. 
\end{lemma}
\begin{proof}
Uniqueness follows from the linearity of $A_\star, B$, and $G$ and the calculation in \Cref{LemmaFEM1}.  To show existence of solutions, we work on an equivalent variational formulation.  In order to do that, we define the spaces
\begin{alignat*}{2}
\mathbb U &:= H^1(\mathbb R^d \setminus\Gamma) \times \mathbf V_h,\\
\mathbb U_0 &:= \{(w^h, \mathbf w^h) \in \mathbb U: (\llbracket \gamma w^h \rrbracket, \gamma^- w^h) \in Y_h \times X_h^\circ \}.
\end{alignat*}
and the trace operator 
\[
H^1(\mathbb R^d \setminus\Gamma)\ni u^h \longmapsto \breve\gamma u^h 
:= (\llbracket \gamma u^h \rrbracket \big|_{Y_h^\circ}, \gamma^- u^h \big|_{X_h})\in(Y_h^\circ)^* \times X_h^*.
\]
The operator $\breve \gamma$ is surjective and admits a bounded right inverse whose norm is independent of $h$ (see similar arguments in several proofs in \cite{LaSa2009}). We now define the bilinear form $a:\mathbb U \times \mathbb U \to \mathbb R$ and the linear form $\ell: \mathbb U \to \mathbb R$, where
\begin{alignat*}{3}
a((u^h, \mathbf u^h),(w, \mathbf w)) &:= (\kappa_1 u^h, w)_{\mathbb R^d \setminus\Gamma} + (\kappa_0 \nabla u^h, \nabla w)_{\mathbb R^d \setminus\Gamma}\\
&\qquad  + (\rho \mathbf u^h, \mathbf w)_{\Omega_-} + (\mathcal C \varepsilon(\mathbf u^h), \varepsilon(\mathbf w))_{\Omega_-} + (\mathbf e L_\Omega^h \mathbf u^h, \varepsilon (\mathbf w))_{\Omega_-}\\
&\qquad -\langle \gamma \mathbf w \cdot \bnu, \llbracket \gamma u^h \rrbracket \rangle + \langle \gamma \mathbf u^h \cdot \bnu, \llbracket \gamma w \rrbracket \rangle,\\
\ell((w, \mathbf w)) &:= (\kappa_1 f, w)_{\mathbb R^d \setminus\Gamma} - (\mathbf f, \nabla w)_{\mathbb R^d \setminus\Gamma} + (\rho \mathbf g, \mathbf w)_{\Omega_-} - (\mathrm G, \varepsilon(\mathbf w))_{\Omega_-}\\
&\qquad - (\mathbf e \mathbf h,\varepsilon (\mathbf w))_{\Omega_-} + \langle \gamma \mathbf w \cdot \bnu, \xi \rangle + \langle \psi_2, \gamma^- w \rangle + \langle \eta_1, \llbracket \gamma w \rrbracket \rangle.
\end{alignat*}
Now we have that the variational problem looking for $(u^h, \mathbf u^h) \in \mathbb U$ such that
\begin{subequations}
\begin{alignat}{2} \label{eq:VarEq3}
\breve\gamma u^h &= (\psi_1, \eta_2),\\
a((u^h, \mathbf u^h),(w, \mathbf w)) & = \ell((w, \mathbf w)) \qquad \forall (w, \mathbf w) \in \mathbb U_0,
\end{alignat}
\end{subequations}
is uniquely solvable, and 
\[
\|(u^h, \mathbf u^h)\|_{\mathbb U} \leq C_1 \|\ell\|_{\mathbb U_0^*} \leq C_2(\|F\|_{\mathbb H} + \|\Xi\|_{\mathbb M}),
\]
where $C_1$ and $C_2$ are constants independent of $h$.The three missing components of $U^h$ are defined as 
\[
\mathbf v^h:= \kappa_0 \nabla u^h + \mathbf f, \qquad \mathrm S^h := \mathcal C \varepsilon(\mathbf u^h) + \mathrm G, \qquad \mathbf r^h := L_\Omega^h \mathbf u^h + \mathbf h.
\]
With these new definitions, our variational problem becomes 
\begin{alignat}{4} \label{eq:VarProb3}
(\kappa_1 u^h, w)_{\mathbb R^d \setminus\Gamma} + (\mathbf v^h, \nabla w)_{\mathbb R^d \setminus\Gamma} + (\rho \mathbf u^h, \mathbf w)_{\Omega_-} \nonumber\\ 
 \quad + (\mathrm S^h + \mathbf e \mathbf r^h, \varepsilon(\mathbf w))_{\Omega_-} - \langle \gamma \mathbf w \cdot \bnu, \llbracket \gamma u^h \rrbracket \rangle \nonumber\\
 + \langle \gamma \mathbf u^h \cdot \bnu, \llbracket \gamma w \rrbracket \rangle &= (\kappa_1 f, w)_{\mathbb R^d \setminus\Gamma} + (\rho \mathbf g, \mathbf w)_{\Omega_-} + \langle \gamma \mathbf w \cdot \bnu, \xi \rangle \nonumber \\
 &\qquad+ \langle \psi_2, \gamma^- w \rangle + \langle \eta_1, \llbracket \gamma w \rrbracket \rangle. 
\end{alignat}
Using the definition of $\mathrm{div}_h$ and $\gamma_{\bnu,h}^t$ and rearranging some of the terms above, we arrive at
\begin{alignat}{3} \label{eq:VarRef3}
&(\kappa_1 (u^h - f), w)_{\mathbb R^d \setminus\Gamma} + (\mathbf v^h, \nabla w)_{\mathbb R^d \setminus\Gamma} \nonumber\\
&\quad + (\rho(u^h - \mathrm{div}_h (\mathrm S^h + \mathbf e \mathbf r^h) - \gamma_{\bnu,h}^t(\llbracket \gamma u^h \rrbracket + \xi) - \mathbf g), \mathbf w)_{\Omega_-} &= -\langle \gamma \mathbf u^h \cdot \bnu - \eta_1, \llbracket \gamma w \rrbracket \rangle \nonumber\\
& &  + \langle \psi_2, \gamma^- w \rangle.
\end{alignat}
Choosing $w = 0$, and testing with all $\mathbf w \in \mathbf V_h$ we see that 
\[
\mathbf u^h = \mathrm{div}_h( \mathrm S^h + \mathbf e \mathbf r^h) + \gamma_{\bnu, h}^t (\llbracket \gamma u^h \rrbracket + \xi) + \mathbf g.
\]
Substituting this definition of $\mathbf u^h$ into \eqref{eq:VarRef3} and testing with $w \in \mathcal D(\mathbb R^d \setminus\Gamma)$ we obtain
\[
(\kappa_1 (u^h - f), w)_{\mathbb R^d \setminus\Gamma} + (\mathbf v^h, \nabla w)_{\mathbb R^d \setminus\Gamma} = (\kappa_1 (u^h - f) - \nabla \cdot \mathbf v^h, w)_{\mathbb R^d \setminus\Gamma} = 0.
\]
This shows that $u^h = \kappa_1^{-1} \nabla \cdot \mathbf v^h + f$ and $\mathbf v^h \in \mathbf H(\mathrm{div}, \mathbb R^d \setminus\Gamma)$.  Making one final substitution into \eqref{eq:VarRef3}, testing with $(w, \mathbf w) \in \mathbb U_0$ and integrating by parts leads to
\begin{alignat*}{2}
0 &= \langle \gamma_\bnu^- \mathbf v^h, \gamma^- w \rangle - \langle \gamma_\bnu^+ \mathbf v^h, \gamma^+ w \rangle + \langle \gamma \mathbf u^h \cdot \bnu - \eta_1, \llbracket \gamma w \rrbracket \rangle - \langle \psi_2, \gamma^- w \rangle\\
 &= \langle \llbracket \gamma_\bnu \mathbf v^h \rrbracket - \psi_2, \gamma^- w \rangle + \langle \gamma \mathbf u^h \cdot \bnu + \gamma_\bnu^+ \mathbf v^h - \eta_2, \llbracket \gamma w \rrbracket \rangle.
\end{alignat*}
By the definition of $\mathbb U_0$, this gives us that 
\[
\llbracket \gamma_\bnu \mathbf v^h \rrbracket - \psi_2 \in X_h, \qquad \gamma \mathbf u^h \cdot \bnu + \gamma_\bnu^+ \mathbf v^h - \eta_2 \in Y_h^\circ,
\]
and we have the desired transmission problem. All of the arguments above can easily be reversed to see the equivalence of the two problems.  Arriving at the bound on $\|U^h\|_{\mathbb V}$ follows similarly as in the proof of \Cref{lemmaSolve1}.
\end{proof}

We now return to the main result of this section. 
 
\begin{proof}[Proof of \Cref{Results3}.]
The estimates on $\|u^h(t)\|_{1, \mathbb R^d \setminus\Gamma}, \|\psi^h(t)\|_{1, \Omega_-}$, and $\|\mathbf u^h(t)\|_{1, \Omega_-}$ are found similarly to those in \Cref{results1}.  To achieve the bound on $\|\varphi^h(t)\|_{1/2, \Gamma}$, we note that
\[
\|\varphi^h(t)\|_{1/2, \Gamma} = \|\llbracket \gamma u^h \rrbracket (t) \|_{1/2, \Gamma} \lesssim \|u^h(t)\|_{1, \mathbb R^d \setminus\Gamma},
\]
from which the bound follows.  Finally, for the estimate on $\|\lambda^h(t)\|_{-1/2, \Gamma}$, we recall that
\begin{alignat*}{3}
\|\lambda^h(t) \|_{-1/2, \Gamma} = \|\llbracket \partial_\bnu u^h \rrbracket(t)\|_{-1/2, \Gamma} = \|\kappa_0^{-1} \llbracket \gamma_\bnu \dot{\mathbf v}^h\rrbracket (t)\|_{-1/2, \Gamma} &\lesssim \|\dot{\mathbf v}(t)\|_{\mathrm{div}, \mathbb R^d \setminus\Gamma}\\
&\lesssim \|\dot{\mathbf v}(t)\|_{\mathbb R^d \setminus\Gamma} + \|\nabla \cdot \dot{\mathbf v}(t)\|_{\mathbb R^d \setminus\Gamma}\\
& \lesssim \|\dot{U}(t)\|_{\mathbb H} + \|\ddot{U}(t)\|_{\mathbb H}.
\end{alignat*}
Now we need only to appeal to \Cref{bigTheorem} (b) to obtain the desired result. 
\end{proof}

\section{Approximation properties} \label{EE}

The goal of this section is the study of the difference between the solution of \eqref{eq:12} (what we called a semidiscrete solution) and a discrete projection acting on the solution of \eqref{eq:SecOrdForm1}.  This study is in essence an analysis of the semidiscretization of \eqref{eq:SecOrdForm1} using \eqref{eq:16}.  We will show that the error (as defined below) is driven by an evolutionary equation of the form \eqref{eq:modelProb2}.  Let $U(t) = (u, \mathbf v, \mathbf u, \mathrm S, \mathbf r)(t)$ be the solution of \eqref{eq:FirstOrdForm} and 
\begin{alignat*}{2}
\varphi(t) &:=  \llbracket \gamma u \rrbracket(t)  = -\gamma^+ u(t):[0, \infty) \to  H^{1/2}(\Gamma),\\
\lambda(t) &:= \llbracket \partial_\bnu u \rrbracket(t) = - \partial_\bnu^+ u(t) = - \kappa_0^{-1} \gamma_\bnu^+ \dot{\mathbf v}(t):[0, \infty) \to H^{-1/2}(\Gamma). 
\end{alignat*}
Consider the finite dimensional spaces $\mathbf V_h, V_h, X_h$, and $Y_h$. In this setting, the solution to the problem \eqref{eq:12}  will be denoted $U^h(t) = (u^h, \mathbf v^h, \mathbf u^h, \mathrm S^h, \mathbf r^h)(t)$ with  $\varphi^h(t) = \llbracket \gamma u^h \rrbracket (t)$ and $\lambda^h(t) = \llbracket \partial_\bnu u^h \rrbracket(t)$. Following \cite{SaSa2016} we introduce an `elliptic projection' tailored to the coupled elliptic system inherent to our problem.  Due to the lack of Dirichlet conditions or mass terms in the elastic part of the bilinear form, we need to consider the finite dimensional space of rigid motions
\[
\mathbf M:= \{\mathbf m \in \mathbf H^1(\Omega_-) : (\mathcal C \varepsilon(\mathbf m), \varepsilon (\mathbf m))_{\Omega_-} = 0\}, 
\]
which we assume to be a subspace of $\mathbf V_h$. We then consider the operator
\[
\mathbf H^1(\Omega_-)\times H^1(\Omega_-)\times \gamma_D V_h 
\ni (\mathbf u,\psi,\mu^h) \longmapsto (\mathbf P_h\mathbf u,P_h \psi)\in \mathbf V_h\times V_h,
\]
given by the unique solution of the coercive problem:
\begin{alignat*}{4}
\gamma_D P_h \psi &= \mu^h,\\
(\mathcal C \varepsilon(\mathbf P_h \mathbf u - \mathbf u), \varepsilon(\mathbf w))_{\Omega_-} + (\mathbf e \nabla (P_h \psi - \psi), \varepsilon(\mathbf w))_{\Omega_-} &= 0 & \;\; & \forall \mathbf w \in \mathbf V_h,\\
 - (\varepsilon(\mathbf P_h \mathbf u - \mathbf u), \mathbf e \nabla \phi)_{\Omega_-} + (\kappa_\psi \nabla (P_h \psi - \psi ), \nabla \phi)_{\Omega_-} &= 0 & &\forall \phi  \in V_{h,D},\\
(\mathbf P_h \mathbf u - \mathbf u, \mathbf m)_{\Omega_-} &= 0 &  &\forall \mathbf m \in \mathbf M.
\end{alignat*}
The above problem defining $(\mathbf P_h \mathbf u, P_h \psi) \in \mathbf V_h \times V_h$ is uniquely solvable since a straightforward application of Poincar\'{e}'s and Korn's inequalities guarantees coercivity when $\mathbf M\subset \mathbf V_h$. Using a uniformly bounded lifting of $\gamma_D$ (see \cite{DoSa2003} and \cite[Section 5]{Sayas2007} to see why this is actually a necessity), we can then prove that
\begin{equation}\label{eq:ellproj}
\|\mathbf u - \mathbf P_h \mathbf u\|_{1, \Omega_-} + \|\psi - P_h \psi\|_{1, \Omega_-} \lesssim   \|\mathbf u - \mathbf I_h\mathbf u\|_{1, \Omega_-} + \|\psi - I_h \psi\|_{1, \Omega_-} 
 + \|\gamma_D \psi - \mu^h \|_{1/2,\Gamma_D},
\end{equation}
where $\mathbf I_h : \mathbf H^1(\Omega_-) \to \mathbf V_h$ and $I_h: H^1(\Omega_-) \to V_h$ are the  best approximation operators on $\mathbf V_h$ and $V_h$ respectively. The hidden constant in the above inequality depends only on the physical parameters (through boundedness and coercivity inequalities) and on the uniform bound for the best possible lifting of the discrete trace.
Notice that the notation of the projection is slightly misleading. In fact, $\mathbf P_h \mathbf u$ and $P_h\psi$ both depend not only on $\mathbf u$ and $\psi$, but also on the given discrete trace $\mu^h$. We will keep the given notation for the sake of simplicity.

We also define the operators $\Pi_h^Y: H^{1/2}(\Gamma) \to Y_h$ and $\Pi_h^X: H^{-1/2}(\Gamma) \to X_h$ as the best approximation operators into $Y_h$ and $X_h$ respectively, and the error quantities as follows
\begin{alignat*}{4}
e_u^h &:= u - u^h, &\qquad   \mathbf e_{\mathbf v}^h &:= \mathbf v - \mathbf v^h,\\
\mathbf e_{\mathbf u}^h &:= \mathbf P_h \mathbf u - \mathbf u^h & \mathrm e_{\mathrm S}^h &:= \partial^{-1} \mathcal C \varepsilon(\mathbf P_h \mathbf u) - \mathrm S^h \qquad \mathbf e_{\mathbf r}^h:= \partial^{-1} L_\Omega^h(\mathbf P_h \mathbf u) - \mathbf r^h.
\end{alignat*}
We denote $E^h:= (e_u^h, \mathbf e_{\mathbf v}^h, \mathbf e_{\mathbf u}^h, \mathrm e_{\mathrm S}^h, \mathbf e_{\mathbf r}^h)$ for brevity. Therefore 
\[
\llbracket \gamma e_u^h \rrbracket = \varphi - \varphi^h, \qquad \llbracket \gamma_\bnu \mathbf e_{\mathbf v}^h \rrbracket = \kappa_0 \partial^{-1} (\lambda - \lambda^h).
\]
The problem of looking for $E^h \in \mathcal C^1([0, \infty), \mathbb H) \cap \mathcal C([0, \infty), \mathbb V)$ which for each $t \ge 0$ satisfies 
\begin{subequations} \label{eq:error1}
\begin{alignat}{3}
\dot{E}^h(t) &= A_\star E^h(t) + (0, \; \mathbf 0, \; (\mathbf P_h \dot{\mathbf u} - \dot{\mathbf u})(t),\;  0,\;  L_D^h(\mu - \mu^h)(t)) ,\\
BE^h(t) &= ((\varphi - \Pi_h^Y \varphi)(t)\big|_{Y_h^\circ}, \; \kappa_0 \partial^{-1}(\lambda - \Pi_h^X \lambda)(t)\big|_{X_h^\circ}, \; (\gamma (\mathbf P_h \mathbf u - \mathbf u) \cdot \bnu(t))\big|_{Y_h}, \; 0),
\end{alignat}
with vanishing initial condition
\begin{equation}
E^h(0) = 0,
\end{equation}
\end{subequations}
can be covered by the abstract framework of \Cref{framework} by letting 
\begin{alignat*}{2}
F &= (0, \; \mathbf 0, \; \mathbf P_h \dot{\mathbf u} - \dot{\mathbf u},\; 0,\; L_D^h(\mu - \mu^h)),\\
G &\equiv 0,\\
\Xi &= (0,\; ((\varphi - \Pi_h^Y \varphi)\big|_{Y_h^\circ}, \; \kappa_0 \partial^{-1}(\lambda - \Pi_h^X \lambda)\big|_{X_h^\circ}, \; (\gamma (\mathbf P_h \mathbf u - \mathbf u) \cdot \bnu)\big|_{Y_h} \; 0)).
\end{alignat*}
Before stating our error estimates, we introduce the following useful lemma.

\begin{lemma} \label{Strang}
For $(\mathbf u, \mathbf u^h, \mu, \mu^h, \eta)  \in H^1(\Omega_-) \times \mathbf V_h \times H^{1/2}(\Gamma_D) \times \gamma_D V_h \times H^{-1/2}(\Gamma_N)$, let $\psi \in H^1(\Omega_-)$ and $\psi^h \in V_h$ be given by
\begin{subequations} \label{eq:psiCont}
\begin{alignat}{3}
\gamma_D \psi &= \mu,\\
(\kappa_\psi \nabla \psi, \nabla \phi)_{\Omega_-} &= (\varepsilon(\mathbf u), \mathbf e \nabla \phi)_{\Omega_-} - \langle \eta, \gamma \phi \rangle_{\Gamma_N} \qquad \forall \phi \in H_D^1(\Omega_-),
\end{alignat}
\end{subequations}
and 
\begin{subequations} \label{eq:psiSD}
\begin{alignat}{3}
\gamma_D \psi^h &= \mu^h,\\
(\kappa_\psi \nabla \psi^h, \nabla \phi^h)_{\Omega_-} &= (\varepsilon(\mathbf u^h), \mathbf e \nabla \phi^h)_{\Omega_-} - \langle \eta, \gamma \phi^h \rangle_{\Gamma_N} \qquad \forall \phi^h \in V_{h,D},
\end{alignat}
\end{subequations}
respectively.  Then 
\[
\|\psi - \psi^h\|_{1, \Omega_-} \lesssim \|\psi - I_h \psi \|_{1, \Omega_-} + \|\mu - \mu^h\|_{1/2, \Gamma_D} + \|\varepsilon(\mathbf u) - \varepsilon(\mathbf u^h)\|_{\Omega_-}.
\]
\end{lemma}
\begin{proof}
Recalling the definitions of the various $L$ and $L^h$ operators above, we can write $\nabla \psi = L_\Omega \mathbf u + L_N \eta + L_D \mu$ and $\nabla \psi^h = L_\Omega^h \mathbf u^h + L_N^h \eta + L_D^h \mu^h$.  We introduce the intermediate quantity $\widetilde{\psi}^h \in V_h$ such that $\nabla \widetilde{\psi}^h = L_\Omega^h \mathbf u + L_N^h \eta + L_D^h \mu^h$.  Since $\gamma_D$ admits a uniformly bounded right-inverse, we have 
\[
\|\psi - \widetilde{\psi}^h\|_{1, \Omega_-} \lesssim  \|\psi - I_h \psi \|_{1, \Omega_-} + \|\mu - \mu^h\|_{1/2, \Gamma_D}.
\]
On the other hand, Strang's First Lemma (cf \cite[Lemma III.1.1]{Braess2007}) gives us 
\[
\|\widetilde{\psi}^h - \psi^h\|_{1, \Omega} \lesssim \sup_{\phi^h \in V_{h,D}} \frac{|(\varepsilon(\mathbf u) - \varepsilon(\mathbf u^h), \mathbf e \nabla \phi^h)|}{\|\phi_h\|_{1, \Omega_-}} \lesssim \|\varepsilon(\mathbf u) - \varepsilon(\mathbf u^h)\|_{\Omega_-}.
\]
Putting these two bounds together, the result follows.
\end{proof}

Now using \Cref{bigTheorem}, we obtain the following semidiscrete error estimates.

\begin{theorem} \label{errEst1}
For $(\mathbf u,\psi, \phi, \lambda, \mu, \mu^h) \in W^{\ell + 1}(\mathbf H^1(\Omega_-)) \times W^{\ell +1} (H^1(\Omega_-)) \times W^{\ell + 1}(H^{1/2}(\Gamma))\times W^\ell(H^{-1/2}(\Gamma))\times W^{\ell+1}(H^{1/2}(\Gamma_D)) \times W^{\ell+1}(\gamma_D V_h)$, with $\ell = 1$, the problem \eqref{eq:error1} is uniquely solvable and for all $t \geq 0$, its solution satisfies
\begin{alignat*}{6}
\|(u - u^h)(t)\|_{1, \mathbb R^d \setminus\Gamma} &+ \|(\mathbf u - \mathbf u^h)(t)\|_{1, \Omega_-} \\
\|\psi - \psi^h\|_{1,\Omega_-} & + \|(\varphi - \varphi^h)(t)\|_{1/2, \Gamma} &&\lesssim H_2(\varphi - \Pi_h^Y \varphi, t | H^{1/2}(\Gamma))\\
& && \quad + H_2(\partial^{-1}(\lambda - \Pi_h^X \lambda), t |H^{-1/2}(\Gamma))\\
& &&\quad  + H_2(\mathbf u -\mathbf I_h \mathbf u  , t | \mathbf H^1(\Omega_-))\\
& && \quad+ H_2(\psi - I_h \psi, t|H^1(\Omega_-)),\\
& && \quad+ H_2(\mu - \mu^h, t|H^{1/2}(\Gamma_D)),
\end{alignat*}
Furthermore, if $\ell = 2$, then we can also bound
\begin{alignat*}{3}
\|(\lambda - \lambda^h)(t) \|_{-1/2, \Gamma} &\lesssim H_3(\varphi - \Pi_h^Y \varphi, t | H^{1/2}(\Gamma))+ H_2(\lambda - \Pi_h^X \lambda, t | H^{-1/2}(\Gamma))\\
& \quad + H_3(\mathbf u-\mathbf I_h \mathbf u , t | \mathbf H^1(\Omega_-)) + H_3(\psi- I_h \psi , t | H^1(\Omega_-))\\
& \quad + H_3(\mu - \mu^h, t|H^{1/2}(\Gamma_D)).
\end{alignat*}
\end{theorem} 

\begin{proof}
First, we note that 
\begin{alignat*}{3}
\|(u - u^h)(t)\|_{1, \mathbb R^d \setminus\Gamma} + \|(\mathbf u - \mathbf u^h)(t)\|_{1, \Omega_-} \leq & \|(u - u^h)(t)\|_{1, \mathbb R^d \setminus\Gamma} + \|(\mathbf u - \mathbf P_h \mathbf u)(t)\|_{1, \Omega_-}\\
& + \|(\mathbf P_h \mathbf u  - \mathbf u^h)(t)\|_{1, \Omega_-}\\
\lesssim & \|E^h(t)\|_{\mathbb H} + \|\dot{E}^h(t)\|_{\mathbb H}\\
& + H_1(\mathbf u - \mathbf P_h \mathbf u , t| H^1(\Omega_-)).
\end{alignat*}
Noting that $\|(\varphi - \varphi^h)(t)\|_{1/2, \Gamma} \lesssim \|(u - u^h)(t)\|_{1, \mathbb R^d \setminus\Gamma}$, and appealing to \Cref{bigTheorem}, we arrive at our first estimate after using \eqref{eq:ellproj} to change the joint elliptic projection by the best approximation operators. \Cref{Strang} gives us the bound 
\begin{alignat*}{3}
\|(\psi - \psi^h)(t)\|_{1, \Omega_-} &\lesssim  \|(\psi-I_h \psi)(t)\|_{1, \Omega_-} + \|(\mu - \mu^h)(t)\|_{1/2, \Gamma_D} + \|\varepsilon(\mathbf u) - \varepsilon(\mathbf u^h)\|_{\Omega_-}\\
&\lesssim \|(\psi-I_h \psi)(t)\|_{1, \Omega_-} + \|(\mu - \mu^h)(t)\|_{1/2, \Gamma_D}  + \|\mathbf u - \mathbf u^h\|_{1, \Omega_-},
\end{alignat*}
from which the estimate of the statement follows. 
The estimate on $\|(\lambda - \lambda^h)(t)\|_{-1/2, \Gamma}$ is found similarly to the bound on $\|\lambda^h(t)\|_{-1/2, \Gamma}$ in \Cref{Results3} by appealing to \Cref{bigTheorem}(b). 
\end{proof}

\section{Some additional issues}  \label{AddIss}

\subsection{Purely acoustic-elastic coupling}
 We now proceed to analyze the wave-structure interaction problem where $\Omega_-$ does not possess any piezoelectric properties.  The following theorem shows that this problem is a specific case of problem  \eqref{eq:SecOrdForm1}.
\begin{theorem} \label{equiv1}
Suppose that $\mathbf e \equiv \mathbf 0, \; \Gamma_N = \emptyset,$ (so that $\Gamma = \Gamma_D$), and $\mu(t) = 0$ for each $t \geq 0$. Then the problem \eqref{eq:SecOrdForm1} is equivalent to the problem of looking for 
\begin{alignat*}{3}
u:[0, \infty) &\to \{w \in H^1(\Omega_+): \kappa_0 \nabla w \in \mathbf H(\mathrm{div}, \Omega_+)\}, \\
\mathbf u:[0, \infty) &\to \{\mathbf w \in \mathbf H^1(\Omega_-): \mathcal C \varepsilon (\mathbf w) \in \mathrm H_\mathrm{sym}(\mathrm{div}, \Omega_-)\},
\end{alignat*}
which for all $t \ge 0$ satisfy
\begin{subequations} \label{eq:SecOrdForm2}
\begin{alignat}{6}
\kappa_1 \ddot{u}(t)&= \nabla \cdot \left(\kappa_0 \nabla u \right)(t), \label{eq:SecOrdForm2a}\\
\rho\, \ddot{\mathbf{u}}(t) &= \mathrm{div} \; \mathcal C\varepsilon(\mathbf{u})(t), \label{eq:SecOrdForm2b}\\ 
\gamma_{\bnu}^+(\kappa_0 \nabla u)(t) + \beta_1(t) + \gamma^- \dot{\mathbf{u}}(t) \cdot \boldsymbol{\nu} & = 0, \label{eq:SecOrdForm2c}\\
(\gamma^+\dot{u}(t) + \dot{\beta}_0(t))\boldsymbol{\nu} + \bgamma_{\bnu}^- \mathcal C\varepsilon(\mathbf{u})(t) & =
	\mathbf{0},\label{eq:SecOrdForm2d}\\
u(0) = 0, \quad \dot{u}(0) = 0, \quad \mathbf{u}(0) &= \mathbf{0}, \quad \dot{\mathbf{u}}(0) = \mathbf{0}, \label{eq:SecOrdForm2e}
\end{alignat}
\end{subequations}
with the added unknown $\psi: [0, \infty) \to H_0^1(\Omega_-)$ such that $\psi(t) = 0$ for each $t \geq 0$. 
\end{theorem}
\begin{proof}
First we notice that \eqref{eq:SecOrdForm2a}, \eqref{eq:SecOrdForm2c}, and \eqref{eq:SecOrdForm2e} are identical to \eqref{eq:SecOrdForm1a},  \eqref{eq:SecOrdForm1d}, and \eqref{eq:SecOrdForm1h} respectively. Letting $\mathbf e \equiv \mathbf 0$ immediately gives us that \eqref{eq:SecOrdForm2b} and \eqref{eq:SecOrdForm2d} are the same as \eqref{eq:SecOrdForm1b} and \eqref{eq:SecOrdForm1e}.  Since $\Gamma_N = \emptyset$, we have that \eqref{eq:SecOrdForm1f} is not a valid equation in this setting and $\mu(t) = 0$ for all $t \geq 0$ implies that we are looking for $\psi \in H_0^1(\Omega_-)$ such that 
\[
\nabla \cdot (\kappa_\psi \nabla \psi(t)) = 0 \quad \mbox{in $L^2(\Omega_-)$} \quad \forall t \geq 0.
\]
Since $\kappa_\psi$ is coercive, we have that $\psi(t) = 0$ for each $t \geq 0$. Thus the two problems are equivalent.   
\end{proof}

Similarly, we can make slight modifications to \eqref{eq:12} and be able to use all of the same analysis to obtain results on a semidiscrete version \eqref{eq:SecOrdForm2}.
\begin{theorem} \label{equiv2}
Under the same hypotheses as \Cref{equiv1}, the problem \eqref{eq:12} is equivalent to the problem of looking for 
\[
(u^h, \mathbf u^h)(t):[0, \infty) \to H^1(\mathbb R^d \setminus\Gamma)  \times \mathbf V_h,
\]
which for every $t \geq 0$ satisfy 
\begin{subequations} \label{eq:AcElasCoup2} 
\begin{alignat}{6}
\kappa_1 \ddot{u}^h(t) &= \kappa_0 \Delta u^h (t), \\
(\rho \ddot{\mathbf u}^h(t),\mathbf w)_{\Omega_-}\!\!
	+ (\mathcal C\varepsilon(\mathbf u^h)(t),\varepsilon(\mathbf w))_{\Omega_-}\!\!
	& = \langle \llbracket\gamma \dot u^h\rrbracket(t)\!-\!\dot\beta_0(t), \gamma \mathbf w\rangle
	 \,\,\, \forall \mathbf w\in \mathbf V_h,
\end{alignat}
as well as 
\begin{alignat}{7}
(\llbracket \gamma u^h \rrbracket (t),\llbracket \partial_\bnu u^h \rrbracket(t)) &\in Y_h \times X_h, \\
(\gamma \dot{\mathbf{u}}^h(t) \cdot \bnu + \kappa_0 \partial_{\bnu}^+ u^h(t) + \beta_1(t),\gamma^- u^h(t)) &\in Y_h^\circ \times X_h^\circ,
\end{alignat}
with vanishing initial conditions
\begin{equation}
u^h(0) = 0, \quad \dot{u}^h(0) = 0, \quad \mathbf{u}^h(0) = \mathbf{0}, \quad \dot{\mathbf{u}}^h(0) = \mathbf{0}.
\end{equation}
\end{subequations}
with the added unknown $\psi^h: [0, \infty) \to V_h$ such that $\psi^h(t) = 0$ for each $t \geq 0$. 
\end{theorem}  
\begin{proof}
The proof is similar to that of \Cref{equiv1}.
\end{proof}

\subsection{A damped elastic wave equation}

The elastic wave equation \eqref{eq:SecOrdForm1b} can be substituted by a damped version of the equation 
\[
\rho \ddot{\mathbf u}(t) + \omega \dot{\mathbf u}(t) = \mathrm{div}\, (\mathcal C \varepsilon(\mathbf u)(t) + \mathbf e \nabla \psi(t)),
\]
where  $\omega \in L^\infty(\Omega_-)$ is non-negative.  In this case the third component of $A_\star$ in \eqref{eq:A} has to be modified to be 
\[
\rho^{-1} \mathrm{div} \, (\mathrm S + \mathbf e \mathbf r) - \rho^{-1} \omega \mathbf u.
\]
A similar modification must be made to the semidiscrete version of $A_\star$ defined in \Cref{SDWP}.  The effect of this change is that in \Cref{lemmaDiss1,LemmaFEM1} rather than equality, we have $(AU, U) \leq 0$ for each $U \in D(A)$ and the ``flipping-sign'' operator $T$ of \Cref{lemmaT1,LemmaT3}, which allows for the reversal of time,  no longer exists.  All of the stability and error estimate results of this paper are still valid, keeping in mind that hidden constants will now also include constants related to the norm of $\omega$, but now $A$ generates a contraction $C_0$-semigroup of operators in $\mathbb H$ instead of a $C_0$-group of isometries. 

\subsection{Full integral formulations}
In this section we formulate an alternate version of the semidiscrete wave-solid interaction problem without piezoelectricity \eqref{eq:AcElasCoup2} with constant coefficients in the elastic law.  The goal is to state the problem entirely on $\Gamma$ using non-local integral representations for both the acoustic and elastic unknowns, with the idea that we post-process to find the values of the solution away from the boundary.  This requires that both unknown quantities now take values in $\mathbb R^d \setminus\Gamma$.  In addition to $Y_h$, we define $\mathbf Y_h \subset \mathbf H^{1/2}(\Gamma)$ to be closed and $\mathbf Y_h^\circ \subset \mathbf H^{-1/2}(\Gamma)$ is defined similarly to $Y_h^\circ$.  As in \Cref{Semidiscrete} we require that $\kappa_0$ and $\kappa_1$ be positive constants, and additionally that $\rho$ and the tensor $\mathcal C$ have constant coefficients.   

We can state the second order formulation of the problem as looking for 
\[
(u^h, \mathbf u^h): [0, \infty) \to H^1(\mathbb R \setminus\Gamma)\times \mathbf H^1(\mathbb R \setminus\Gamma),
\]
which for every $t \geq 0$ satisfy
\begin{subequations} \label{eq:SecOrdForm5}
\begin{alignat}{7}
\kappa_1 \ddot{u}^h(t) &= \kappa_0 \Delta u^h (t), \label{eq:SecOrdForm5a}\\
\rho \ddot{\mathbf{u}}^h(t) &= \mathrm{div} \; \mathcal C \varepsilon(\mathbf{u}^h)(t), \label{eq:SecOrdForm5b}\\
\llbracket \gamma \dot{\mathbf{u}}^h \rrbracket (t) \cdot \bnu  - \kappa_0 \llbracket \partial_{\bnu} u^h \rrbracket (t) + \beta_1(t) &= 0, \label{eq:SecOrdForm5c} \\
\llbracket \bgamma_{\bnu} \mathcal C \varepsilon (\mathbf{u}^h) \rrbracket (t) - \llbracket \gamma \dot{u}^h\rrbracket (t) \bnu + \dot{\beta}_0(t) \bnu &= \mathbf{0}, \label{eq:SecOrdForm5d}
\end{alignat}
as well as
\begin{alignat}{5}
(\llbracket \gamma u^h \rrbracket (t), \llbracket \gamma \mathbf{u}^h \rrbracket (t)) &\in Y_h \times \mathbf Y_h, \label{eq:SecOrdForm5e}\\
(\gamma^+ \dot{\mathbf{u}}^h(t) \cdot \bnu + \kappa_0 \partial_{\bnu}^- u^h(t), \; \bgamma_{\bnu}^+ \mathcal C\varepsilon(\mathbf{u}^h)(t) + \gamma^- \dot{u}^h (t) \bnu) &\in Y_h^\circ \times \mathbf Y_h^\circ, \label{eq:SecOrdForm5f}
\end{alignat}
with vanishing initial conditions
\begin{equation}
u^h(0) = 0, \quad  \dot{u}^h(0) = 0, \quad  \mathbf{u}^h(0) = \mathbf{0}, \quad  \dot{\mathbf{u}}^h(0) = \mathbf{0}.
\end{equation}
\end{subequations}
We arrive at an equivalent problem if we replace \eqref{eq:SecOrdForm5f} with 
\[
(\gamma^- \dot{\mathbf{u}}^h(t) \cdot \bnu + \kappa_0 \partial_{\bnu}^+ u^h(t) + \beta_1(t), \; \bgamma_{\bnu}^- \mathcal C\varepsilon(\mathbf{u}^h)(t) + \gamma^+ \dot{u}^h(t)\bnu + \dot{\beta}_0(t)\bnu) \in Y_h^\circ \times \mathbf Y_h^\circ.
\]
Thinking in terms of retarded potentials and integral equations, we define the densities $(\phi^h, \boldsymbol{\phi}^h) := (\llbracket \gamma u^h \rrbracket, \llbracket \gamma \mathbf u^h \rrbracket) \in Y_h \times \mathbf Y_h$ to arrive at the following result. 
\begin{theorem} \label{BEMStab}
For $(\beta_0, \beta_1) \in W^{2}(H^{1/2}(\Gamma)) \times W^1(H^{-1/2}(\Gamma))$, we have that the solution to \eqref{eq:SecOrdForm5} is unique and satisfies
\begin{alignat*}{4}
&\|u^h(t)\|_{1, \mathbb R^d \setminus\Gamma} + \|\mathbf u^h (t)\|_{1, \mathbb R^d \setminus\Gamma}\\
&\quad  + \|\phi^h(t)\|_{1/2, \Gamma} + \|\boldsymbol{\phi}^h(t)\|_{1/2, \Gamma} &&\lesssim  H_2((\beta_0, \partial^{-1} \beta_1), t |H^{1/2}(\Gamma) \times H^{-1/2}(\Gamma)).
\end{alignat*}
\end{theorem}
The process of proving this result uses the same theoretical framework discussed in the analysis of the previous problems and the argument is similar.  In the interest of space and to avoid too much repetition, we forgo this process and move on to the results concerning approximation errors.  

If we choose $Y_h \times \mathbf Y_h = H^{1/2}(\Gamma) \times \mathbf H^{1/2}(\Gamma)$, then we will have $Y_h^\circ \times \mathbf Y_h^\circ = \{0\} \times \{ \mathbf 0 \}$.  The solution $(u, \mathbf u)$ to the above problem with this choice of spaces is continuous with the corresponding densities $(\phi, \boldsymbol{\phi}):=(-\gamma^+ u, \gamma^- \mathbf u) \in H^{1/2}(\Gamma) \times \mathbf H^{1/2}(\Gamma)$.  Once again considering the semidiscrete solution $(u^h, \mathbf u^h)$ from the above result, and defining the best approximation operators $\Pi_h:  H^{1/2}(\Gamma) \to Y_h$ and $\boldsymbol{\Pi}_h: \mathbf H^{1/2}(\Gamma) \to \mathbf Y_h$, we obtain the following error estimates.
\begin{theorem} \label{BEMErr}
Consider $(\phi, \boldsymbol{\phi}) \in W^2(\mathbb H^{1/2})$, and let $(u^h, \mathbf u^h)$ be the semidiscretizations of the solutions $(u, \mathbf u)$ of the system \eqref{eq:SecOrdForm5}. The following estimate holds
\begin{alignat*}{4}
&\|(u - u^h)(t)\|_{1, \mathbb R^d \setminus\Gamma} + \| (\mathbf u - \mathbf u^h)(t) \|_{1, \mathbb R^d \setminus\Gamma}\\
&\quad + \|(\phi - \phi^h)(t)\|_{1/2,\Gamma} + \|(\boldsymbol{\phi} - \boldsymbol{\phi}^h)(t)\|_{1/2,\Gamma} &&\lesssim H_2(\phi - \Pi_h \phi, t | H^{1/2}(\Gamma))\\
& && \quad + H_2(\boldsymbol{\phi} - \boldsymbol{\Pi}_h \boldsymbol{\phi}, t|\mathbf H^{1/2}(\Gamma)).
\end{alignat*}
\end{theorem}

\subsection{Comparison with existing results}
The problems explored in this work have been analyzed recently in \cite{HsSaSa2016} and \cite{SaSa2016}.  The novelty of this paper is the method with which the problems are analyzed as we have tried to emphasize throughout. Here we briefly stress some of the advantages that the new analysis provides. 

First of all, the path followed in the aforementioned references required using the Laplace transform, analyzing the problem in the Laplace domain and inverting back into the time domain to obtain the final estimates. Since the Laplace transform is not an isometry, the estimates obtained in this fashion are not sharp. This leads to a loss of regularity in the estimates and increased smoothness requirements in the problem data. Often times these requirements are unnecessarily strong and are only due to the proof technique, as can be seen by the improved results in the present communication.

In fact, we can see that the estimates we present here require decreased time differentiability from the problem data. Comparing the results in \Cref{Results3} with Corollary 5.1 of \cite{SaSa2016}, our current results require one less time derivative of $\beta_0, \beta_1$ and $\eta$ and two less time derivatives of $\mu^h$.  The error estimates in \Cref{errEst1}, when compared with similar results from Corollary 5.2 from the same paper, require two less time derivatives for $\mathbf u$ and $\psi$, one less time derivative for $\lambda$ and $\phi$. On the other hand, the corresponding bound in \cite{SaSa2016} does not depend on $\mu$ and $\mu^h$.

By way of Theorems \ref{equiv1} and \ref{equiv2} we can compare Theorems \ref{Results3} and \ref{errEst1} with Corollaries 4.1 and 4.2 of \cite{HsSaSa2016}.  We note that there is a typo in the statement of Corollary 4.2 in that the quantity $\mathbf u$ should be in the space $W_+^4(\mathbf H^1(\Omega_-)) \cap W_+^5(\mathbf L^2(\Omega_-)).$  With this correction, we see that the new results require one less time derivative for all quantities involved. Finally, we can compare Theorems \ref{BEMStab} and \ref{BEMErr} with Corollaries 3.3 and 3.4 of \cite{HsSaSa2016} and we see that the new results require two orders of differentiability less for all quantities.

Moreover, the bounds proven following the Laplace domain technique contain terms which depend on the time variable, while all of the results obtained with the technique introduced in the present paper are independent of $t$. Overall, it is clear that the analysis presented in the current work provides a noticeable improvement in regularity requirements, and sharper error constants thus allowing for more general problem data.
\section{Numerical Experiments}\label{NumExp}
We now present some numerical examples of the kind of problems that can be tackled using a fully-discrete version of the integro-differential formulation \eqref{eq:16}. The aim of the section is not to discuss a concrete choice of discretization scheme and its numerical properties, but rather to illustrate the applicability of the formulation studied. We present three examples with simplified plane geometries that present interesting physical situations that are accurately captured by the model. 

For the experiments chosen, the spacial discretization strategy uses finite elements for the piezoelectric unknowns and Galerkin boundary elements for the acoustic wavefield, while the time discretization was done using second order backwards differentiation (BDF2) time stepping for the finite element evolution and BDF2-based Convolution Quadrature for the unknown acoustic densities. The analysis of the stability and convergence properties of such a coupling are not in the scope of the present work and have been done in \cite{SaSa2016}, where details of the implementation are also discussed.

In all the test problems, we consider that the speed of sound in the fluid is $c=1$ and that the Lam\'e parameters of the solid are $\lambda=2$ and $\mu=3$. To express the entries of the elastic stiffness tensor $\mathcal C$, the piezoelectric tensor $\mathbf e$, and the dielectric tensor $\kappa_\psi$ we will make use of Voigt's notation and identify the symmetric pairs of indices by 
\[
(1,1) \leftrightarrow 1 \qquad (2,2) \leftrightarrow 2
\qquad (1,2) \leftrightarrow 3.
\]
This formally reduces the four-index tensor $\mathcal C$ into a $3\times 3$ symmetric matrix, the three-index piezoelectric tensor into a $2\times 3$ matrix and the dielectric tensor $\kappa_\psi$ into a $3\times 1$ vector. For these experiments, the physical coefficients were chosen to be
\begin{equation}\label{eq:7.2}
\mathcal C = \left(\begin{array}{ccc} 2.118 & 0.6 & 0  \\ 0.6 & 2.118 & 0 \\ 0 & 0 & 0.9 \end{array}\right)\, , \quad \mathbf e = \left(\begin{array}{ccc} 1 & 5 & 5 \\ 5 & 1 & 5 \end{array}\right)\, , \quad  \kappa_\psi =\left(\begin{array}{c} 4 \\ 4 \\ 1 \end{array}\right).
\end{equation}
%

\paragraph{A piezoelectric pentagon.} 
As a first test problem we consider the acoustic scattering by an inhomogeneous piezoelectric obstacle. The sinusoidal plane pulse
\begin{equation}\label{eq:7.1}
v^{inc}= 3\chi_{[0,0.3]}(\tau)\sin{(88\tau)},\quad \tau:= t-\mathbf r\cdot\mathbf d,\quad \mathbf r:=(x,y),\quad \mathbf d := (1,5)/\sqrt{26},
\end{equation}
impinges upon the pentagonal scatterer depicted in \Cref{fig:7.1} (a) with a mass distribution
\[
\rho_\Sigma = 5 + 25e^{-100r^2} \qquad r := |\mathbf r|.
\] 
All the edges of the obstacle were considered as Dirichlet boundary, where the smooth grounding potential
\begin{equation}\label{eq:7.3}
\psi(\mathbf x, t) = 10\mathcal H(t),
\end{equation}
was imposed as a boundary condition. The function $\mathcal H(t)$ is the following polynomial approximation to Heaviside's step function
\begin{equation}\label{eq:7.4}
\mathcal H(t):= t^5(1-5(t-1)\!+15(t-1)^2\!\!-35(t-1)^3\!\!+70(t-1)^4\!\!-126(t-1)^5)\chi_{[0,1]}(t)\!+\!\chi_{[1,\infty)}(t).
\end{equation}
\begin{figure}\center{
	\begin{subfigure}[b]{0.27\textwidth}
		\includegraphics[width=\linewidth]{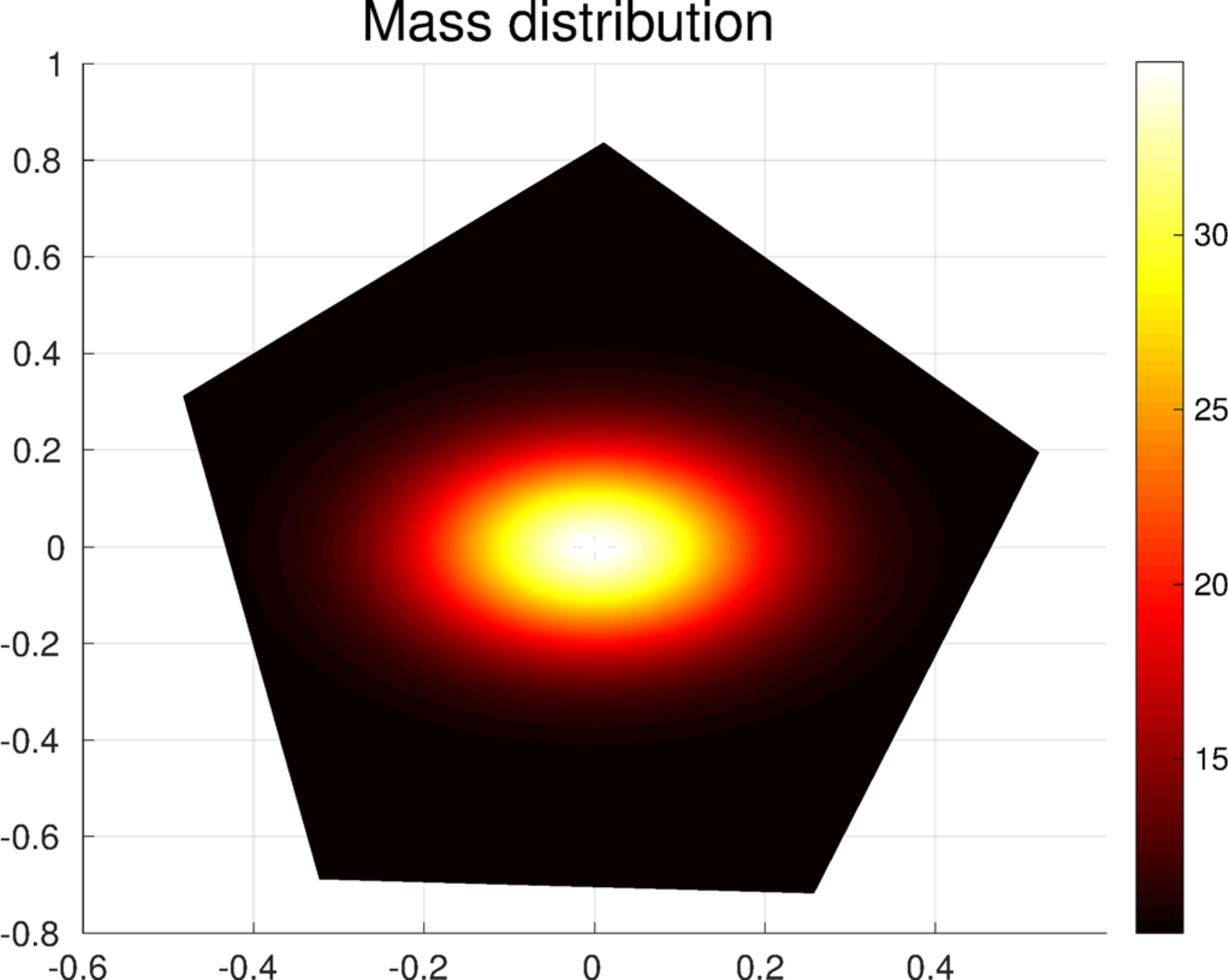}
		\caption{} \label{fig:7.1a}
	\end{subfigure}
	\qquad
	\begin{subfigure}[b]{0.27\textwidth}
		\includegraphics[width=\linewidth]{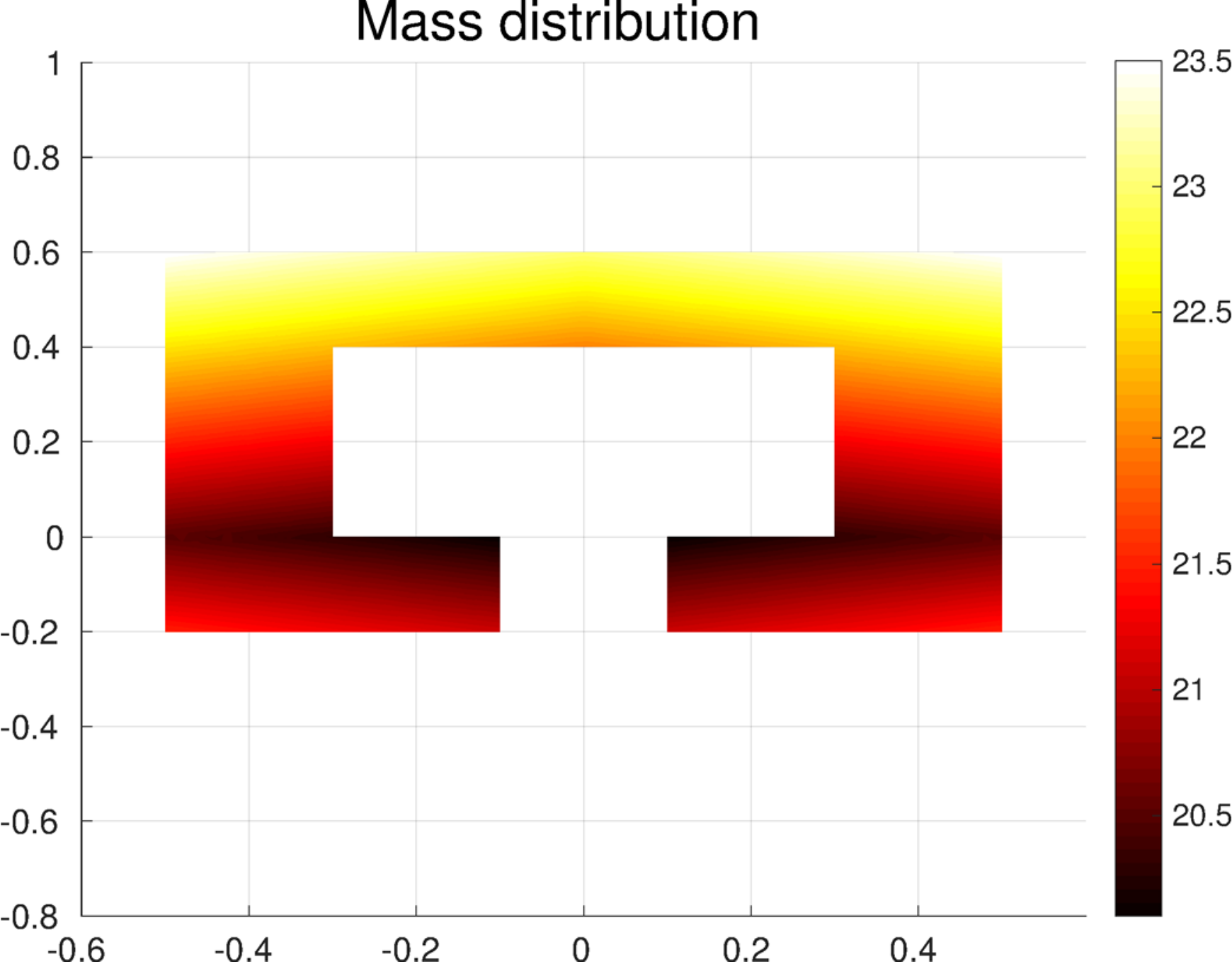}
		\caption{} \label{fig:7.1b}
	\end{subfigure}
	\qquad
	\begin{subfigure}[b]{0.27\textwidth}
		\includegraphics[width=\linewidth]{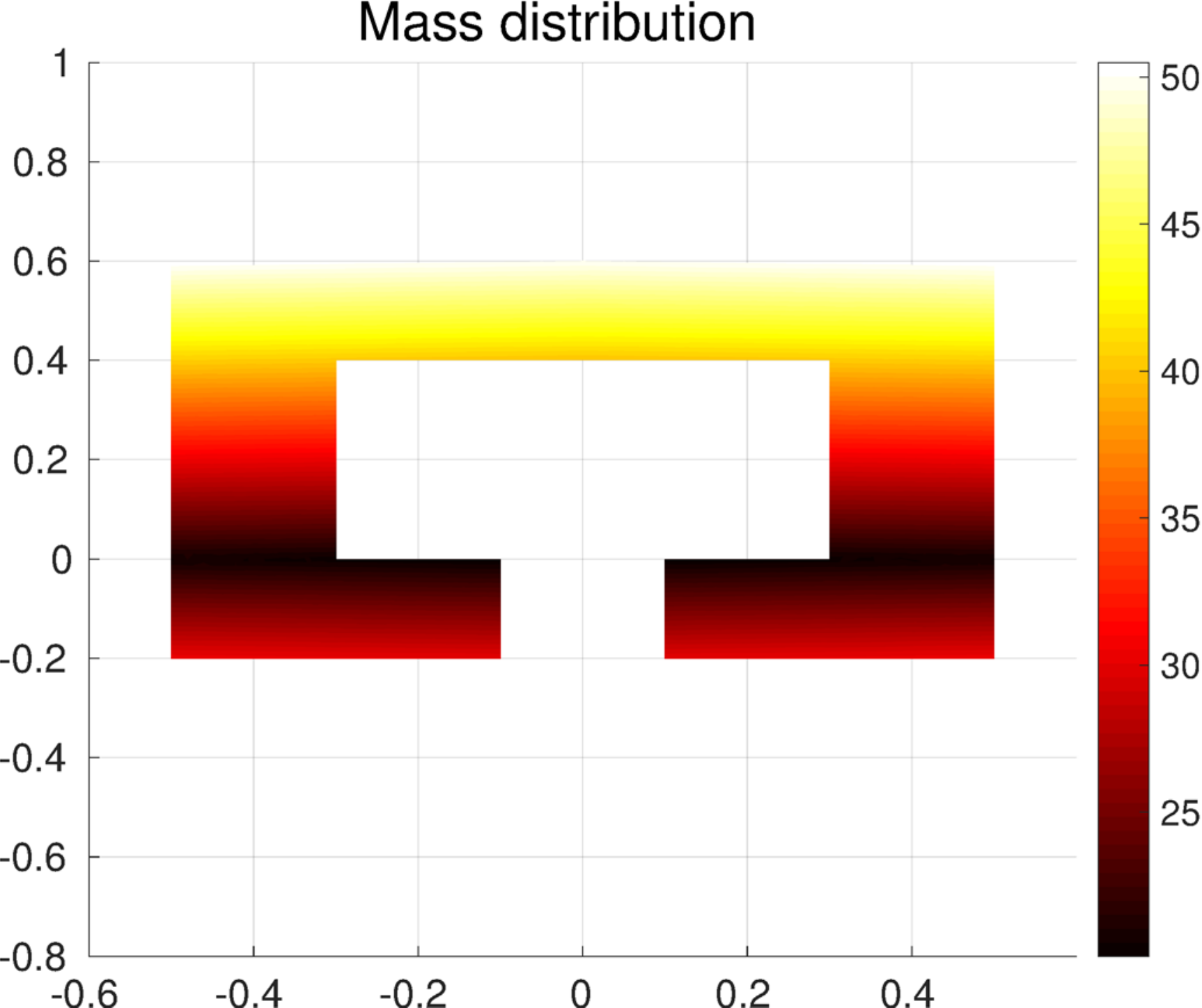}
		\caption{} \label{fig:7.1c}
	\end{subfigure}
}
\caption{{\footnotesize Geometries and mass densities for the scatterers used. For the pentagonal obstacle in (a) the density follows a Gaussian distribution peaking at its center which slows down waves propagating closer to its core. For the trapping geometry (b) the mass changes linearly along the $x$ and $y$ coordinates with a steeper slope in the $y$ direction. The distribution (c) was used for the transition to time-harmonic regime and has a similar behavior as the previous one but with a steeper slope in the $y$ direction.}}\label{fig:7.1}
\end{figure}
The spacial discretization was done using $\mathcal P_3/\mathcal P_2$ continuous/discontinuous Galerkin boundary elements for the acoustic densities and $\mathcal P_3$ Lagrangian finite elements for the elastic and electric unknowns. The simulation was run until a final time $T=6$ using 2000 equispaced points in time for a step size of $\kappa=0.003$ and a MATLAB-generated triangulation with 9024 elements and a mesh parameter of $h=0.014$. The boundary element grid consisted of 232 panels  aligned with the boundary edges of the triangulation. 

Figure \ref{fig:7.2} shows the time dynamics of the $\mathbf L^2(\Omega_-)$ and $\mathbf H^1(\Omega_-)$ norms of the elastic displacement, the $L^2(\Omega_-)$ norm of the electric potential, the $\mathbf L^2(\Omega_-)$ of the electric field, and the $H^{1/2}(\Gamma)$ and $H^{-1/2}(\Gamma)$ norms of the Dirichlet and Neumann acoustic traces of the approximate solutions. These norms provide a qualitative estimate of the energy of the system peaking initially as the incident wave interacts with the obstacle and then decaying in time as the internal reflections release the elastic and electric energy back into the fluid. Figure \ref{fig:7.3} shows the acoustic time signature of the total acoustic wave as measured in ten different points in the fluid.
\begin{figure}\center{\begin{tabular}{ccc}
\includegraphics[width=.31\linewidth]{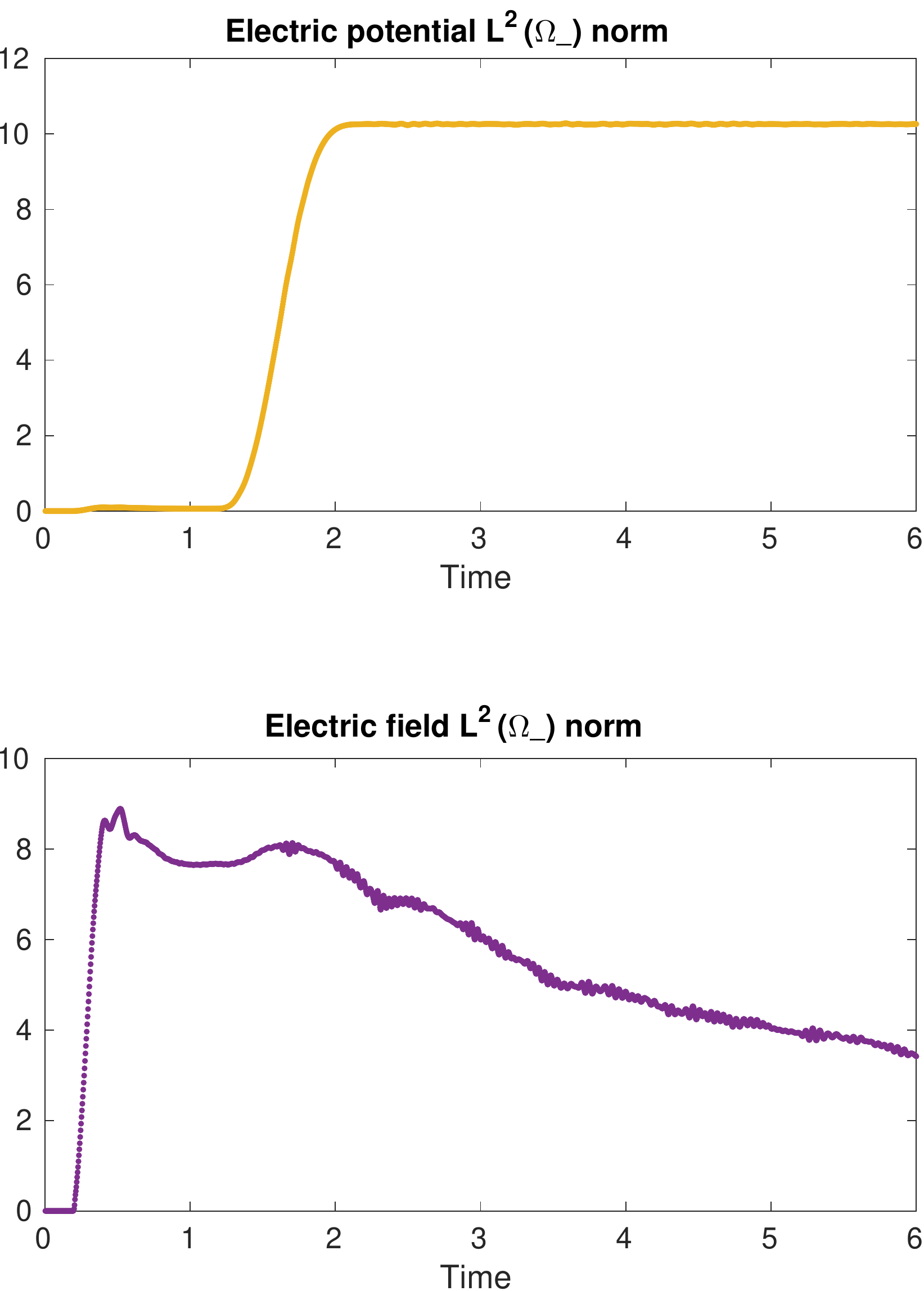} &
\includegraphics[width=.31\linewidth]{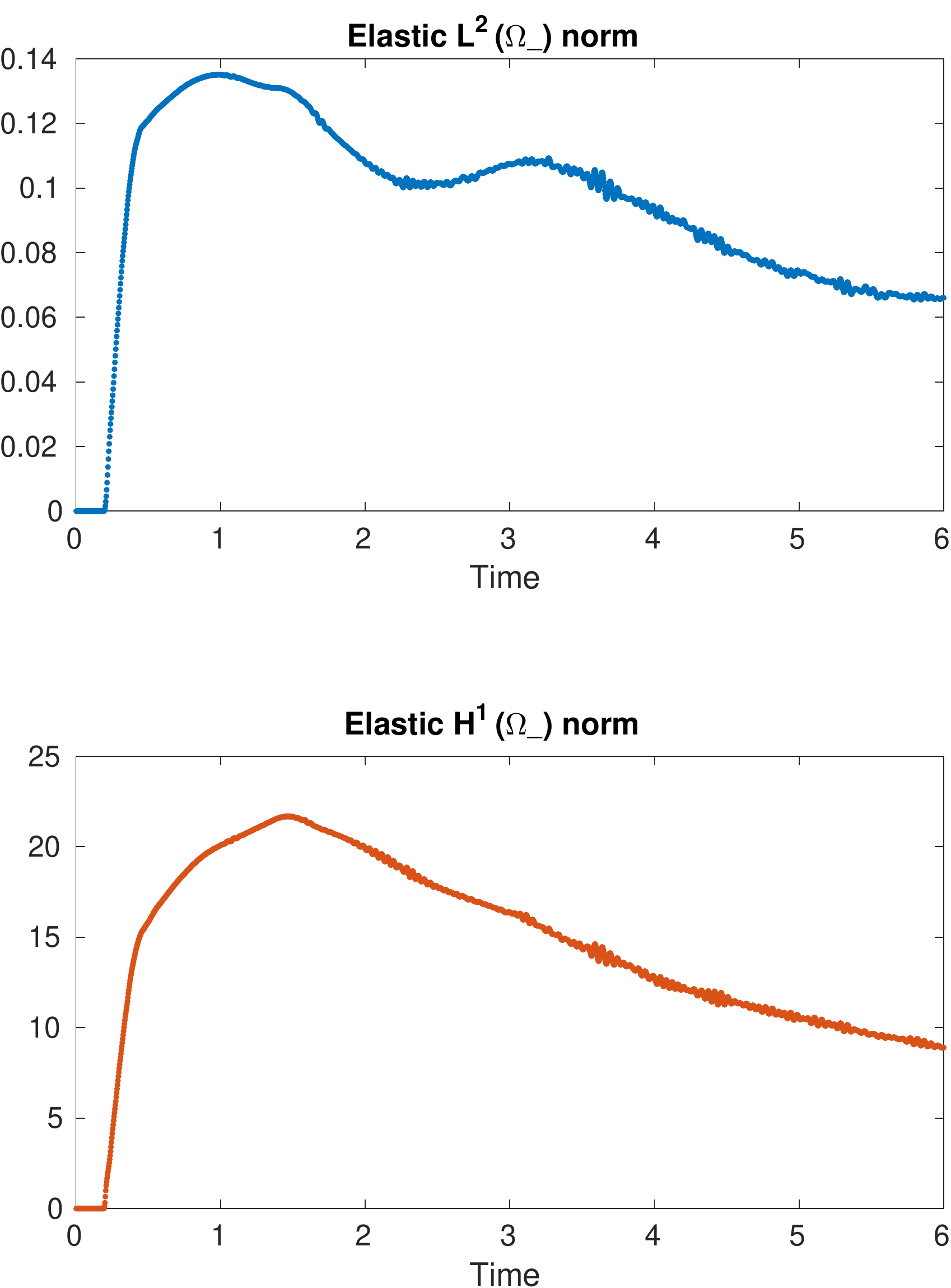} &
\includegraphics[width=.31\linewidth]{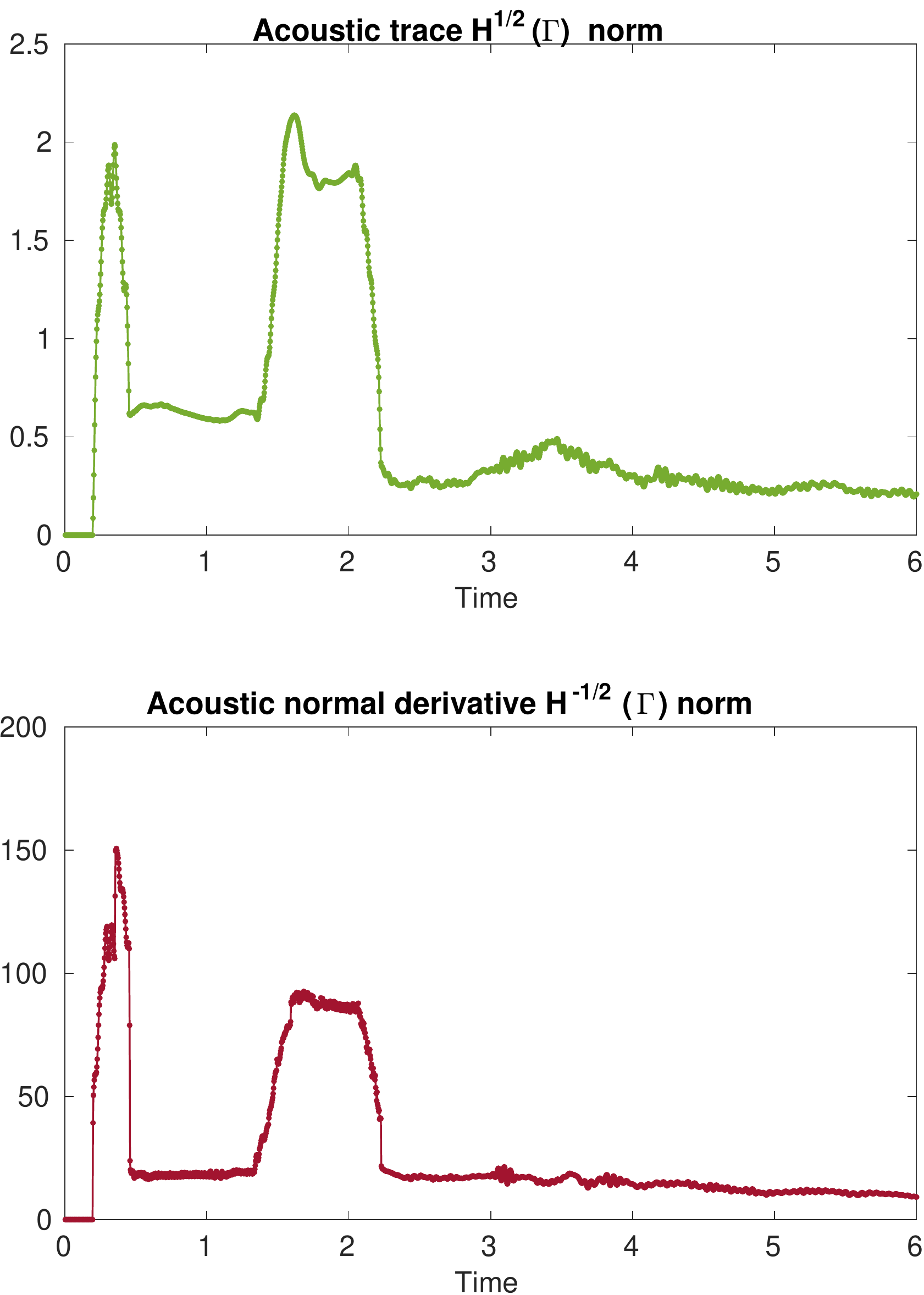}
\end{tabular}}
\caption{{\footnotesize Columnwise from left to right: norms of the electric potential, elastic displacement and acoustic densities as functions of time.}}\label{fig:7.2}
\end{figure}

\begin{figure}\center{ \begin{tabular}{ccc}
\includegraphics[width=.31\linewidth]{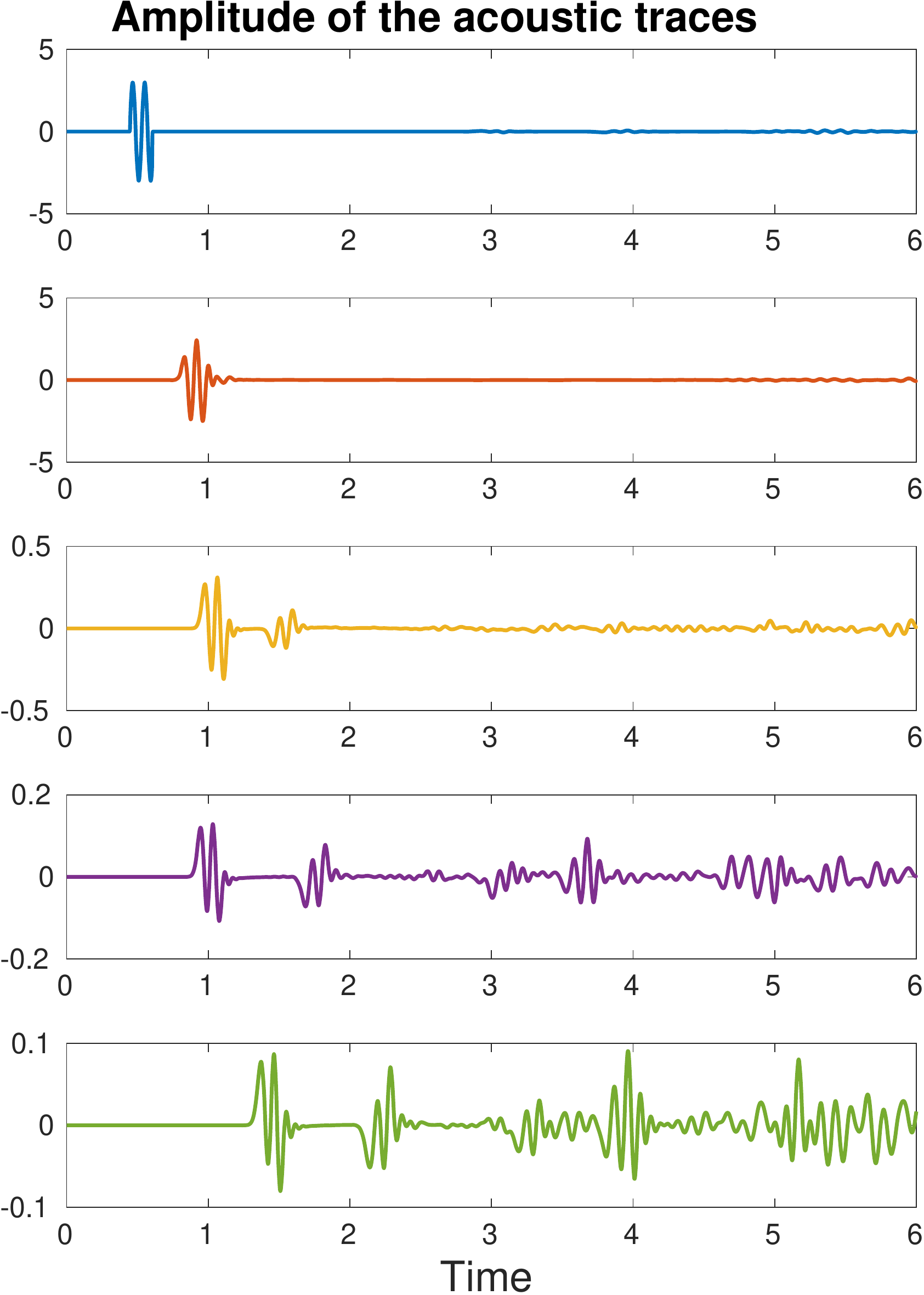} &
\includegraphics[width=.31\linewidth]{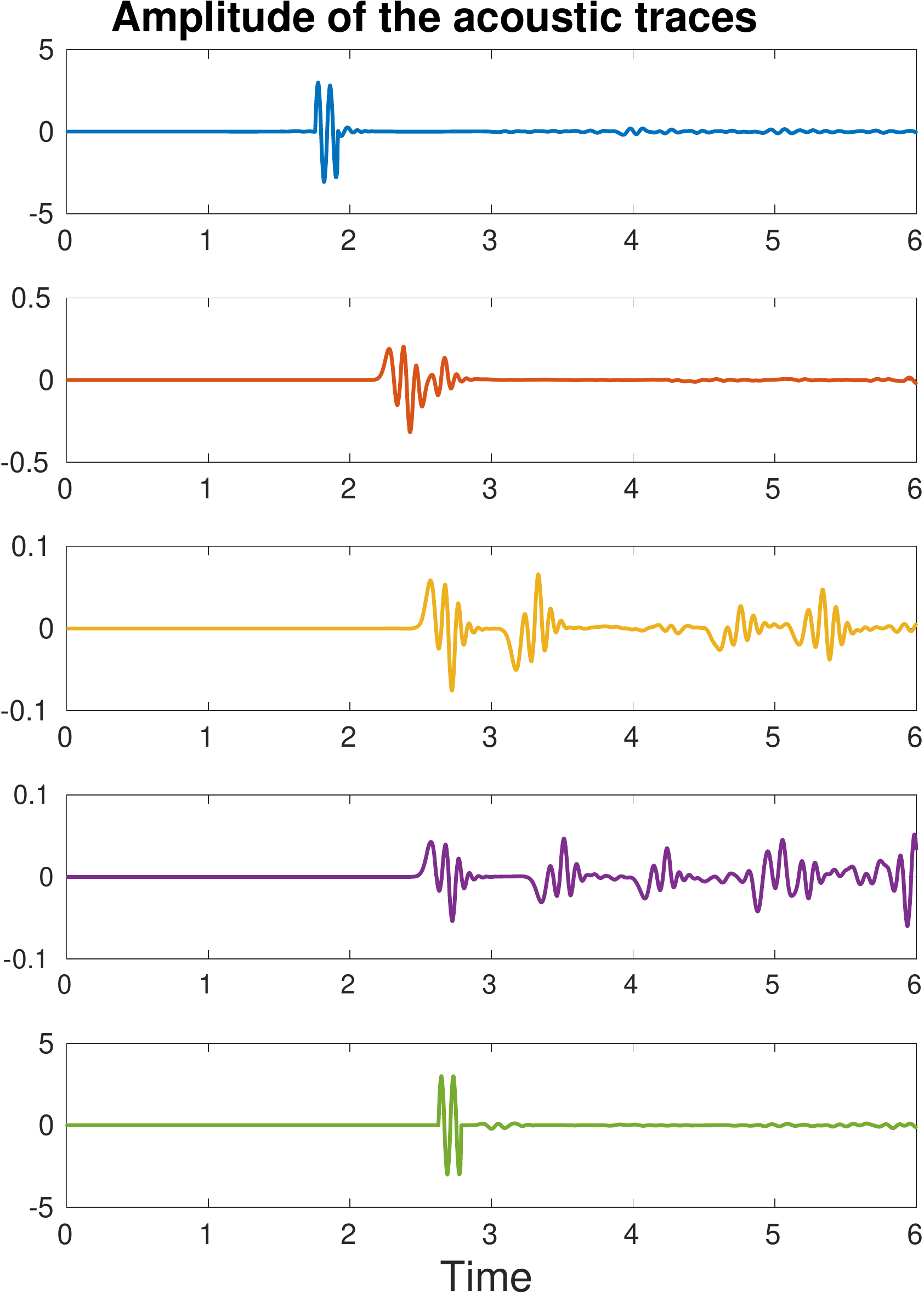} &
\includegraphics[width=.31\linewidth]{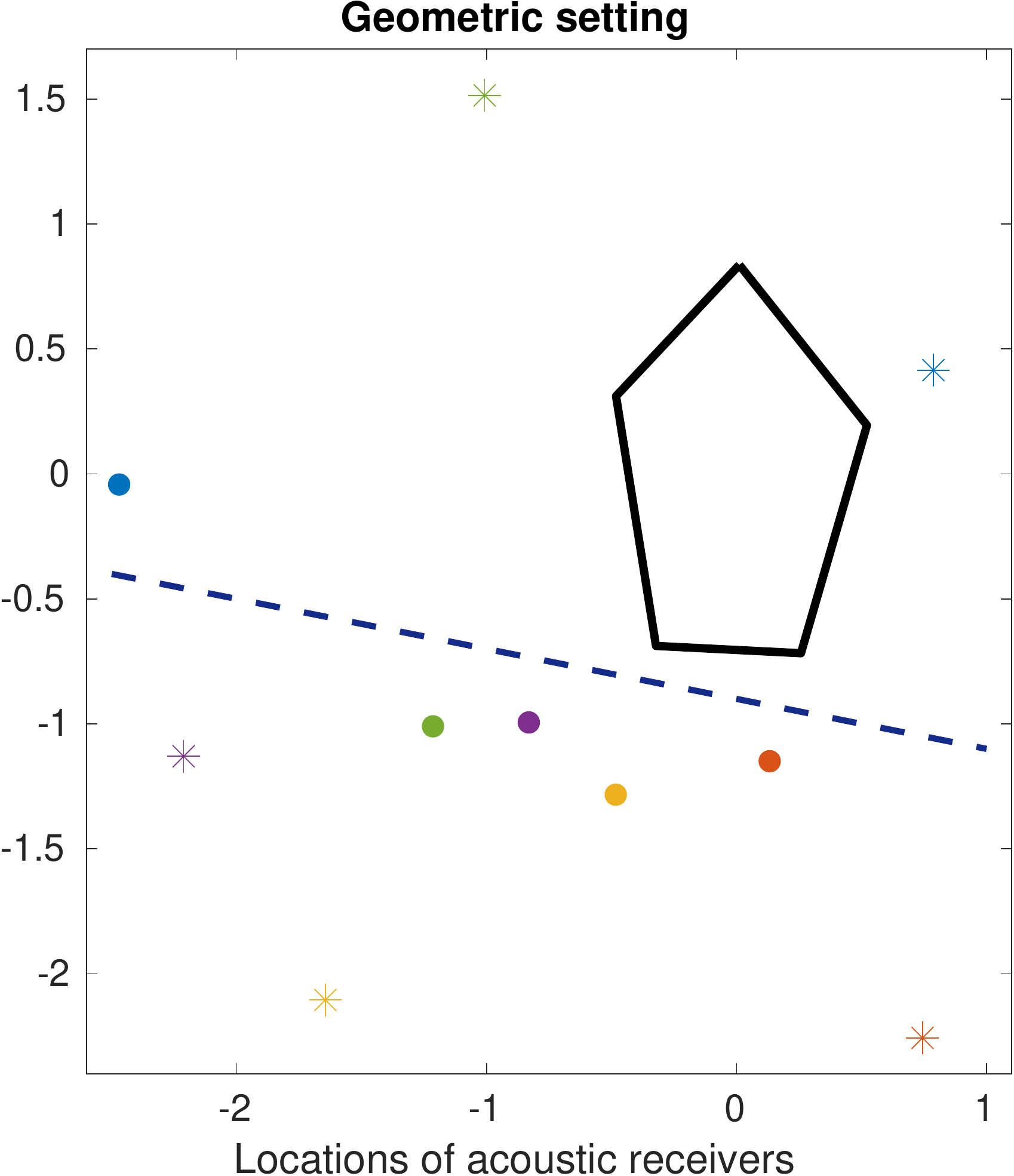}
\end{tabular}}
\caption{{\footnotesize The acoustic time-signatures on the left column correspond to receivers placed at the locations marked with a dot in the schematic, while the ones on the central column correspond to those marked with an asterisk. On the right: Locations of the acoustic receivers; the dotted line represents the initial location of the acoustic pulse which propagates in the direction $\mathbf d=(1,5)$.}}\label{fig:7.3}
\end{figure}

\paragraph{A trapping geometry.} 
The second experiment is done with the trapping geometry shown in Figure \ref{fig:7.1} (b), the density is now
\[
\rho_\Sigma = 20 + |x|+10|y|.
\]
All the remaining physical parameters are given in \eqref{eq:7.2}, while the grounding condition on the potential, imposed on the entire boundary is given by \eqref{eq:7.3} just like in the previous example. The profile of the acoustic pulse is the same as before \eqref{eq:7.1} but its propagation direction is changed to be $\mathbf d:=(-1,1)/\sqrt{2}$. The geometry was meshed with with 2992 triangular elements generated by MATLAB's {\tt pdetool}, with a mesh parameter $h=0.0172$. For the boundary element mesh we used 236 panels aligned with the finite element boundary edges. Hence the mesh approximation properties are similar to those of the previous example. We used $\mathcal P_3/\mathcal P_2$ continuous/discontinuous Galerkin boundary elements for the boundary element unknowns and $\mathcal P_3$ Lagrangian finite elements for the elastic displacement and electric potential.

Time discretization was done with the exact same parameters as before, final time $T=6$, 4000 equispaced points with step size $\kappa=0.0015$.  The behavior of the elastic, electric, and acoustic norms of the discrete solution are shown in \Cref{fig:7.4}, while the time signature of the total acoustic wave, at ten different points of the fluid and the locations of the receivers are given in Figure \ref{fig:7.5}. 

\begin{figure}\center{\begin{tabular}{ccc}
\includegraphics[width=.31\linewidth]{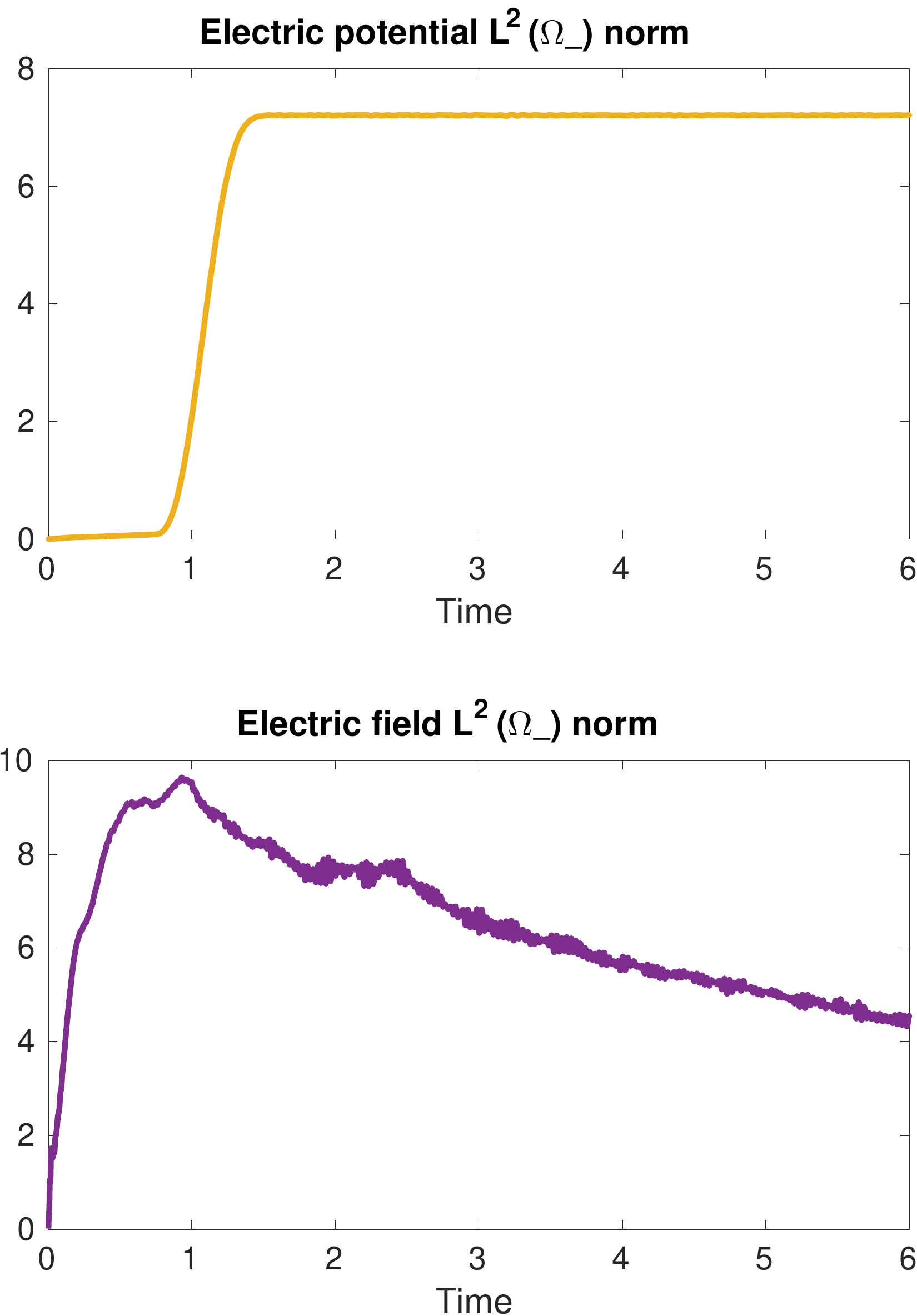} &
\includegraphics[width=.31\linewidth]{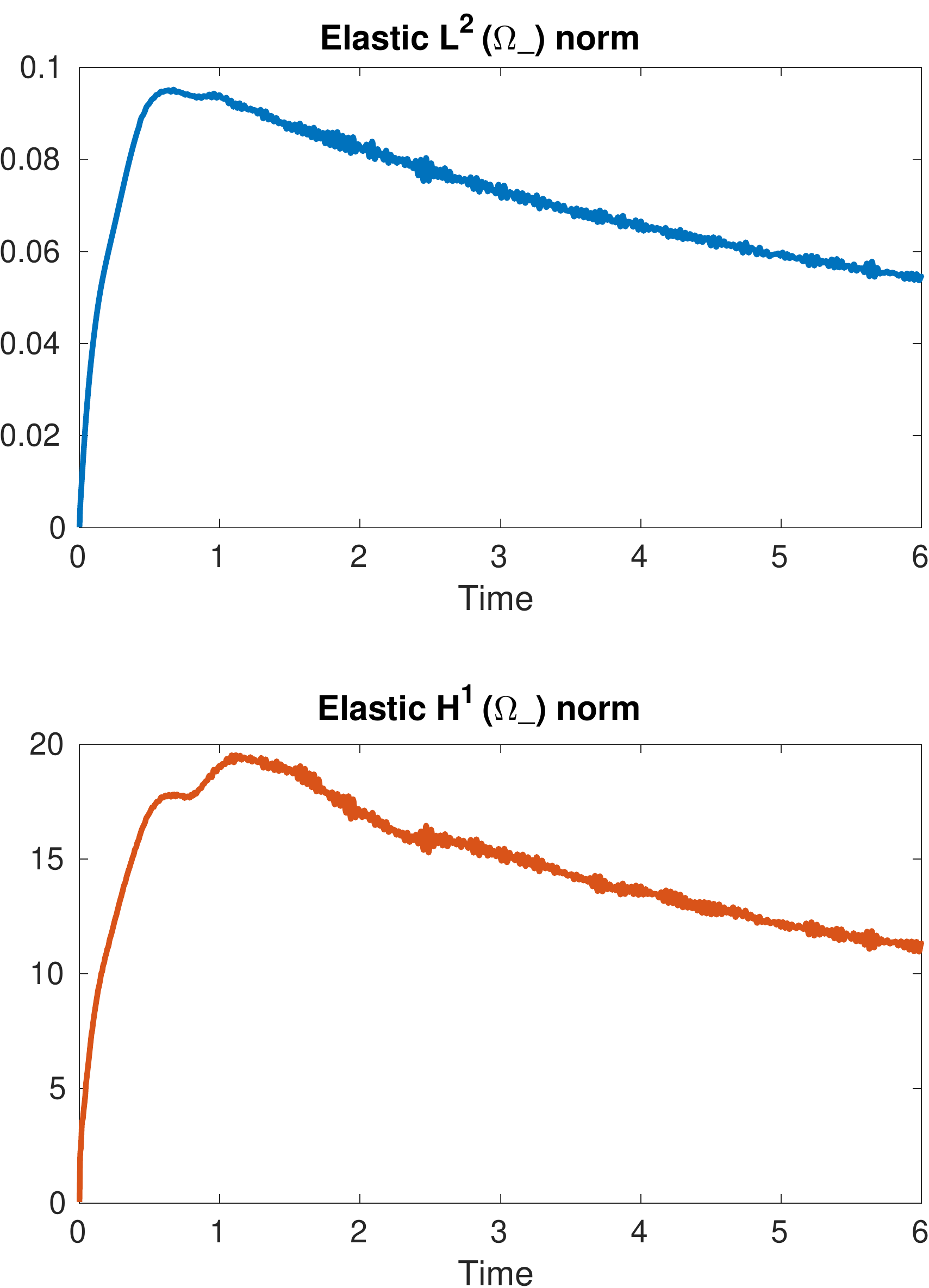} &
\includegraphics[width=.31\linewidth]{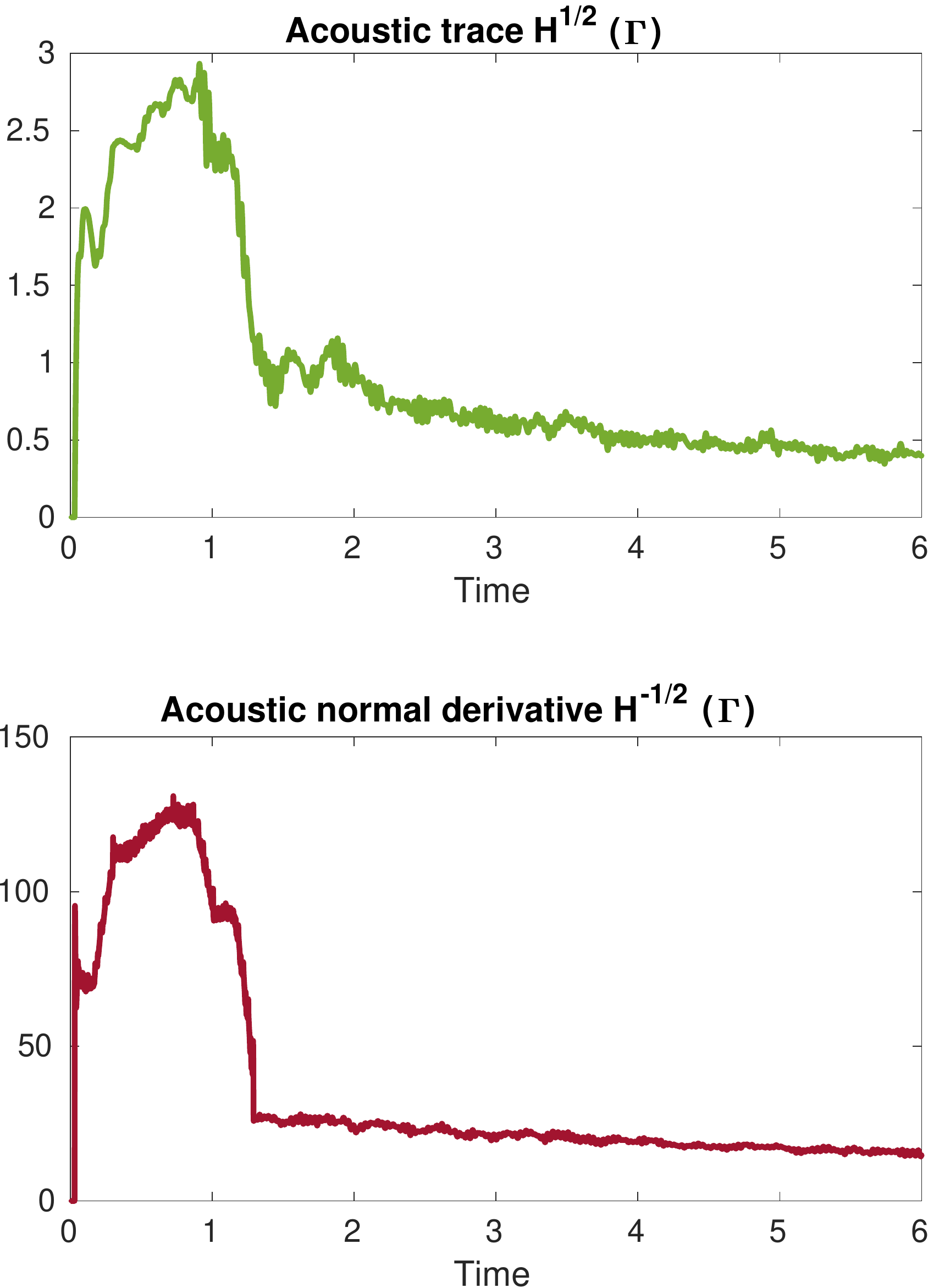}
\end{tabular}}
\caption{{\footnotesize Columnwise from left to right: norms of the electric potential, elastic displacement and acoustic densities as functions of time.}}\label{fig:7.4}
\end{figure}

\begin{figure}\center{\begin{tabular}{ccc}
\includegraphics[width=.31\linewidth]{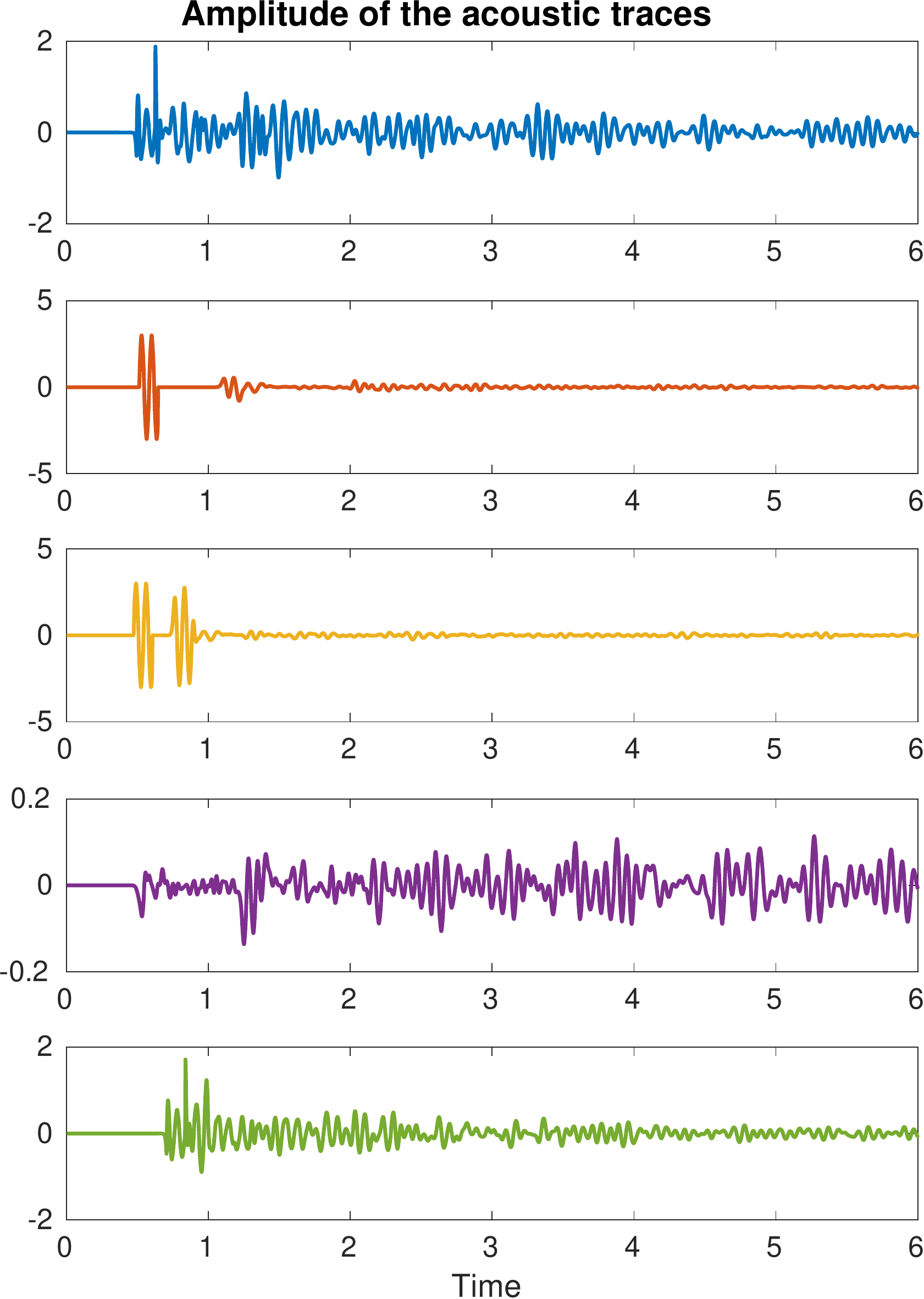}\,
\includegraphics[width=.31\linewidth]{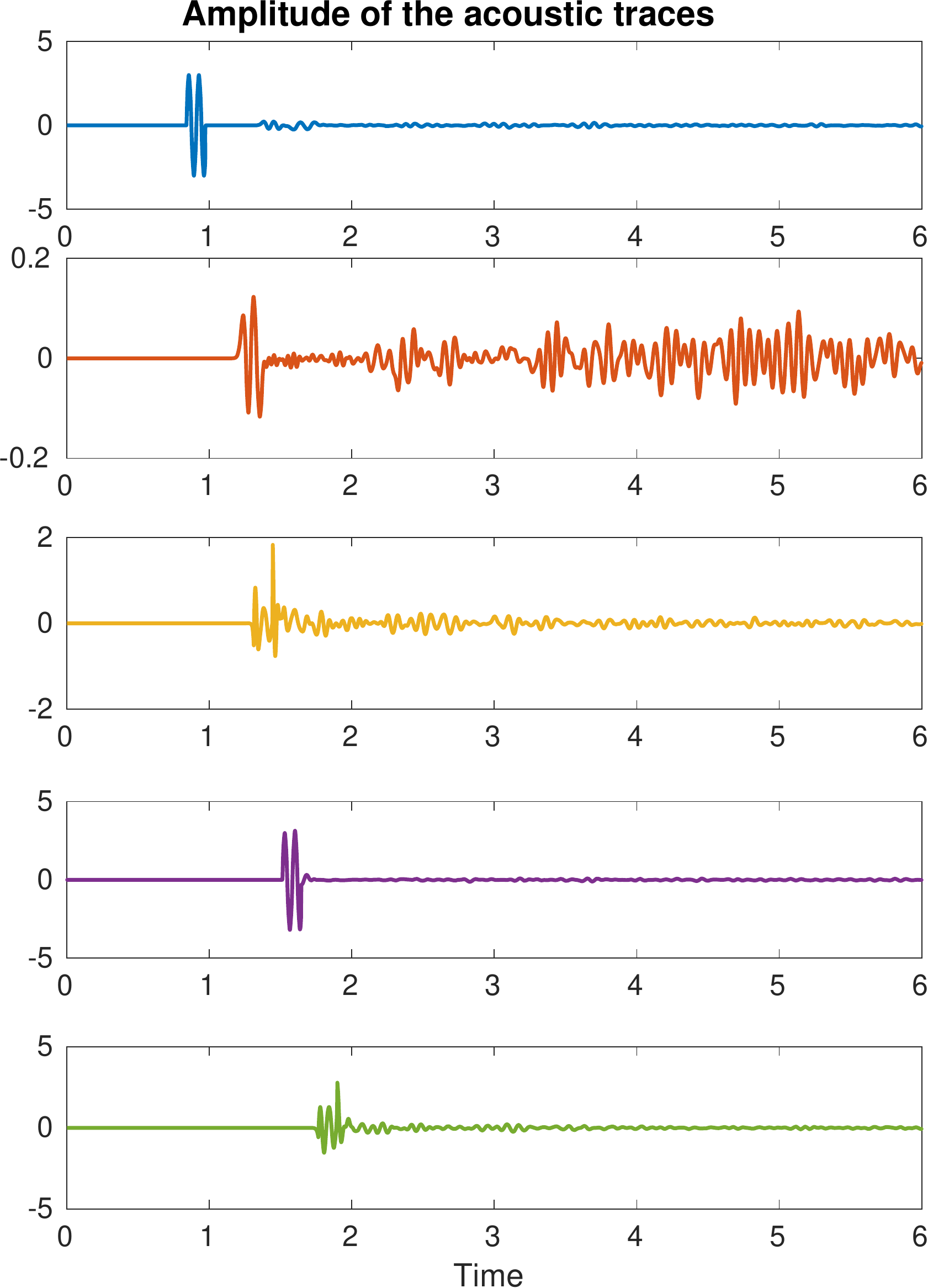}\,
\includegraphics[width=.31\linewidth]{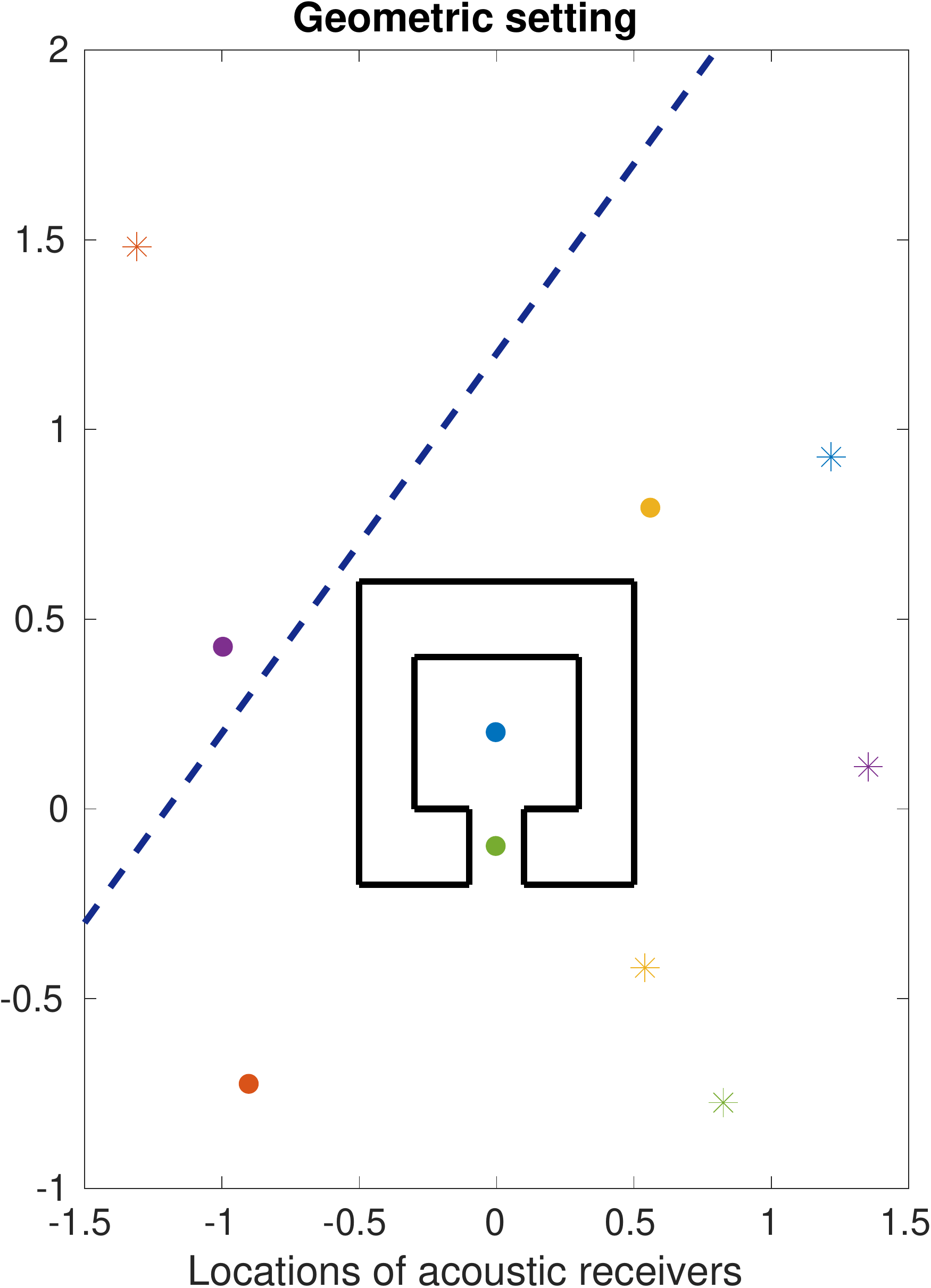}
\end{tabular}}
\caption{{\footnotesize The acoustic time-signatures on the left column correspond to receivers placed at the locations marked with a dot in the schematic, while the ones on the central column correspond to those marked with an asterisk. On the right: Locations of the acoustic receivers; the dotted line represents the initial location of the acoustic pulse which propagates in the direction $\mathbf d =(1,-1)$.}}\label{fig:7.5}
\end{figure}

\paragraph{Transition into time-harmonic.} 
We consider now the case where a causal sinusoidal wave (as opposed to the windowed pulses from the previous two experiments) given by 
\[
v^{inc} = 3\mathcal H(\tau)\sin(6\pi\tau)\,, \quad \tau:= t-\mathbf r\cdot\mathbf d\,, \quad \mathbf r:=(x,y)\,, \mathbf d := (-1,1)/\sqrt{2},
\]
impinges into the trapping scatterer depicted in Figure \ref{fig:7.5} (c) with mass distribution
\[
\rho_\Sigma = 20 + |x|+50|y|,
\]
and grounding condition on the potential on the entire boundary given by \eqref{eq:7.3}.  Since the incident wave is no longer compactly supported in time, the wavefield is expected to transition into the time-harmonic regime. The interaction is simulated for $t\in[0,7]$ using the same geometry as in the previous example, BDF2-CQ with a time step of size $\kappa=1.5\times 10^{-3}$ and mesh parameter $h=0.0172$, but with increased polynomial degree $\mathcal P_4$ for the Lagrangian finite element mesh and $\mathcal P_4/\mathcal P_3$ continuous/discontinuous Galerkin BEM for the acoustic wave.

As a measure of the energy in the system, we plot the $\mathbf L^2(\Omega_-)$, $L^2(\Omega_-)$, and $\mathbf H^1(\Omega_-)$ norms of the approximate finite element solutions and the $H^{-1/2}(\Gamma)$ and $H^{1/2}(\Gamma)$ norms of the approximate acoustic densities in Figure \ref{fig:7.6}. Regarding the acoustic wavefield, Figure \ref{fig:7.7} shows the total acoustic signal at 10 different locations and the geometric setup of the experiment. Snapshots of the total acoustic field and the magnitude of the elastic displacement are shown in Figure \ref{fig:7.8} at different stages of the transition.

\begin{figure}\center{\begin{tabular}{ccc}
\includegraphics[width=.31\linewidth]{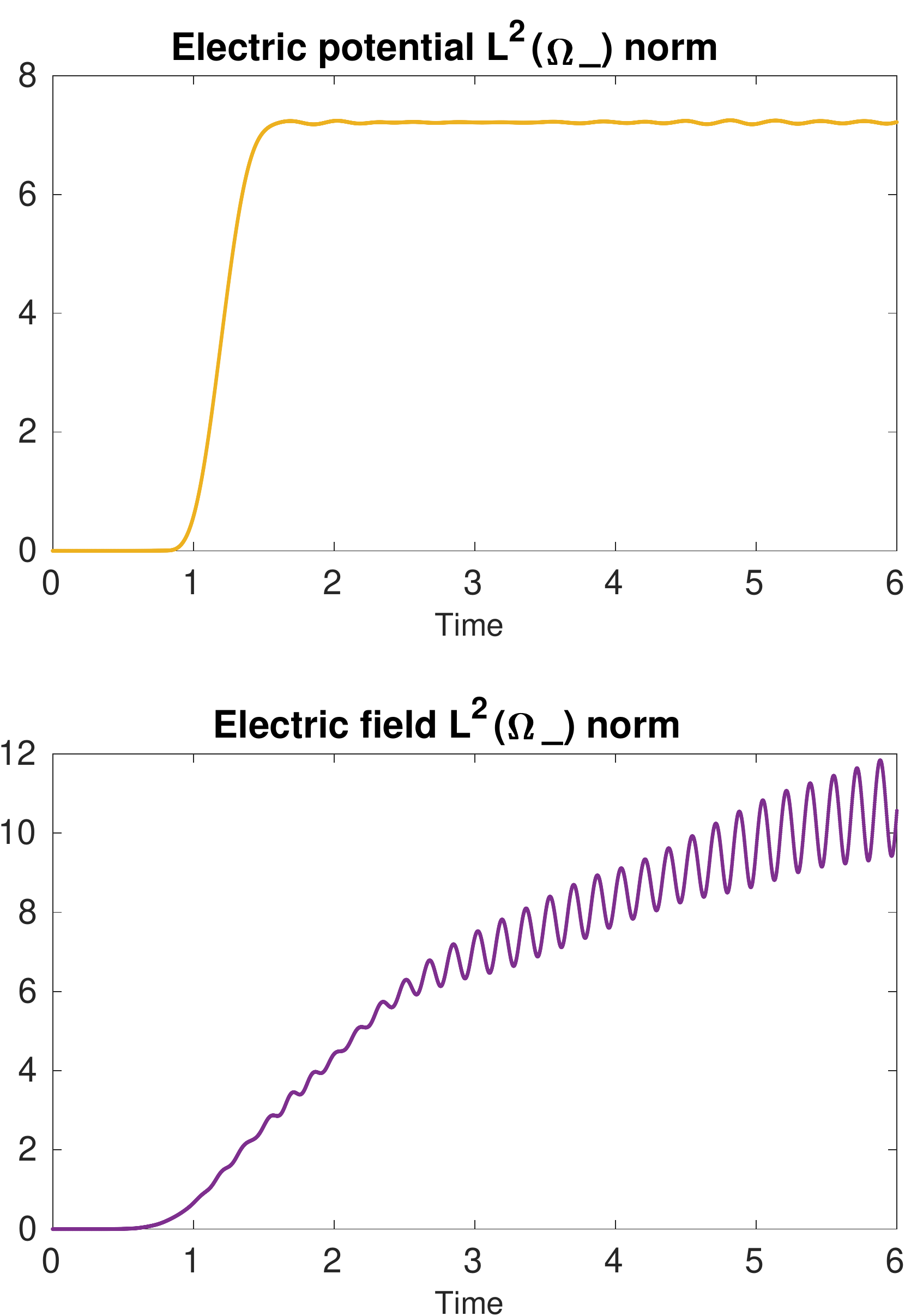} &
\includegraphics[width=.31\linewidth]{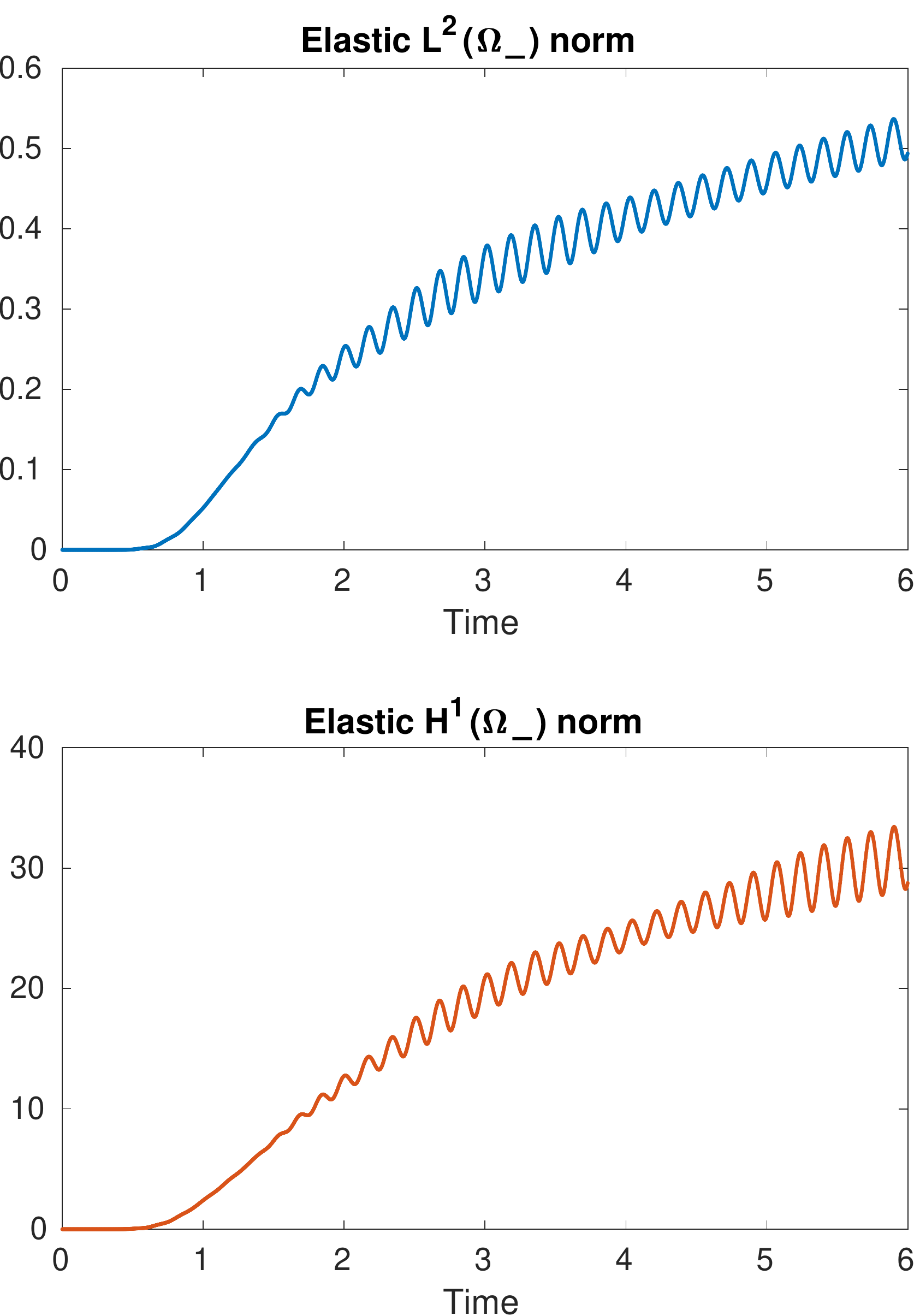} &
\includegraphics[width=.31\linewidth]{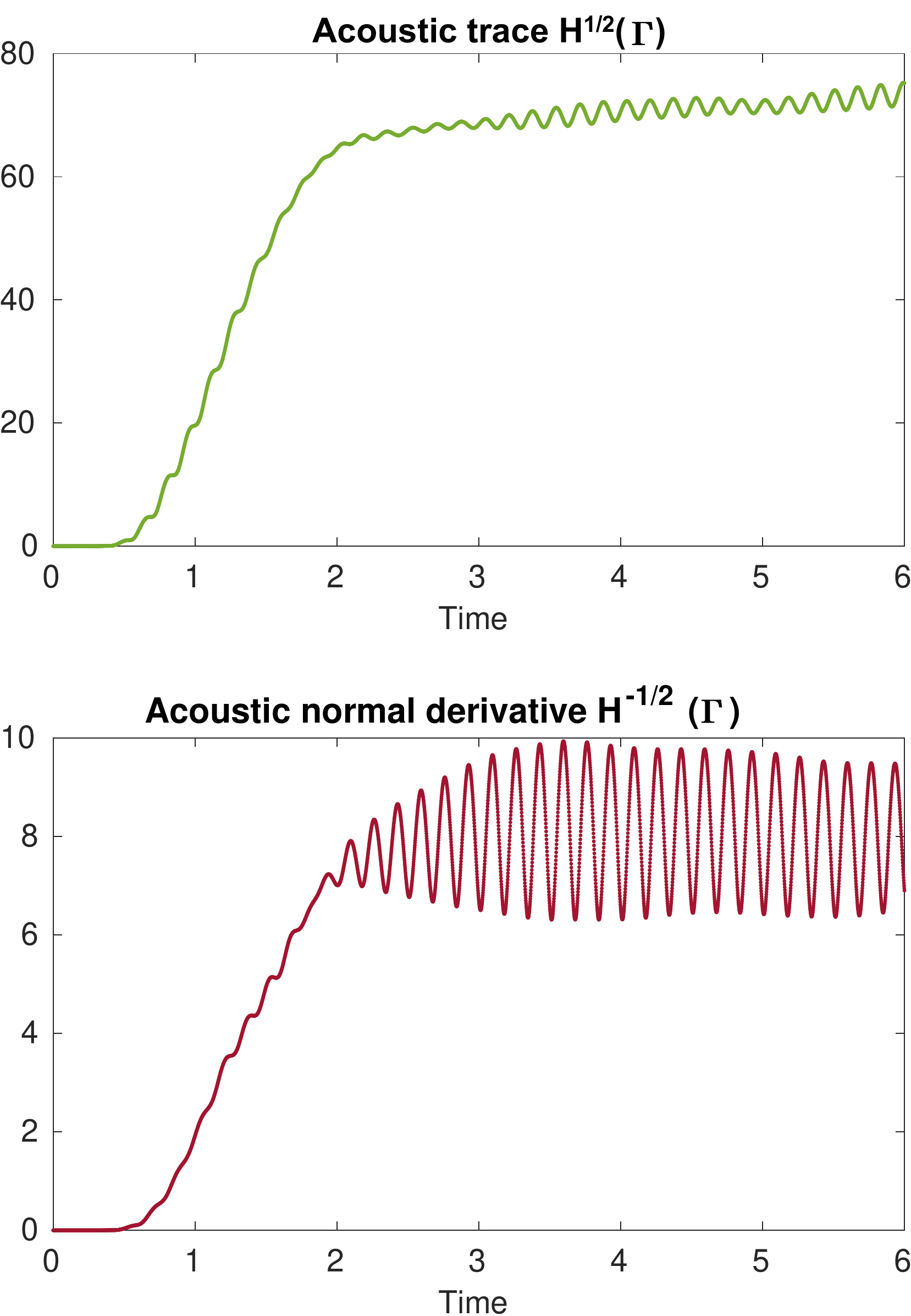}
\end{tabular}}
\caption{{\footnotesize Columnwise from left to right: norms of the electric potential, elastic displacement and acoustic densities as functions of time.}}\label{fig:7.6}
\end{figure}

\begin{figure}\center{\begin{tabular}{ccc}
\includegraphics[width=.31\linewidth]{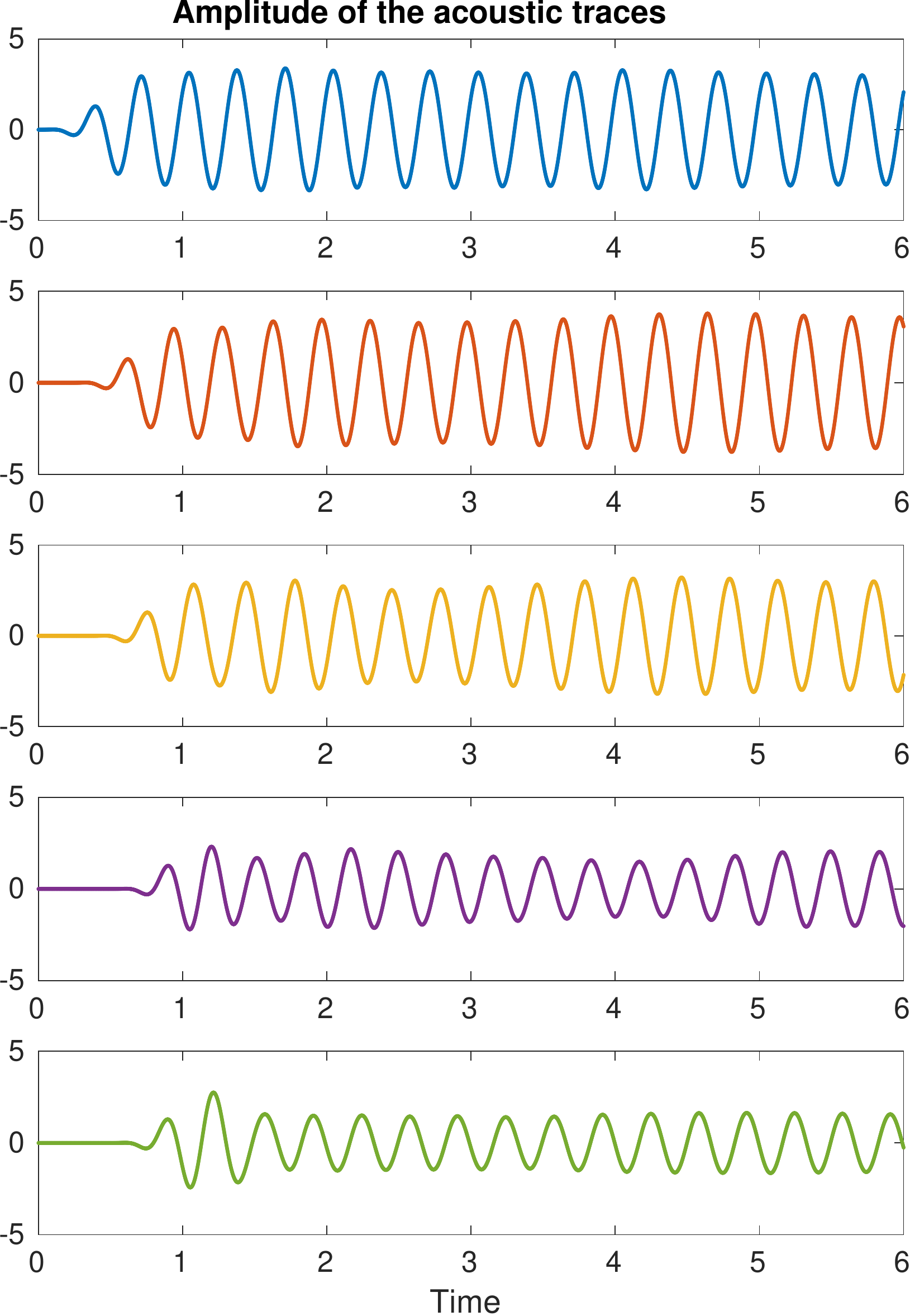}\,
\includegraphics[width=.31\linewidth]{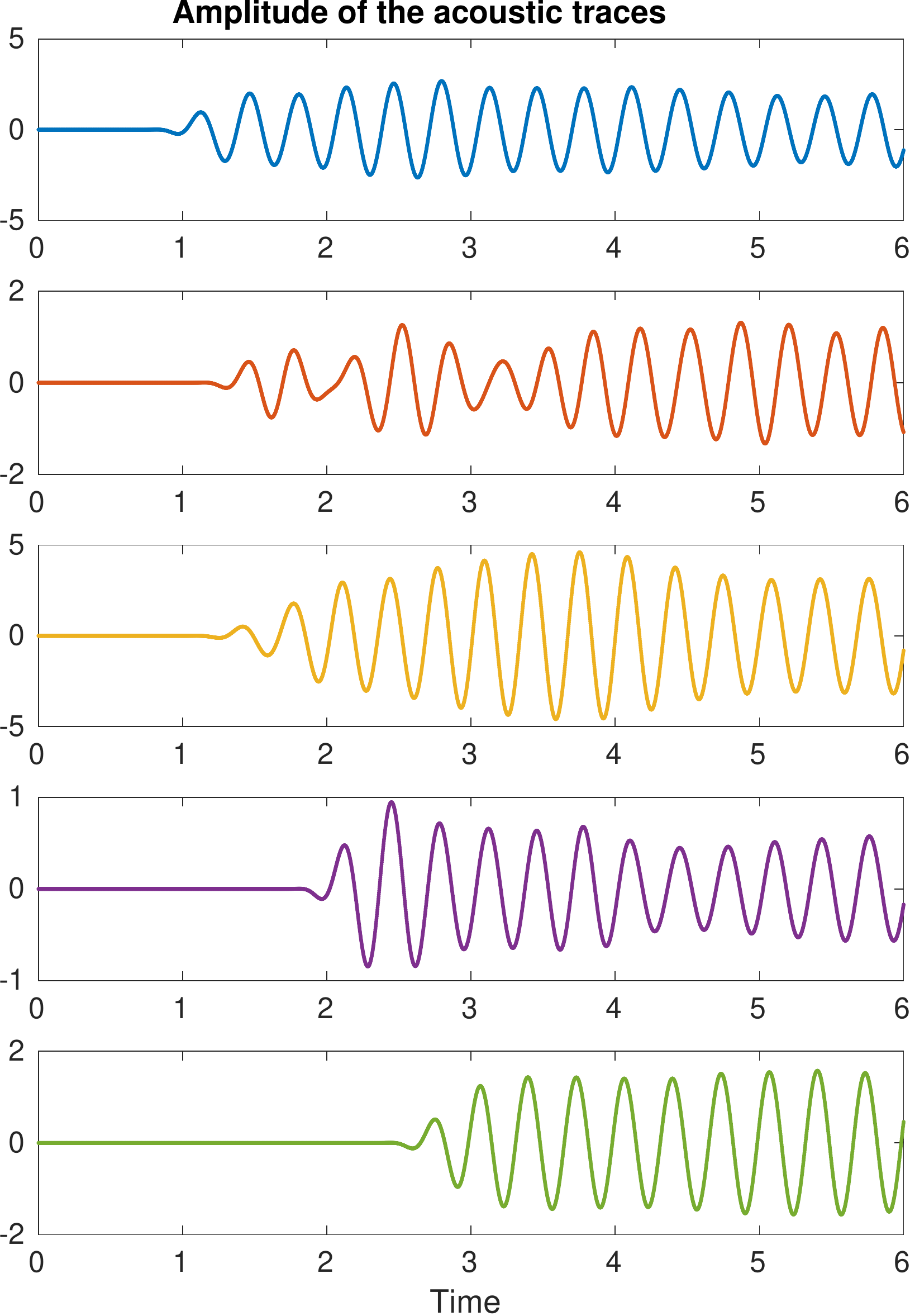}\,
\includegraphics[width=.31\linewidth]{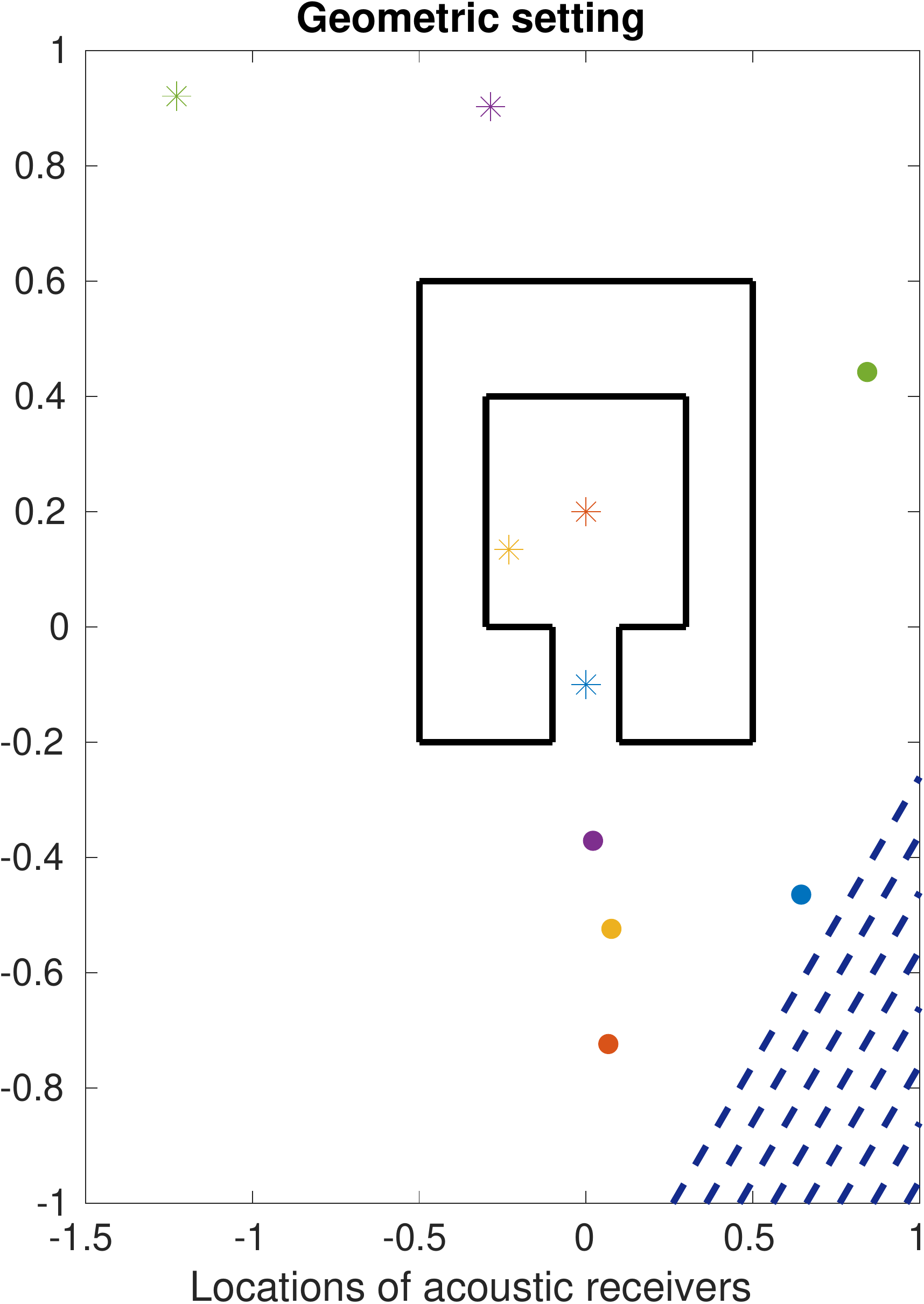}
\end{tabular}}
\caption{{\footnotesize The acoustic time-signatures on the left column correspond to receivers placed at the locations marked with a dot in the schematic, while the ones on the central column correspond to those marked with an asterisk. On the right: Locations of the acoustic receivers; the dotted lines represent the initial location of the acoustic wave which propagates in the direction $\mathbf d =(-1,1)$.}}\label{fig:7.7}
\end{figure}

\begin{figure}\center{\begin{tabular}{ccc}
\includegraphics[width=.31\linewidth]{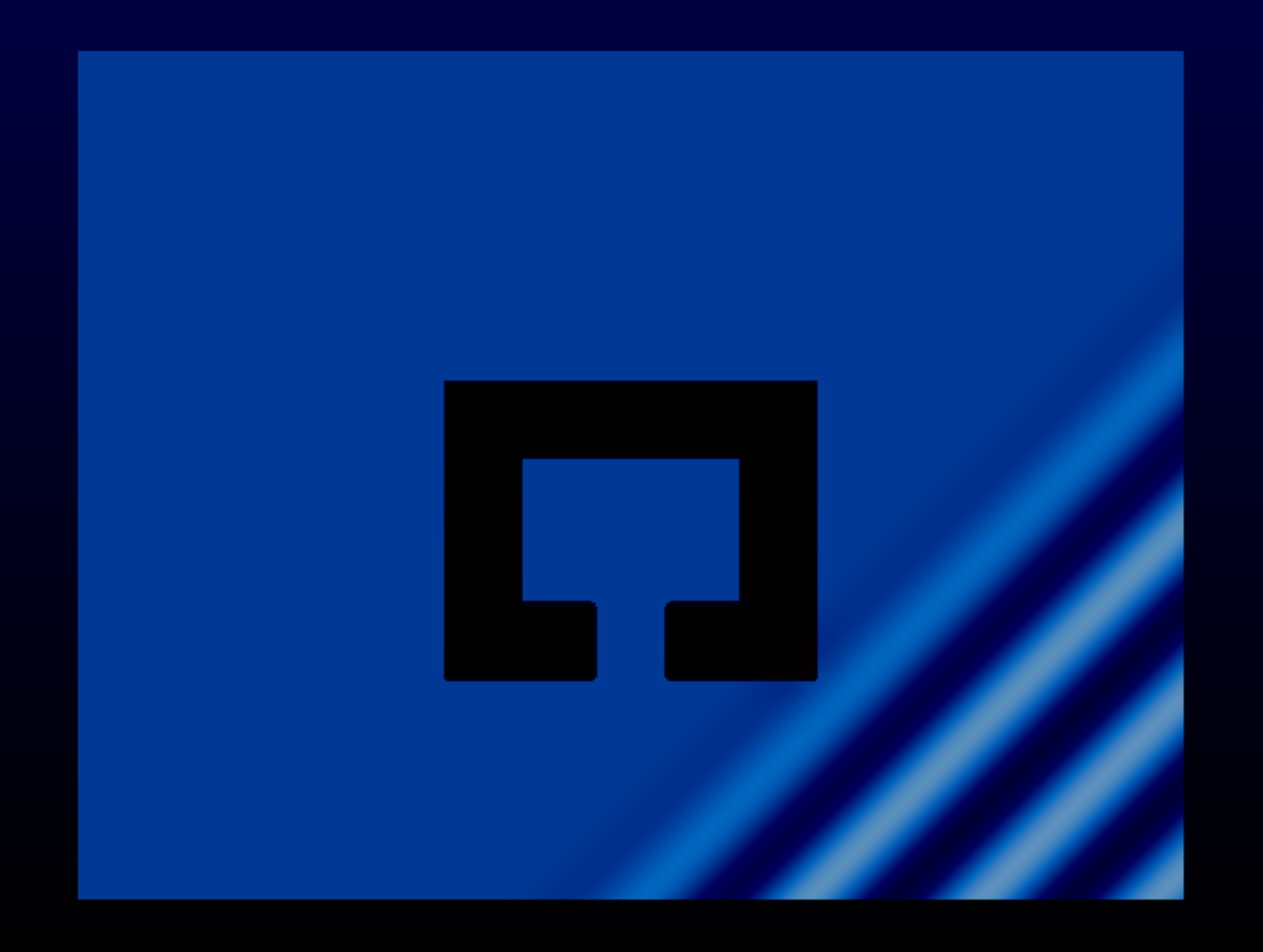}  &
\includegraphics[width=.31\linewidth]{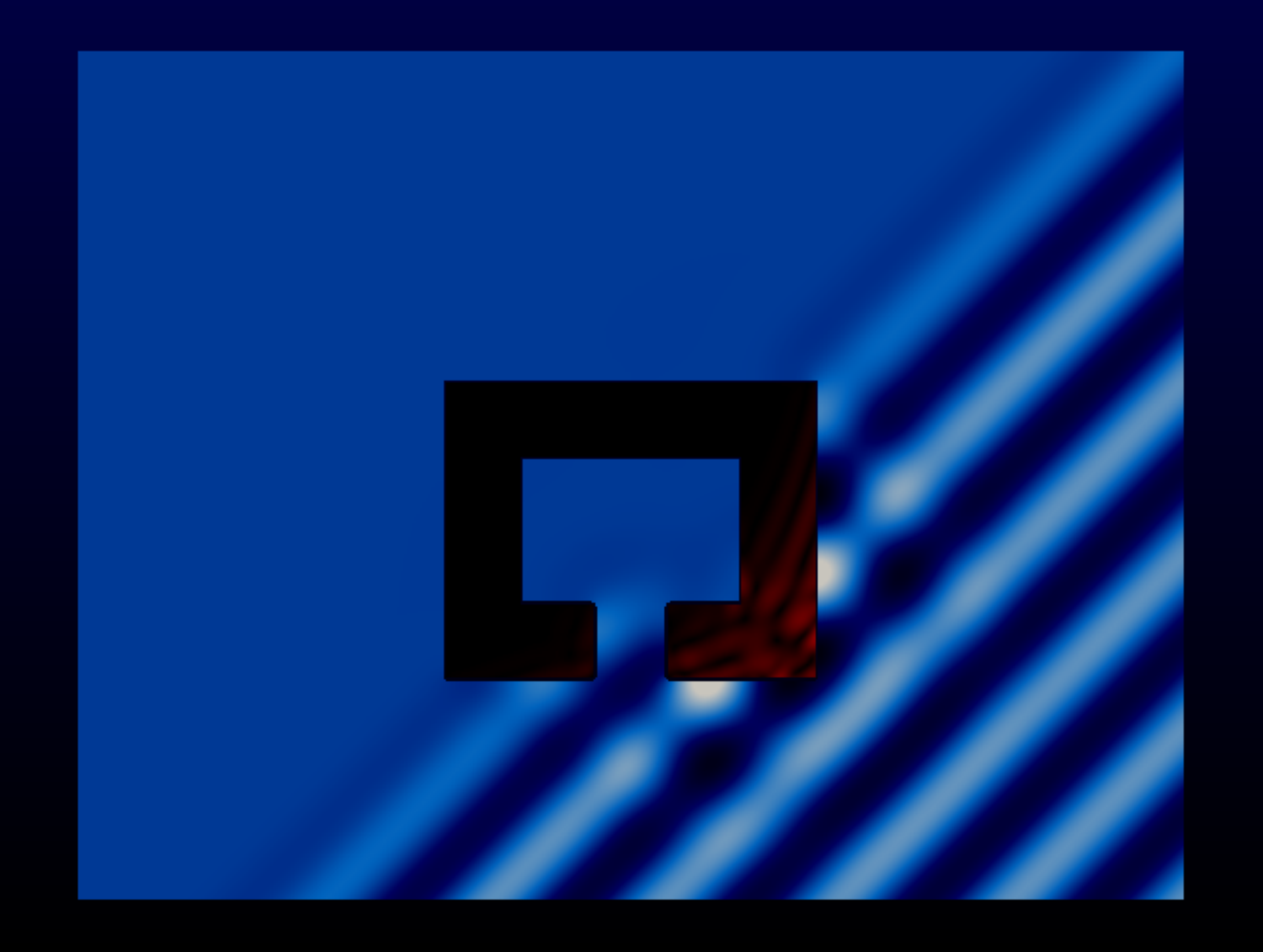} &
\includegraphics[width=.31\linewidth]{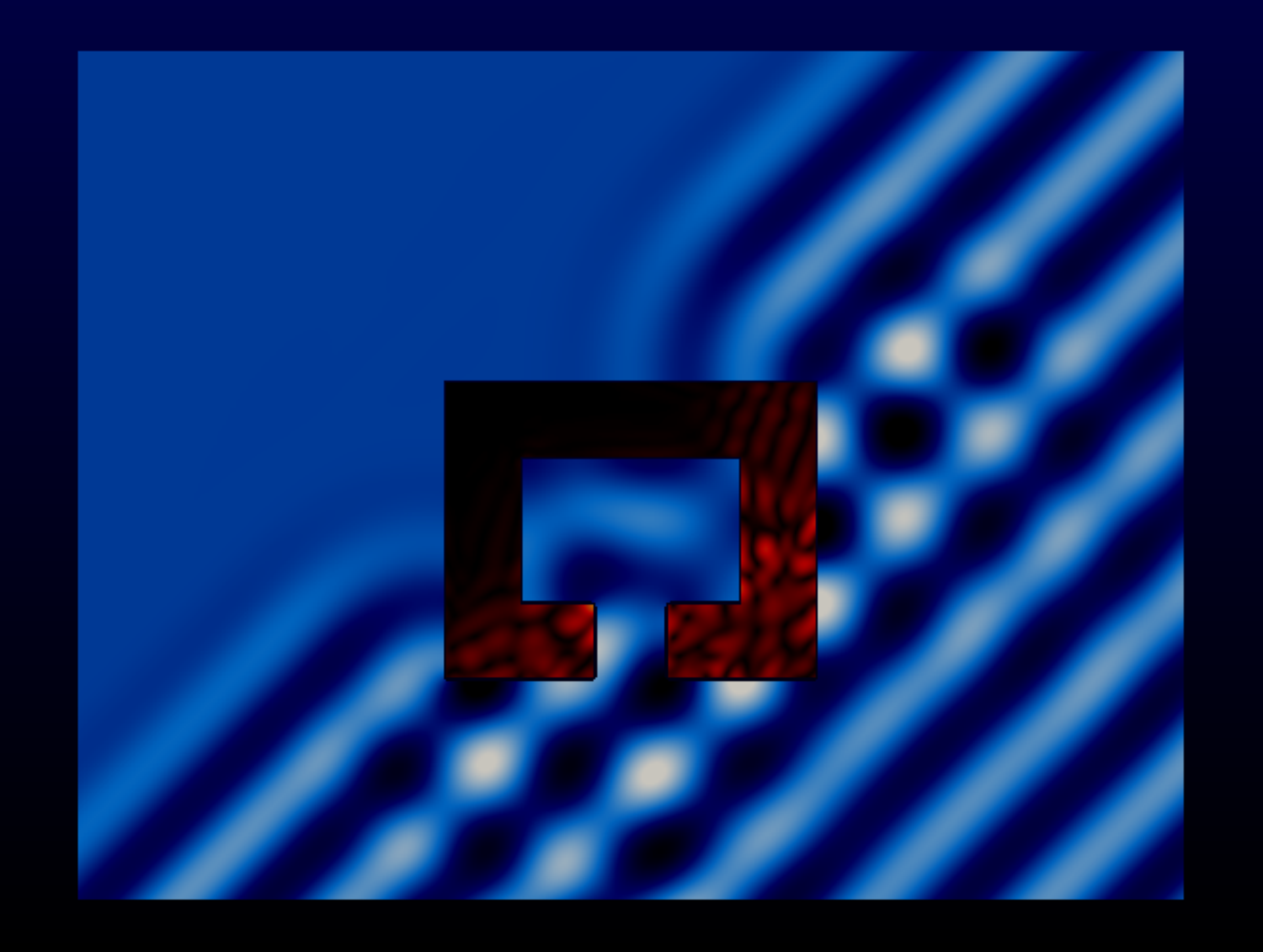}\\
\includegraphics[width=.31\linewidth]{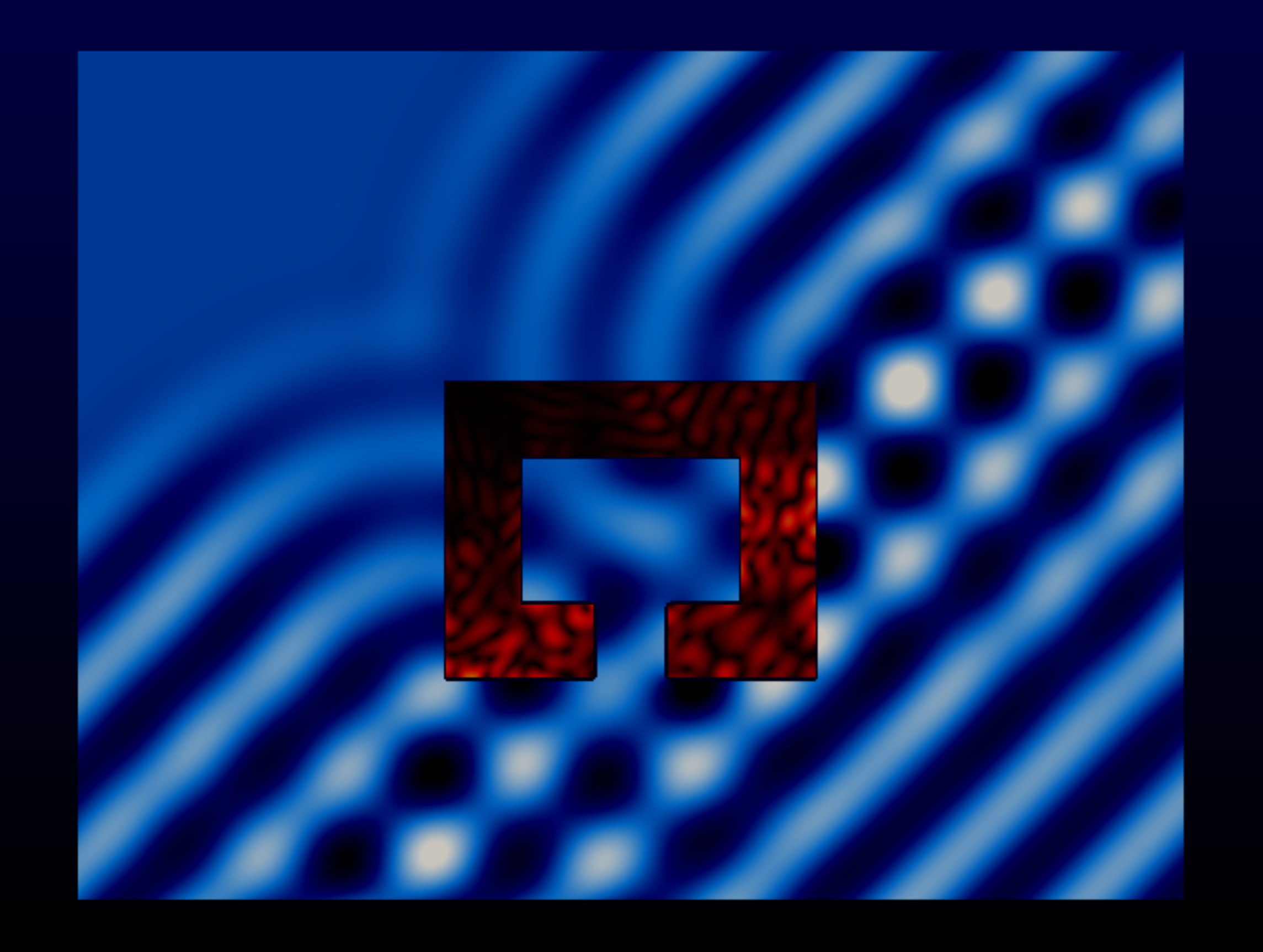} &
\includegraphics[width=.31\linewidth]{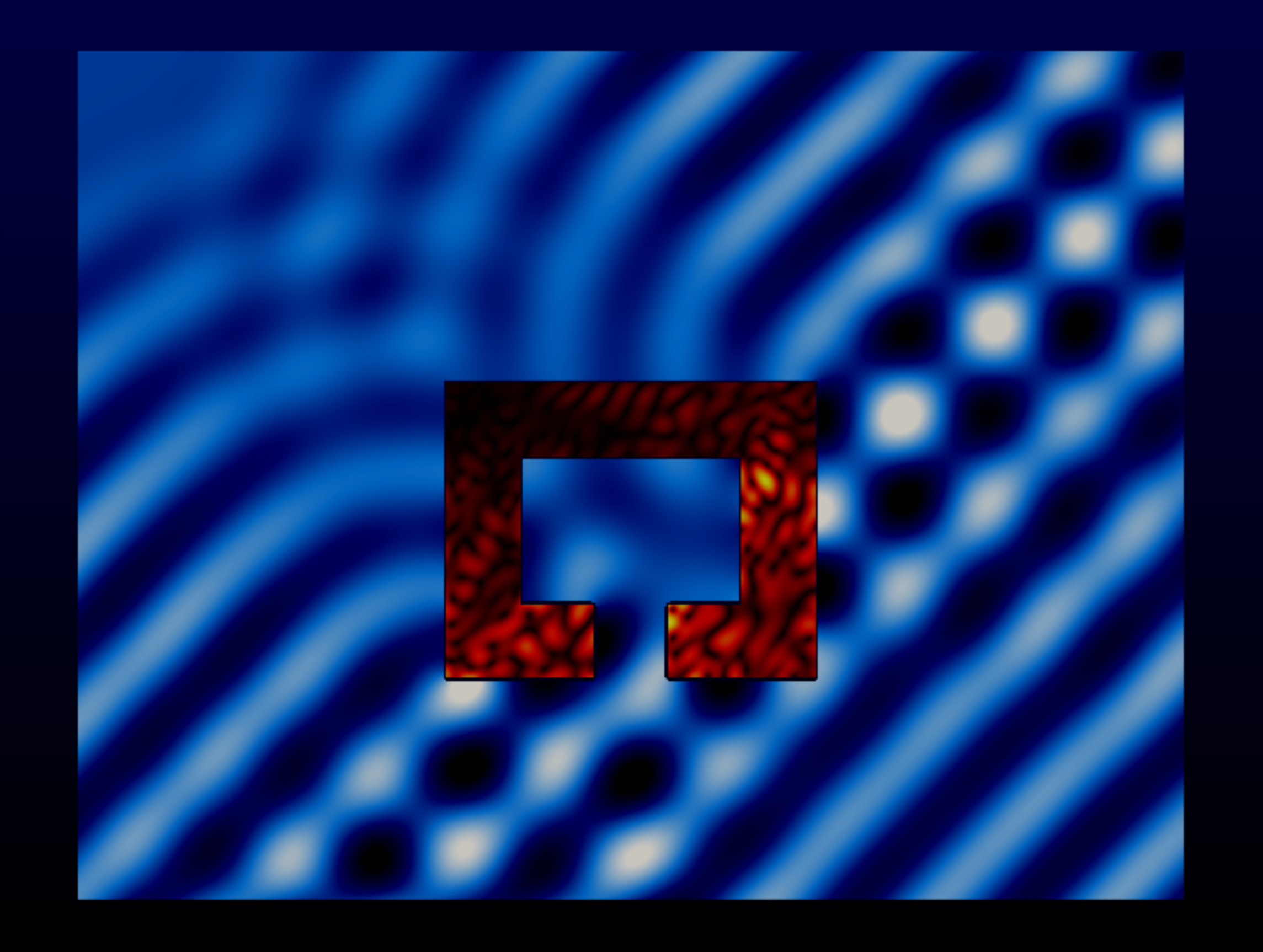} &
\includegraphics[width=.31\linewidth]{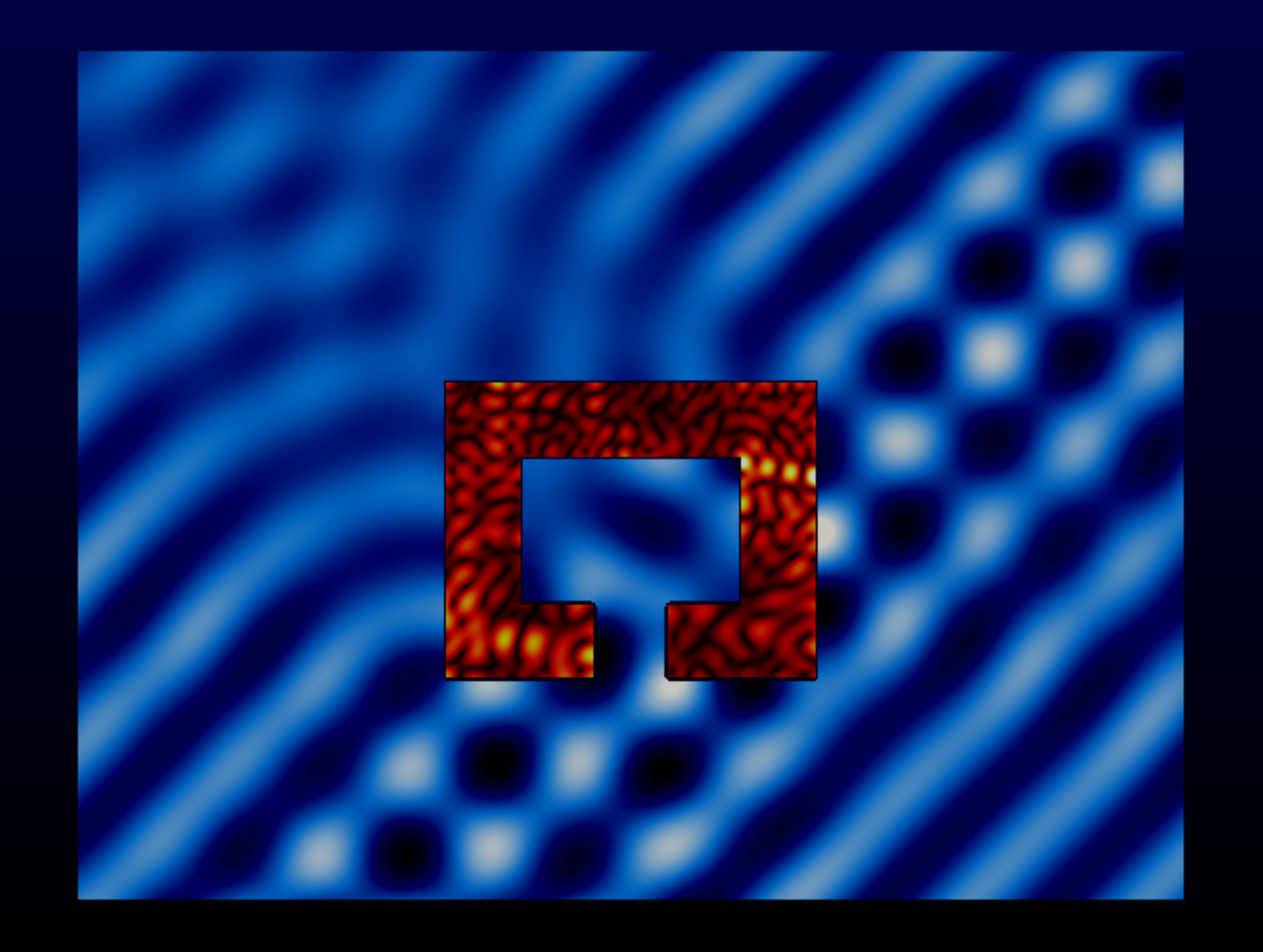}
\end{tabular}}
\caption{{\footnotesize Due to the periodic nature of the incident acoustic wave, the resulting interacting fields transition into a time-harmonic regime. We present snapshots of the total acoustic field and the magnitude of the elastic displacement for times $t=0.75, 1.5, 2.25, 3, 3.75, 7.5$.}}\label{fig:7.8}
\end{figure}

\paragraph{Generation of acoustic/elastic waves.} 
In this example we present the situation where there is no acoustic incident wave, but the grounding potential oscillates periodically in time. The alternating electric potential generates elastic stress that propagates in the interior of the solid and acts as a source of acoustic waves in the exterior domain through the interface coupling conditions. This situation is of practical relevance, since a piezoelectric patch can be used both as sensor of external vibrations and -when the grounding potential is tuned accordingly- as generator of waves that can neutralize the external field. 

In the example we use the pentagonal geometry of the first experiment depicted in figure \ref{fig:7.1a} with the same physical parameters but imposing the time-dependent grounding condition
\[
\psi = 6\mathcal H(t)\sin(4\pi t),
\]
imposed on the entire interface $\Gamma$. The space discretization was done using the same MATLAB generated triangulation with mesh parameter $h=1.72 \times 10^{-2}$ with 2992 elements and a matching Boundary Element mesh with 236 panels. The elastic and electric unknowns were approximated using $\mathcal P_4$ Lagrangian finite elements and the acoustic field used $\mathcal P_4/\mathcal P_3$ continuous/discontinuous Galerkin boundary elements. For time discretization BDF2 time stepping and BDF2-Based CQ were used for a uniform time grid with $\kappa= 1.2 \times 10^{-3}$ in the interval $t\in[0, 6]$. As a measure of the system's energy, the norms of the discrete elastic, electric and acoustic approximations are shown in figure \ref{fig:7.9}, while Figure \ref{fig:7.10} shows the time signatures of the generated acoustic wavefield in ten different locations in the fluid. Given the periodic nature of the forcing potential, the behavior of the system eventually becomes harmonic. Snapshots of the post-processed acoustic pressure and the elastic displacement for different times are shown in Figure \ref{fig:7.11}.

\begin{figure}\center{\begin{tabular}{ccc}
\includegraphics[width=.31\linewidth]{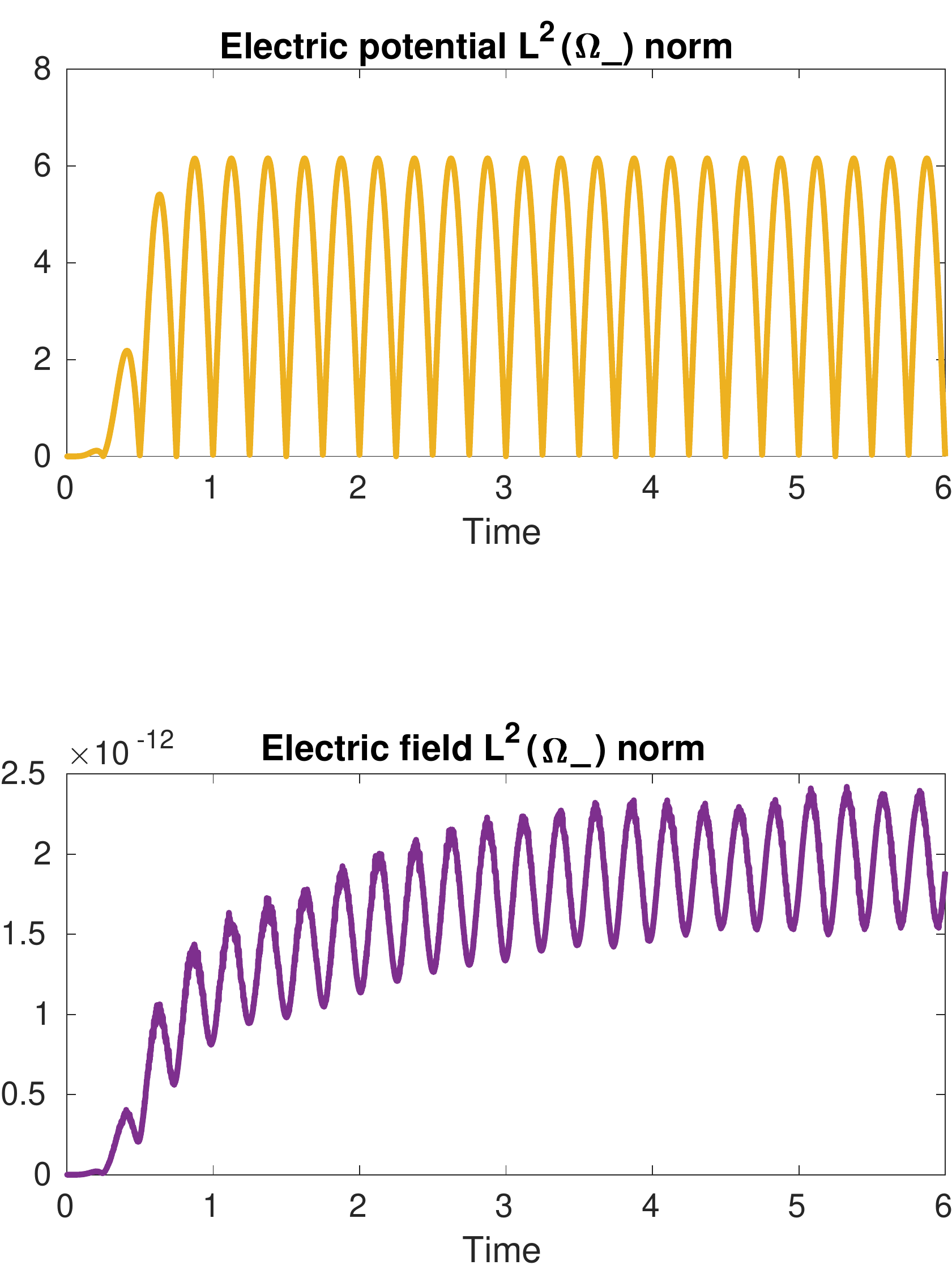} &
\includegraphics[width=.31\linewidth]{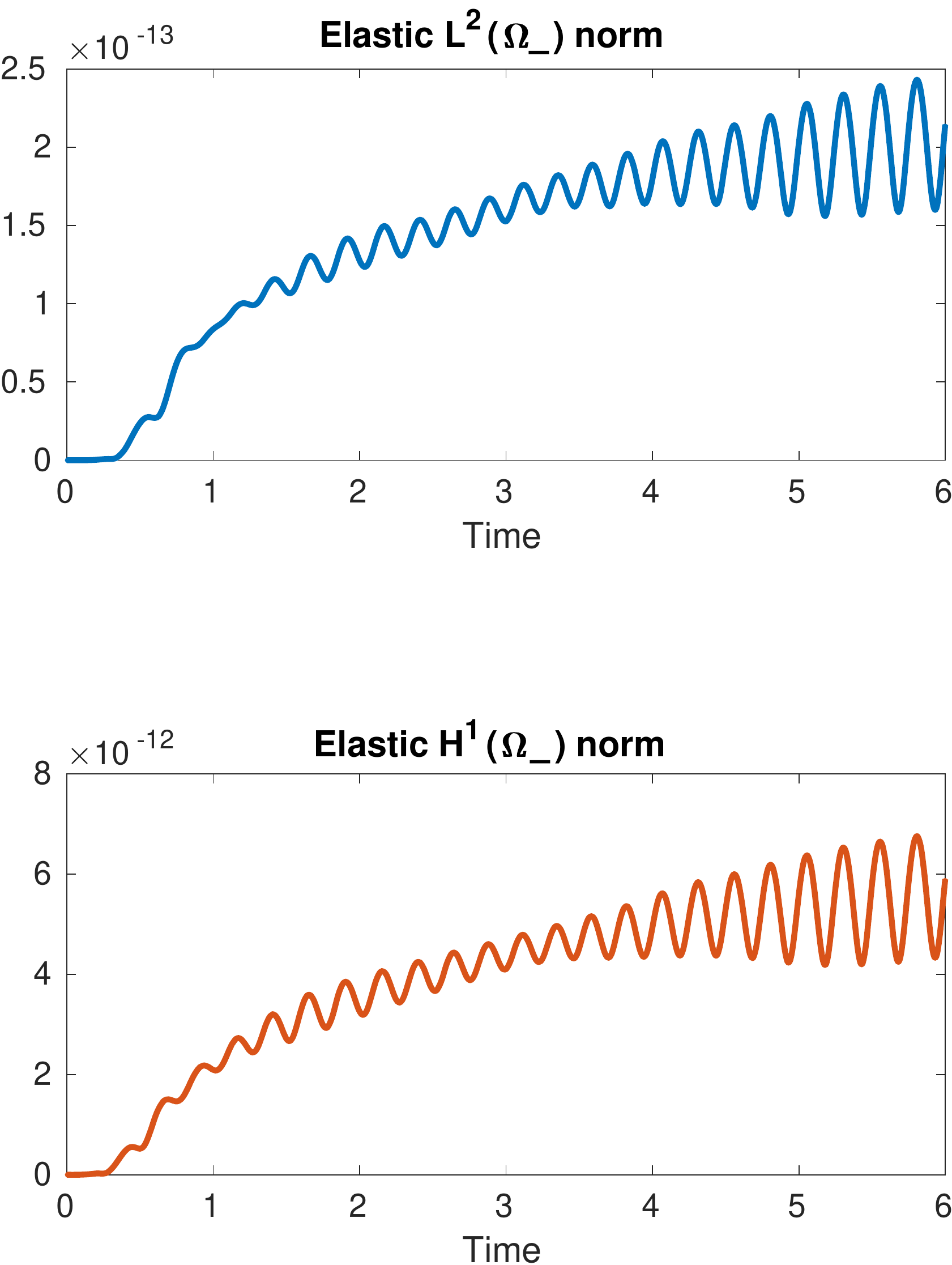}  &
\includegraphics[width=.31\linewidth]{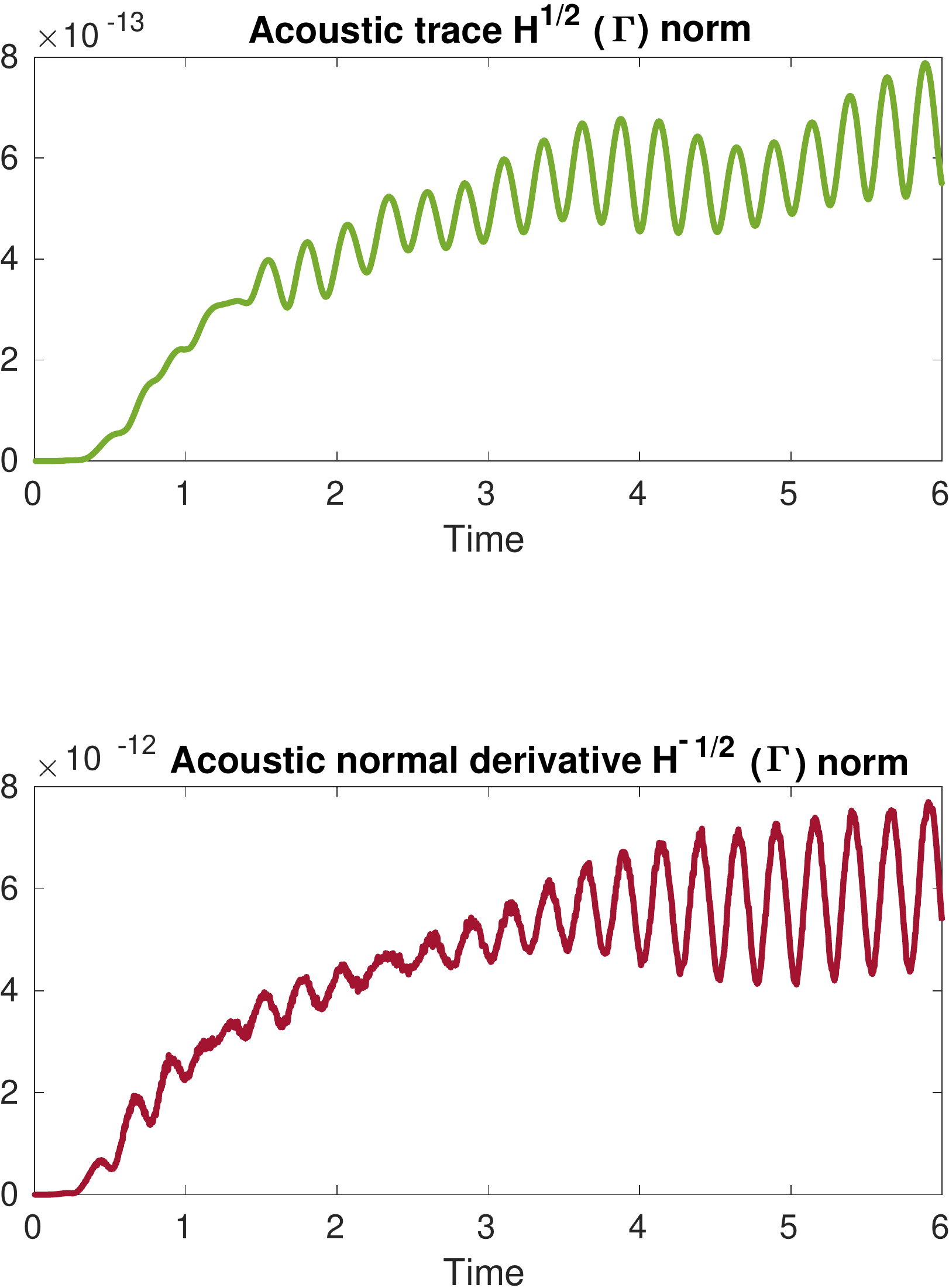}
\end{tabular}}
\caption{{\footnotesize Columnwise from left to right: norms of the electric potential, elastic displacement and acoustic densities as functions of time.}}\label{fig:7.9}
\end{figure}

\begin{figure}\center{\begin{tabular}{ccc}
\includegraphics[width=.31\linewidth]{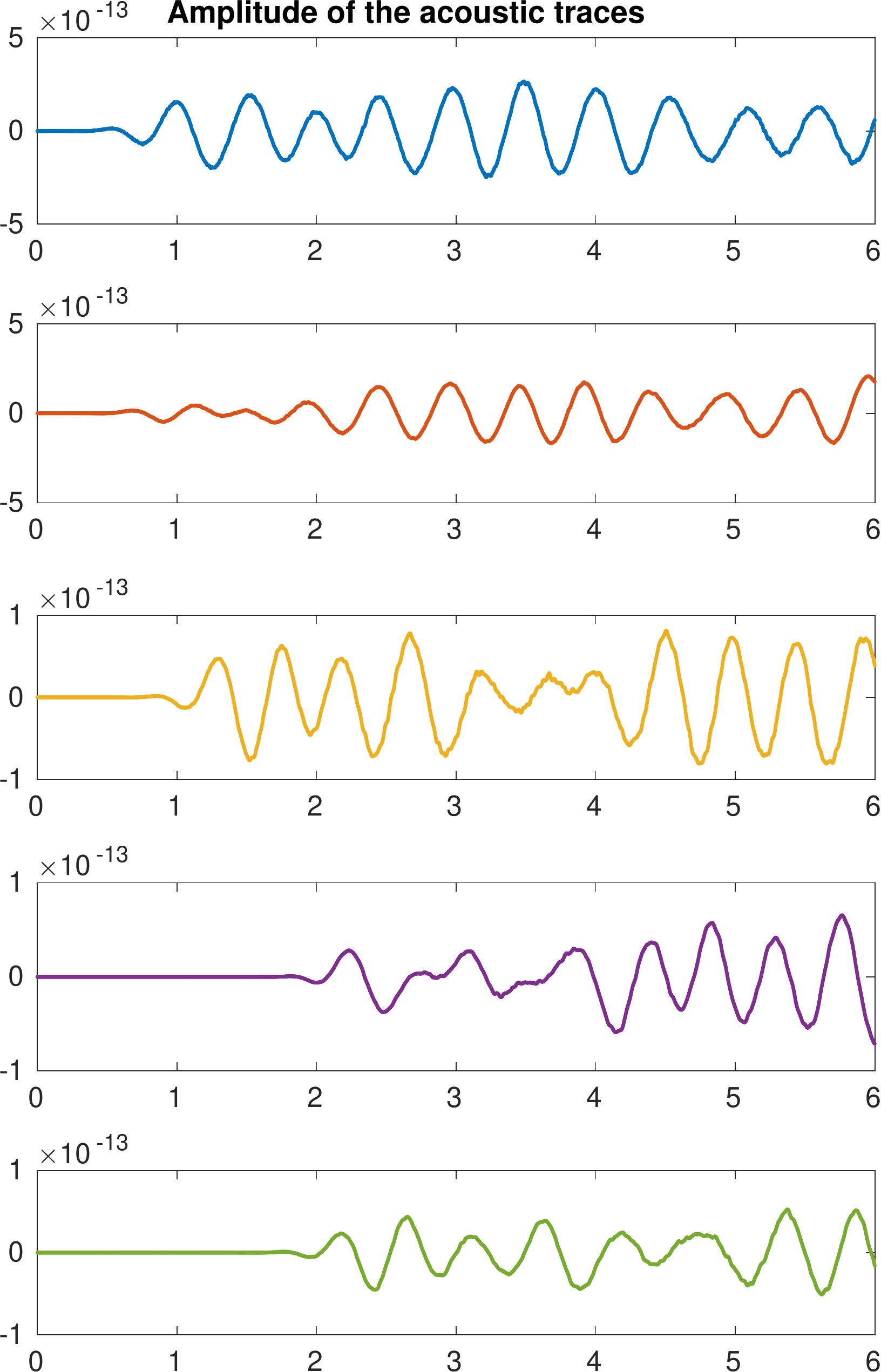} &
\includegraphics[width=.31\linewidth]{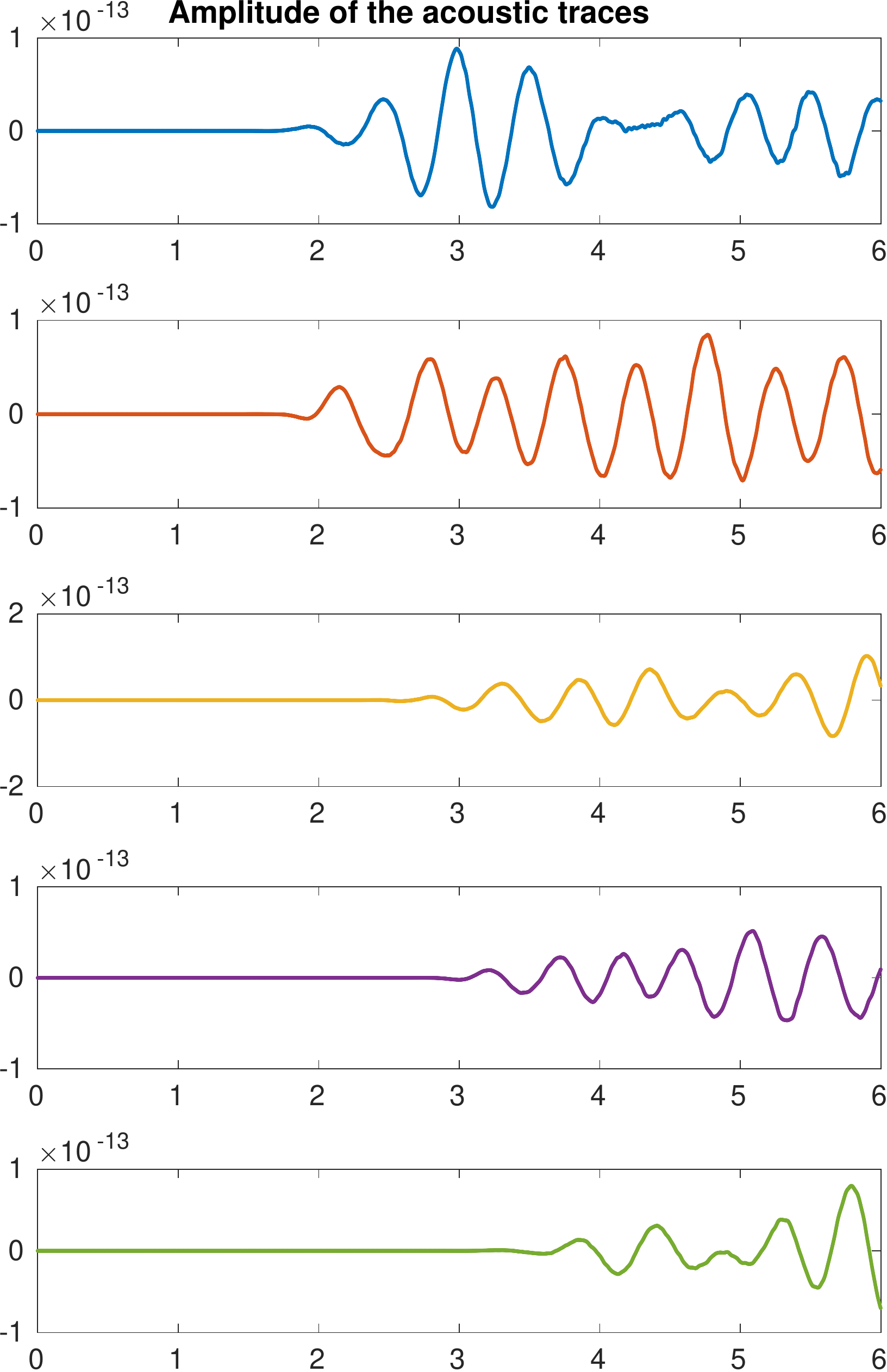} &
\includegraphics[width=.31\linewidth]{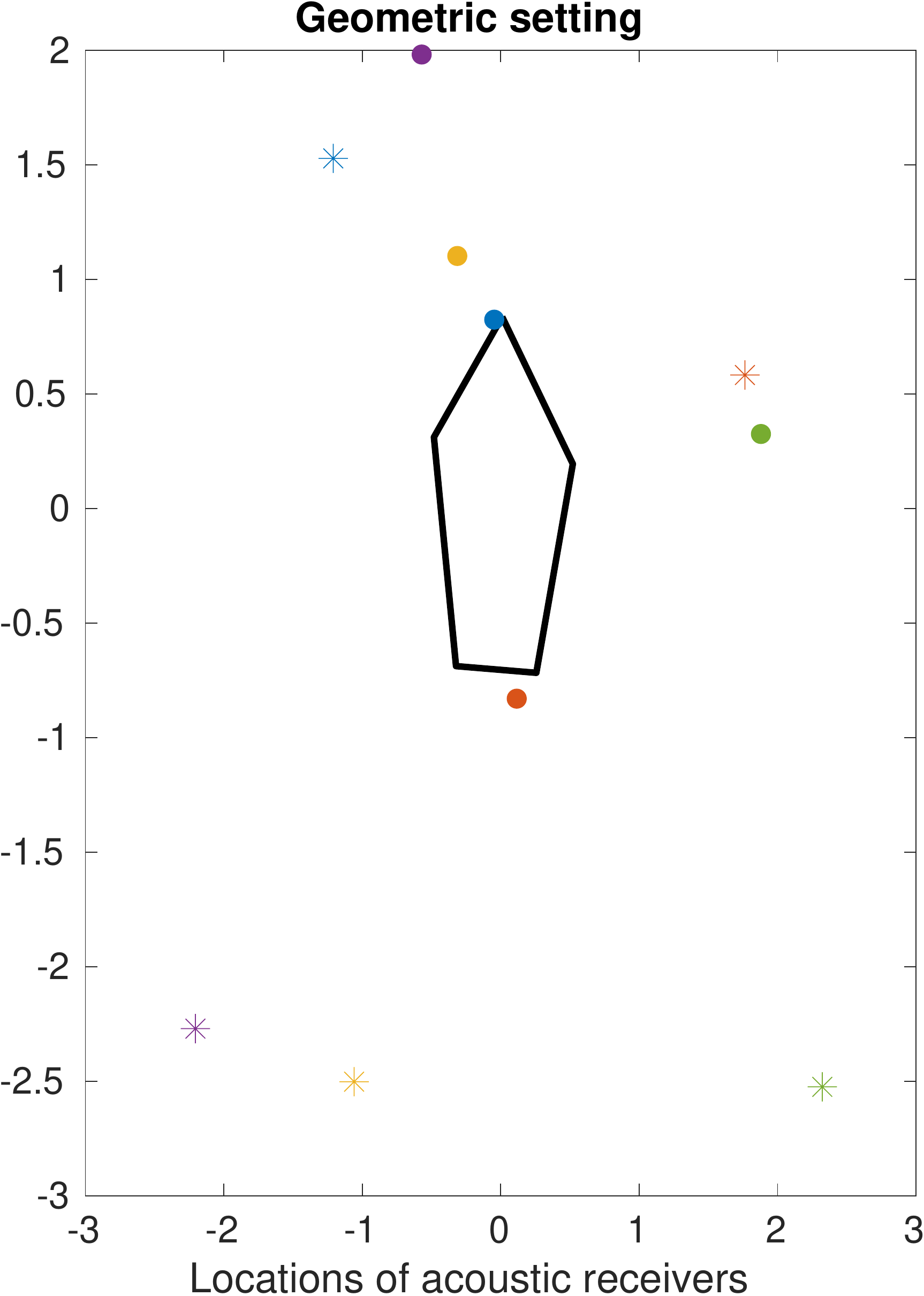}
\end{tabular}}
\caption{{\footnotesize The acoustic time-signatures on the left column correspond to receivers placed at the locations marked with a dot in the schematic, while the ones on the central column correspond to those marked with an asterisk. On the right: Locations of the acoustic receivers.}}\label{fig:7.10}
\end{figure}

\begin{figure}\center{\begin{tabular}{ccc}
\includegraphics[width=.31\linewidth]{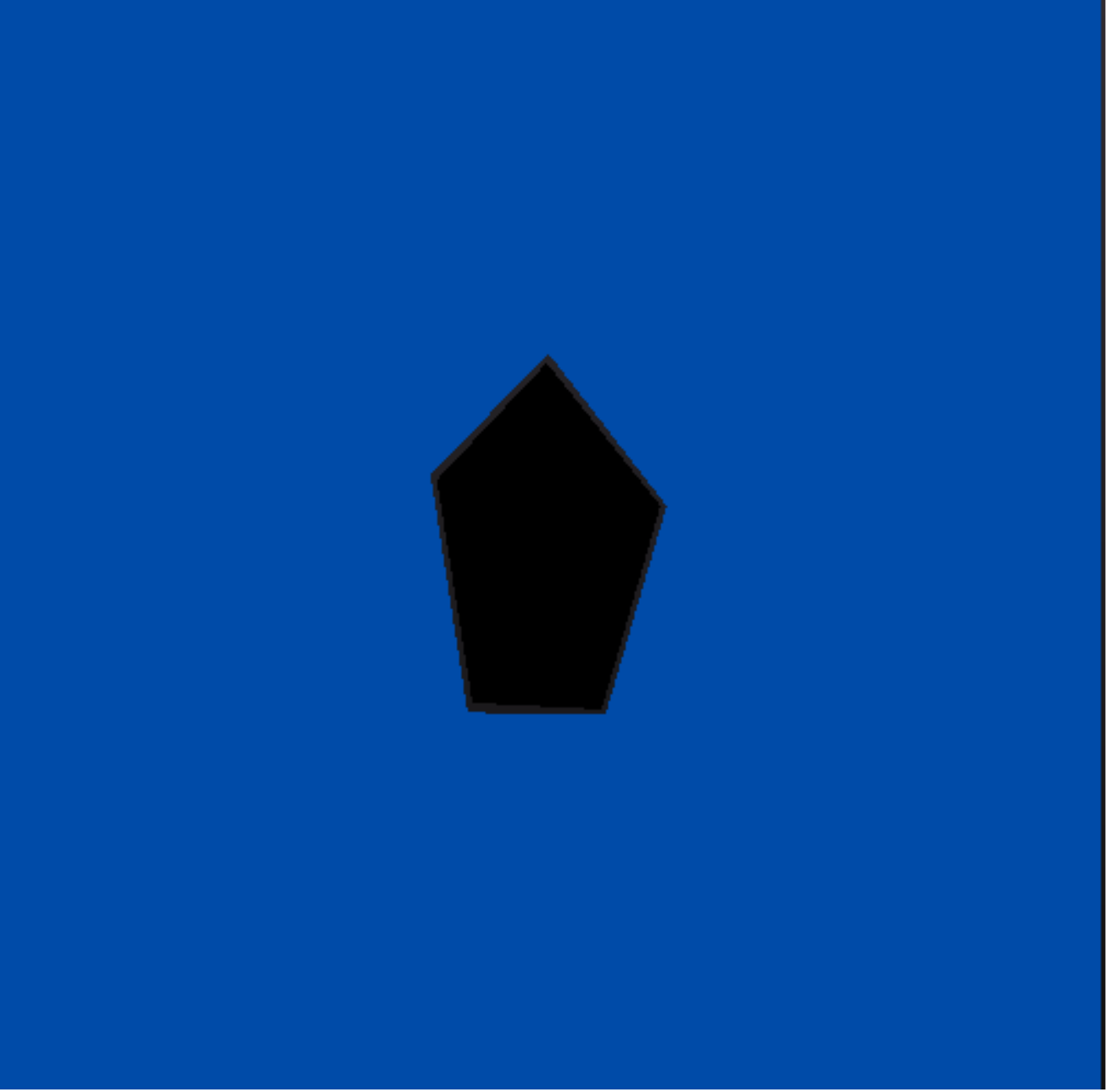}    &
\includegraphics[width=.31\linewidth]{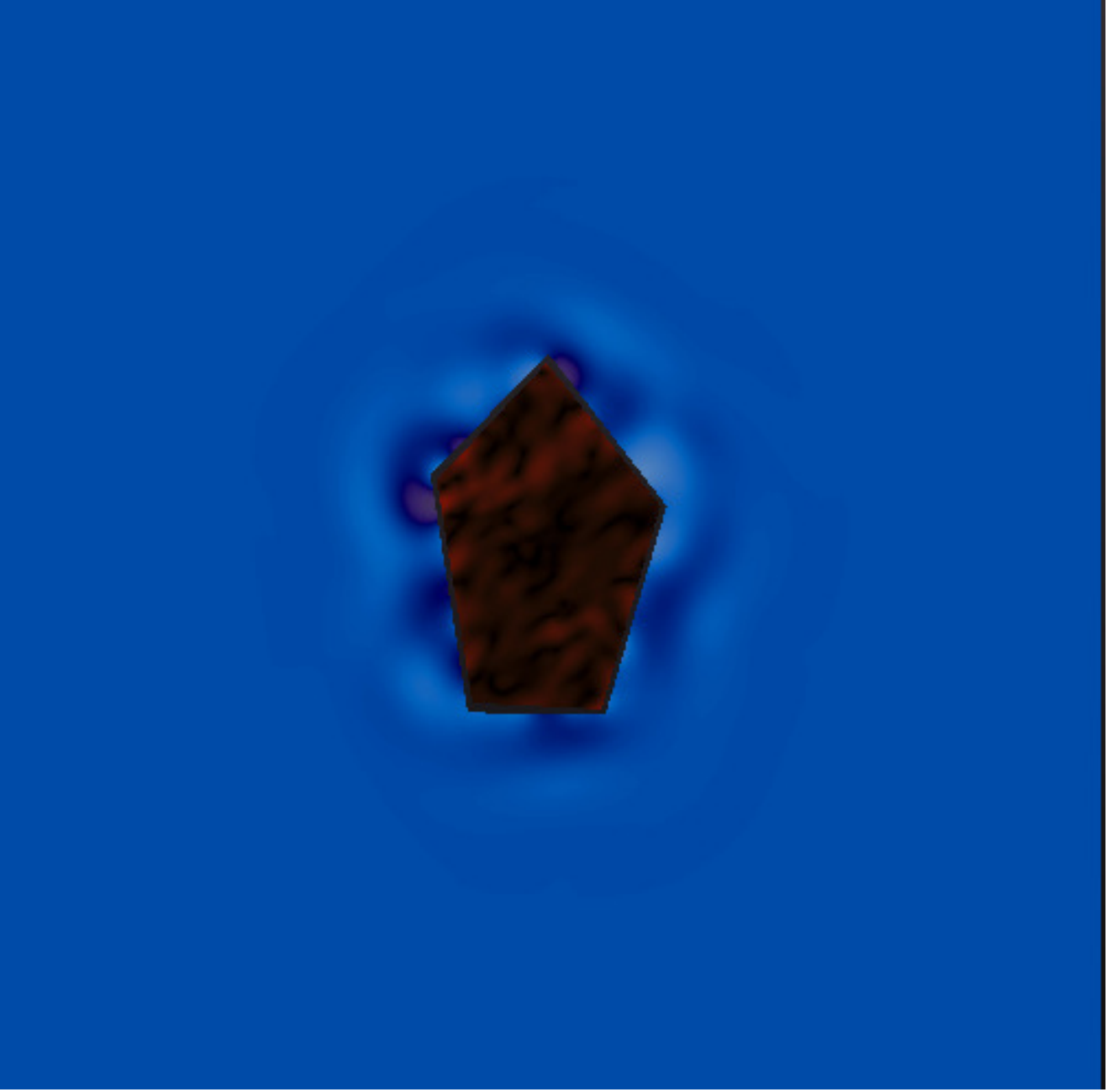}  &
\includegraphics[width=.31\linewidth]{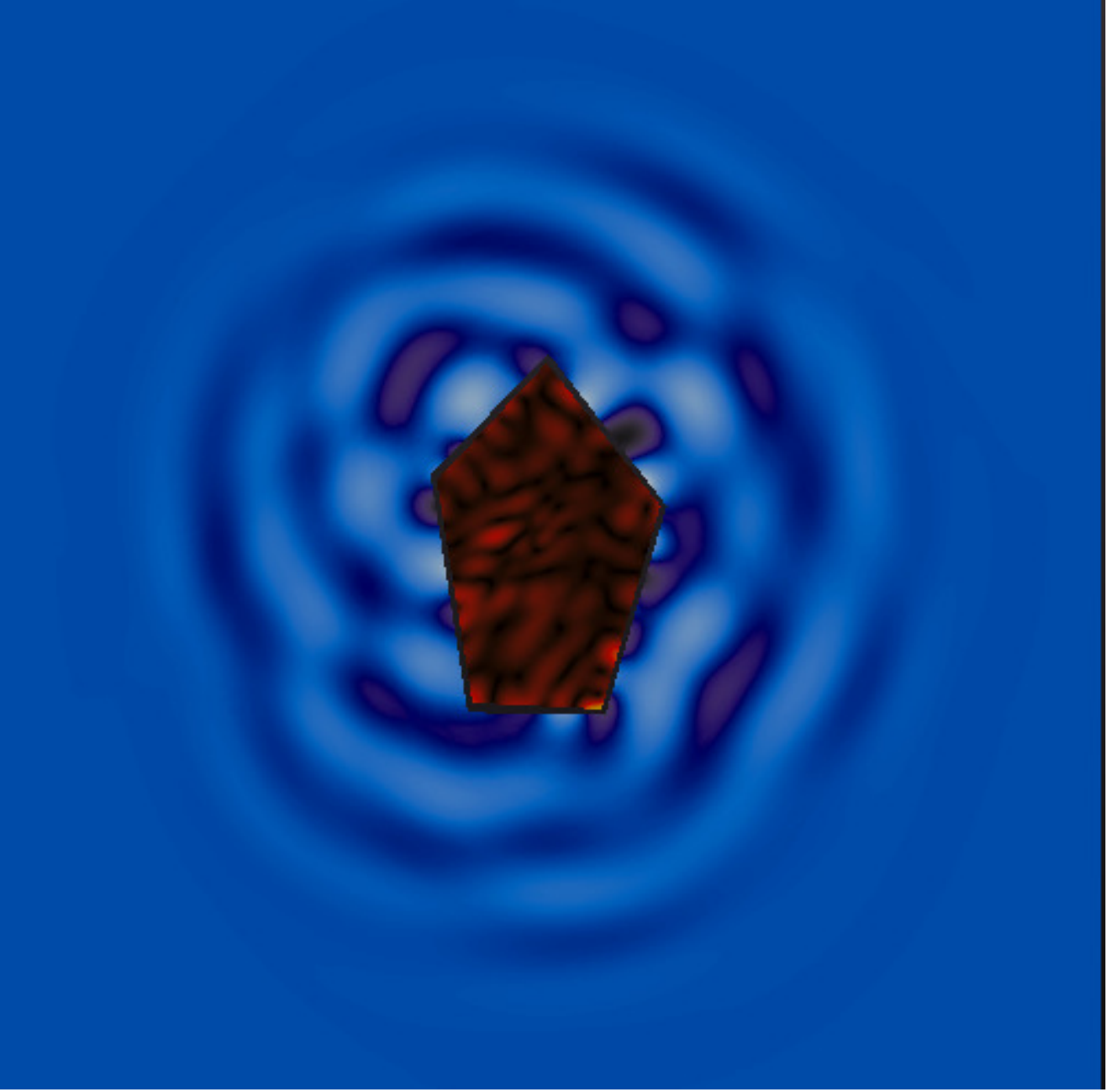} \\
\includegraphics[width=.31\linewidth]{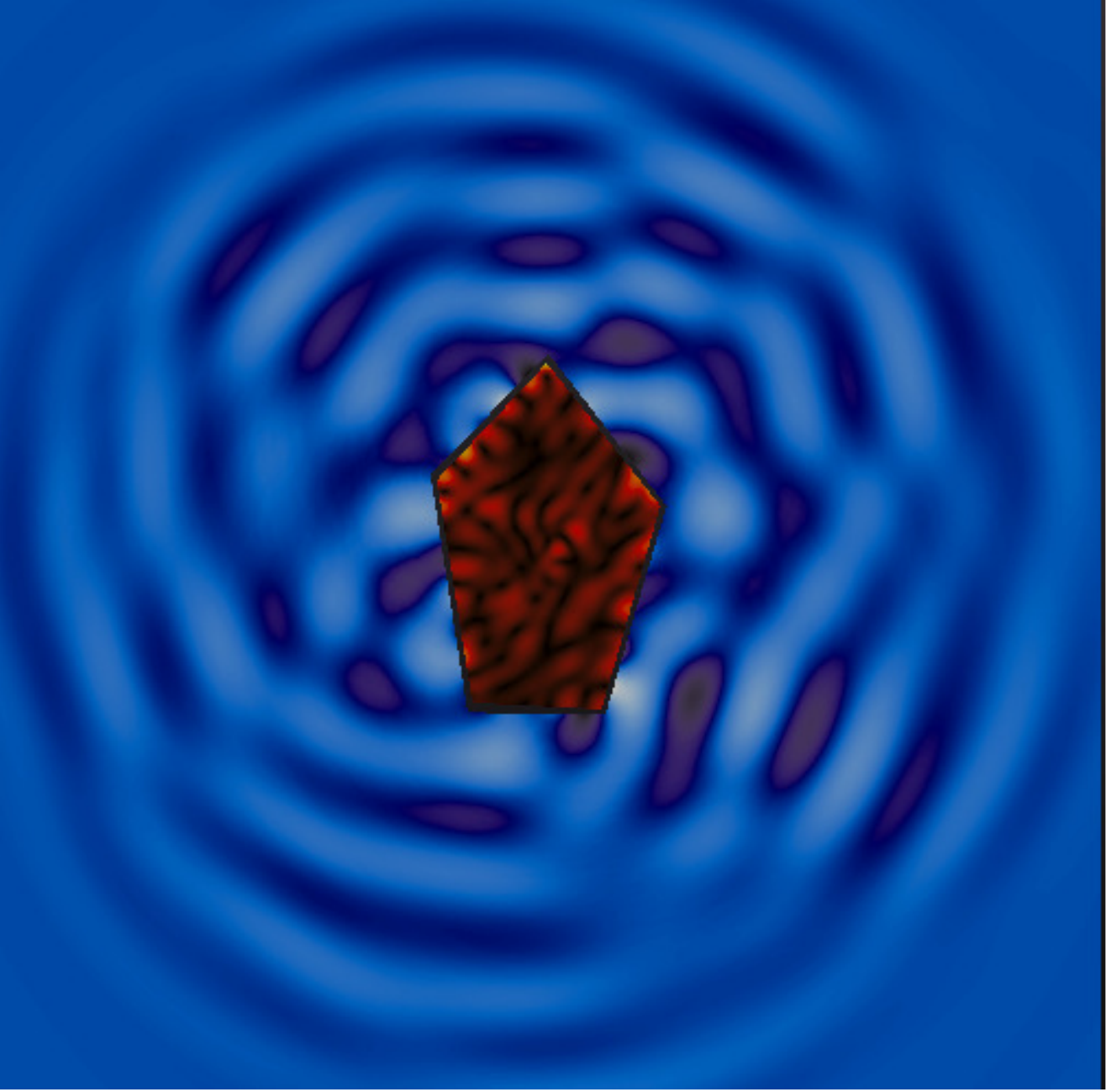} &
\includegraphics[width=.31\linewidth]{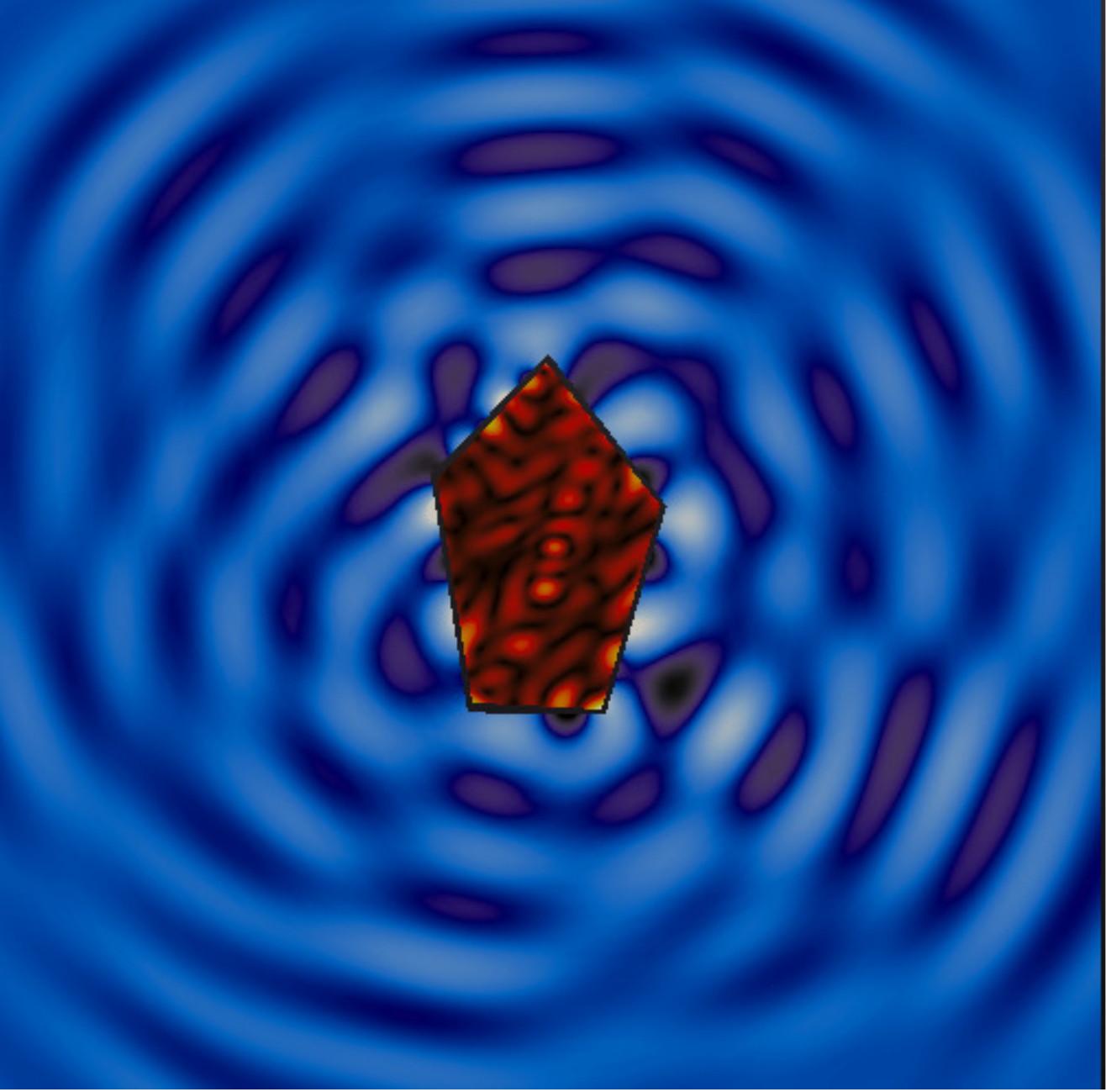} &
\includegraphics[width=.31\linewidth]{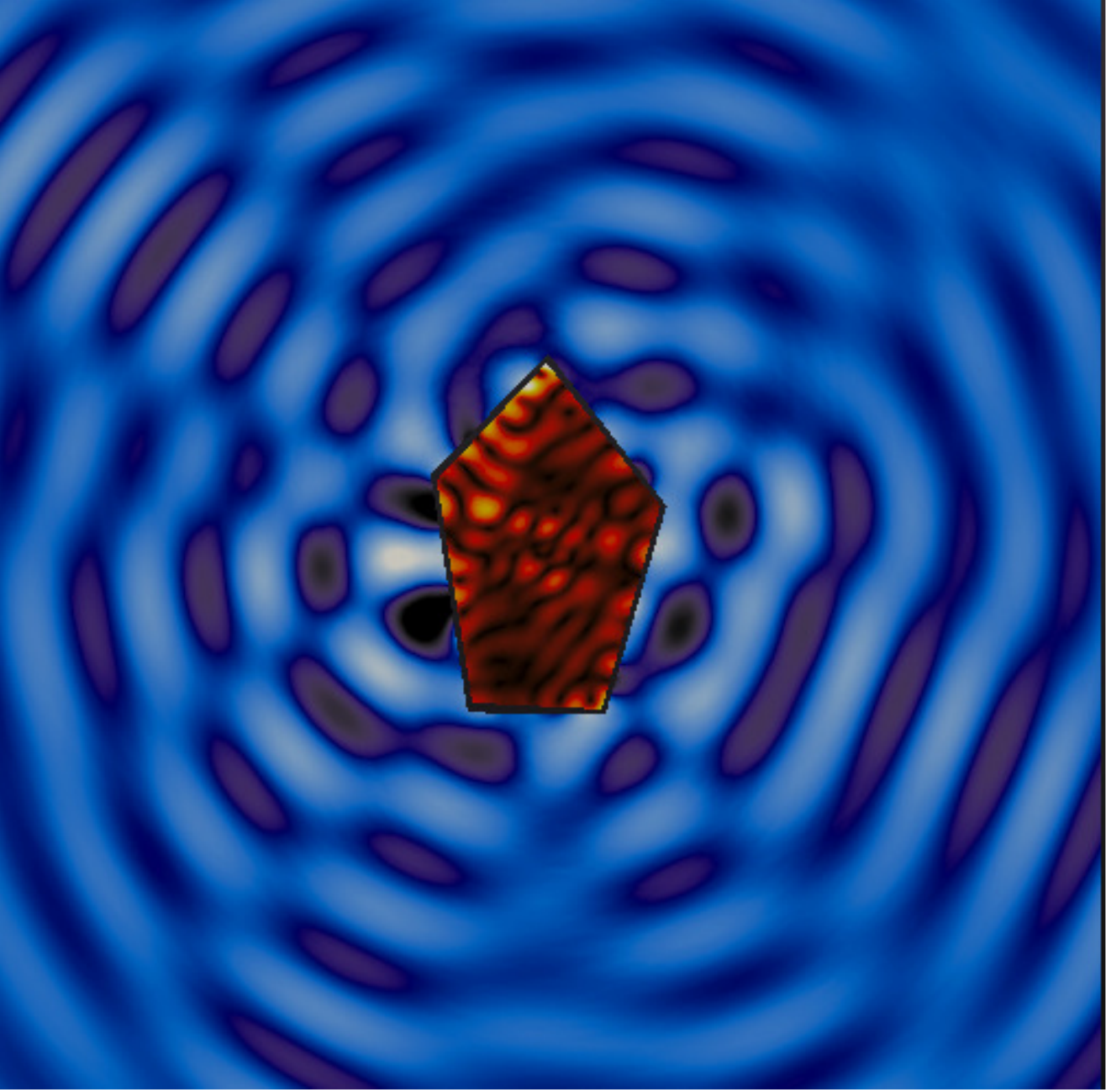}
\end{tabular}}
\caption{{\footnotesize A time-harmonic grounding potential generates internal elastic stresses that in turn produce an acoustic wavefield.  The figure presents snapshots of the total acoustic field and the magnitude of the elastic displacement for times $t=0, 0.9, 1.8, 2.7, 3.6, 6$.}}\label{fig:7.11}
\end{figure}

\paragraph{Concluding Remarks.}  
We have studied the problem of the scattering of acoustic waves by a piezoelectric solid, as well as its abstract semidiscretization in space. We have focused on the analysis in the time domain to obtain stability bounds (for both problems) and error estimates. The continuous problem, its semidiscretization, and the evolution of the approximation error are recast as an abstract differential equation of the first order.

The semidiscrete problem is presented as an abstract transmission problem, where the interior fields are posed in Galerkin form with exotic transmission conditions. However, we show it to be equivalent to a Galerkin approximation for a coupled formulation using retarded boundary integral operators and volume terms, that can be dealt with numerically as a coupled Finite Element-Boundary Element formulation.  

We have also presented some variations of the model and its semidiscretization: decoupling the electric potential from the problem by deactivating the piezoelectric tensor, adding a drag term to the elastic equation (thus losing energy conservation), or using exact equations in the interior domain with more exotic transmission conditions for coupling (thus having an equivalent system of semidiscrete integral equations modeling the entire problem).

Our approach improves previously known results from Laplace domain analysis.  Several numerical experiments have been presented in order to highlight some of the possibilities of our formulation.  Rather than convergence studies, we have opted to show that our method captures the physical attributes one might expect from solutions to these problems.  We again direct the reader who is interested in convergence studies to those that can be found in \cite{SaSa2016, HsSaSa2016}.

\section{Appendix: The Abstract Framework} \label{framework}

In this section we present the abstract framework which we follow to derive our results.  This framework is based on \cite{HaQiSaSa2015}, although it is slightly simplified in two ways: (a) the flipping operator $T$ allows us to avoid having to deal with two different signs for an associated elliptic problem (this is used to prove that $A$ and $-A$ are both maximal dissipative); (b) the hypothesis on the surjectivity of the boundary operator is hidden in a hypothesis on existence of solutions for an elliptic problem. We will use some concepts and results from the theory of contraction $C_0$-semigroups of operators in Hilbert spaces and their relation to abstract evolution equations \cite{Pazy1983}. 

We begin with the Hilbert spaces $\mathbb H, \mathbb V, \mathbb M_1$, and $\mathbb M_2$, with the property that $\mathbb V \subset \mathbb H$ with bounded injection.  On these spaces, we define  bounded linear operators $A_\star: \mathbb V \to \mathbb H$, $B: \mathbb V \to \mathbb M_2$, and $G: \mathbb M_1 \to \mathbb H$.  We assume that 
 there exist two positive constants, $C_1$ and $C_2$ such that
\[
C_1\|U\|_{\mathbb V} \leq \|A_\star U\|_{\mathbb H} + \|U\|_{\mathbb H} \leq C_2 \|U\|_{\mathbb V}, \qquad \forall U \in \mathbb V.
\]
We also define $A:= A_\star \big|_{\mathrm{Ker}\, B}$ and $D(A): = \mathrm{Ker}\, B$.  With these definitions, for each problem we need to verify the following hypotheses:
\begin{itemize}
\item For all $U \in D(A)$, $(AU, U)_{\mathbb H} = 0$.
\item There exists an isometric involution $T:\mathbb H \to \mathbb H$, which is a bijection in $D(A)$ and satisfies $TAU = -ATU$ for all $U \in D(A)$.
\item  For all data, which take the form of $F \in \mathbb H$, and $\Xi = (\xi, \chi) \in \mathbb M:= \mathbb M_1 \times \mathbb M_2$, the problem
\begin{equation} \label{eq:modelProb}
U = A_\star U + F + G\xi, \qquad BU = \chi,
\end{equation}
is solvable and there is a positive constant $C_\mathrm{lift}$, such that 
\[
\|U\|_{\mathbb V} \leq C_\mathrm{lift} \left(\|F\|_{\mathbb H} + \|\Xi\|_{\mathbb M} \right).
\]
\end{itemize}

Under these hypotheses, we have the following.  The operators $\pm A$ are dissipative by the homogeneity of $(AU, U)$.  This property, along with the linearity of all operators involved, also shows that if there is a solution to \eqref{eq:modelProb}, that solution is unique.  Since we take $\Xi \in \mathbb M$ to be arbitrary, we have the surjectivity of $B$.  If we take $\Xi = 0$ in \eqref{eq:modelProb}, then we see that $I-A$ is onto, hence $A$ is maximal dissipative.  Furthermore, with the existence of $T$ and the surjectivity of $I-A$, we prove the surjectivity of $I + A$, so that $-A$ is also maximal dissipative.  From this we can conclude that $A$ is the infinitesimal generator of a $C_0$-group of isometries in $\mathbb H$, corresponding to the evolution of initial conditions in the  problem
\[
\dot U(t)= A_\star U(t) \quad t\in \mathbb R, \qquad U(0)=U_0.
\]
The result deals mainly with strong solutions of a non-homogeneous initial value problem and it also recognizes their relation to causal distributional solutions of the same problem. This identification is needed to relate the problems under study to retarded potential integral formulations, where the vanishing past behavior (for negative times) of some of the fields is needed to make the definitions meaningful. Given a Banach space $X$, we consider the Sobolev spaces 
\begin{alignat*}{6}
W^k(X) &:= \{ f\in \mathcal C^{k-1}([0,\infty);X)\,:\, f^{(k)}\in L^1((0,\infty);X), \quad f^{(\ell)}(0)=0 \quad \ell\le k\},\\
W_+^k(X) & := \{f \in \mathcal C^{k-1}(\mathbb R, X): f^{(k)} \in L^1(\mathbb R; X),\quad f\equiv 0 \mbox{ in $(-\infty,0)$}\}.
\end{alignat*}

\begin{theorem} \label{bigTheorem}
\begin{enumerate}
\item[{\rm (a)}] If $F \in W^1(\mathbb H)$ and $\Xi = (\xi, \chi) \in W^2(\mathbb M)$, then there exists a unique $U \in \mathcal C^1([0, \infty), \mathbb H) \cap \mathcal C ([0, \infty), \mathbb V)$ which satisfies
\begin{subequations} \label{eq:modelProb2}
\begin{alignat}{3}
\dot{U}(t) & = A_\star U(t) + F(t) + G\xi(t), &\qquad &t \geq 0,\\
BU(t) &= \chi(t) &\qquad & t\geq 0,\\
U(0)&= 0,
\end{alignat}
\end{subequations}
and  the following estimates hold:
\begin{alignat*}{3}
\|U(t)\|_{\mathbb H} &\lesssim H_1(\Xi, t |\mathbb M) + \int_0^t \|F(\tau)\|_{\mathbb H} \; \mathrm d \tau,\\
\|\dot{U}(t)\|_{\mathbb H} &\lesssim H_2(\Xi, t |\mathbb M) + \int_0^t \|\dot{F}(\tau)\|_{\mathbb H} \; \mathrm d \tau.
\end{alignat*}
All constants hidden in the symbol $\lesssim$ depend exclusively on $C_{\mathrm{lift}}$. 
\item[{\rm (b)}] If $F \in W^2(\mathbb H)$ and $\Xi \in W^3(\mathbb M)$, then there exists a unique $U$ solving \eqref{eq:modelProb2}, with the additional bound
\[
\|\ddot{U}(t)\|_{\mathbb H} \lesssim H_3(\Xi, t |\mathbb M) + \int_0^t \|\ddot{F}(\tau)\|_{\mathbb H} \; \mathrm d \tau.
\]
\item[{\rm (c)}] If $F \in W_+^1(\mathbb H)$ and $\Xi \in W_+^2(\mathbb M)$ then the restriction of a distributional solution of 
\[
\dot{U} = A_\star U + F + G\xi, \qquad BU = \chi,
\]
to $[0, \infty)$ is the unique solution of $\eqref{eq:modelProb2}$.
\end{enumerate}

\end{theorem}

\paragraph{Remark.}  The above result still holds if $(AU,U) \leq 0$ for each $U \in D(A)$. In that case the operator $T$ does not exist and $A$ is a maximal dissipative infinitesimal generator of a \textit{contraction} $C_0$-semigroup of operators in $\mathbb H$. 

\bibliographystyle{abbrv} 
\bibliography{Refer}

\begin{thebibliography}{10}

\bibitem{AdFo2003}
R.~A. Adams and J.~J.~F. Fournier.
\newblock {\em Sobolev spaces}, volume 140 of {\em Pure and Applied Mathematics
  (Amsterdam)}.
\newblock Elsevier/Academic Press, Amsterdam, second edition, 2003.

\bibitem{AkNa2002}
M.~Akamatsu and G.~Nakamura.
\newblock Well-posedness of initial-boundary value problems for piezoelectric
  equations.
\newblock {\em Appl. Anal.}, 81(1):129--141, 2002.

\bibitem{BaHa1986}
A.~Bamberger and T.~H. Duong.
\newblock Formulation variationnelle pour le calcul de la diffraction d'une
  onde acoustique par une surface rigide.
\newblock {\em Math. Methods Appl. Sci.}, 8(4):598--608, 1986.

\bibitem{BaSa2012}
L.~Banjai and M.~Schanz.
\newblock Wave propagation problems treated with convolution quadrature and
  {BEM}.
\newblock In {\em Fast boundary element methods in engineering and industrial
  applications}, volume~63 of {\em Lect. Notes Appl. Comput. Mech.}, pages
  145--184. Springer, Heidelberg, 2012.

\bibitem{BiMa1991}
J.~Bielak and R.~C. MacCamy.
\newblock Symmetric finite element and boundary integral coupling methods for
  fluid-solid interaction.
\newblock {\em Quart. Appl. Math.}, 49(1):107--119, 1991.

\bibitem{Braess2007}
D.~Braess.
\newblock {\em Finite elements}.
\newblock Cambridge University Press, Cambridge, third edition, 2007.
\newblock Theory, fast solvers, and applications in elasticity theory,
  Translated from the German by Larry L. Schumaker.

\bibitem{ChNa2015}
G.~Chkadua and D.~Natroshvili.
\newblock Interaction of acoustic waves and piezoelectric structures.
\newblock {\em Math. Methods Appl. Sci.}, 38(11):2149--2170, 2015.

\bibitem{Cimatti2004}
G.~Cimatti.
\newblock The piezoelectric continuum.
\newblock {\em Ann. Mat. Pura Appl. (4)}, 183(4):495--514, 2004.

\bibitem{DeLaOh2009}
J.-F. De{\"u}, W.~Larbi, and R.~Ohayon.
\newblock Variational formulations of interior structural-acoustic vibration
  problems.
\newblock In G.~Sandberg and R.~Ohayon, editors, {\em Computational Aspects of
  Structural Acoustics and Vibration}, pages 1--21. Springer Vienna, Vienna,
  2009.

\bibitem{DoSa2003}
V.~Dom{\'{\i}}nguez and F.-J. Sayas.
\newblock Stability of discrete liftings.
\newblock {\em C. R. Math. Acad. Sci. Paris}, 337(12):805--808, 2003.

\bibitem{FlKaTrWo2010}
B.~Flemisch, M.~Kaltenbacher, S.~Triebenbacher, and B.~I. Wohlmuth.
\newblock The equivalence of standard and mixed finite element methods in
  applications to elasto-acoustic interaction.
\newblock {\em SIAM Journal on Scientific Computing}, 32(4):1980--2006, 2010.

\bibitem{HaSa2016}
M.~Hassell and F.-J. Sayas.
\newblock Convolution quadrature for wave simulations.
\newblock In {\em Numerical simulation in physics and engineering}, volume~9 of
  {\em SEMA SIMAI Springer Ser.}, pages 71--159. Springer, [Cham], 2016.

\bibitem{HaQiSaSa2015}
M.~E. Hassell, T.~Qiu, T.~S\'anchez-Vizuet, and F.-J. Sayas.
\newblock A new and improved analysis of the time domain boundary integral
  operators for the acoustic wave equation.
\newblock {\em J. Integral Equations Appl.}, 29(1):107--136, 2017.

\bibitem{HaSa2016b}
M.~E. Hassell and F.-J. Sayas.
\newblock A fully discrete {BEM}--{FEM} scheme for transient acoustic waves.
\newblock {\em Comput. Methods Appl. Mech. Engrg.}, 309:106--130, 2016.

\bibitem{HsSaSa2016}
G.~C. Hsiao, T.~S\'anchez-Vizuet, and F.-J. Sayas.
\newblock Boundary and coupled boundary--finite element methods for transient
  wave--structure interaction.
\newblock {\em IMA J. Numer. Anal.}, 37(1):237--265, 2017.

\bibitem{HsSaWe2015}
G.~C. Hsiao, F.-J. Sayas, and R.~J. Weinacht.
\newblock Time-dependent fluid-structure interaction.
\newblock {\em Math. Methods Appl. Sci.}, 40(2):486--500, 2017.

\bibitem{ImJo2012}
S.~Imperiale and P.~Joly.
\newblock Mathematical and numerical modelling of piezoelectric sensors.
\newblock {\em ESAIM Math. Model. Numer. Anal.}, 46(4):875--909, 2012.

\bibitem{LaSa2009}
A.~R. Laliena and F.-J. Sayas.
\newblock Theoretical aspects of the application of convolution quadrature to
  scattering of acoustic waves.
\newblock {\em Numer. Math.}, 112(4):637--678, 2009.

\bibitem{Lubich1994}
C.~Lubich.
\newblock On the multistep time discretization of linear initial-boundary value
  problems and their boundary integral equations.
\newblock {\em Numer. Math.}, 67(3):365--389, 1994.

\bibitem{McLean2000}
W.~McLean.
\newblock {\em Strongly elliptic systems and boundary integral equations}.
\newblock Cambridge University Press, Cambridge, 2000.

\bibitem{Pazy1983}
A.~Pazy.
\newblock {\em Semigroups of linear operators and applications to partial
  differential equations}, volume~44 of {\em Applied Mathematical Sciences}.
\newblock Springer-Verlag, New York, 1983.

\bibitem{SaSa2016}
T.~S\'anchez-Vizuet and F.-J. Sayas.
\newblock Symmetric boundary-finite element discretization of time dependent
  acoustic scattering by elastic obstacles with piezoelectric behavior.
\newblock {\em J. Sci. Comput.}, 70(3):1290--1315, 2017.

\bibitem{Sayas2007}
F.-J. Sayas.
\newblock Infimum-supremum.
\newblock {\em Bol. Soc. Esp. Mat. Apl. S$\vec{\rm e}$MA}, (41):19--40, 2007.

\bibitem{Sayas2016}
F.-J. Sayas.
\newblock {\em Retarded potentials and time domain integral equations: a
  roadmap}, volume~50 of {\em Springer Series in Computational Mathematics}.
\newblock Springer International Publishing, first edition, 2016.

\bibitem{Schwartz1966}
L.~Schwartz.
\newblock {\em Th\'eorie des distributions}.
\newblock Publications de l'Institut de Math\'ematique de l'Universit\'e de
  Strasbourg, No. IX-X. Nouvelle \'edition, enti\'erement corrig\'ee, refondue
  et augment\'ee. Hermann, Paris, 1966.

\bibitem{TaCh2013}
J.-Q. Tarn and H.-H. Chang.
\newblock A {H}amiltonian state space approach to anisotropic elasticity and
  piezoelasticity.
\newblock {\em Acta Mech.}, 224(6):1271--1284, 2013.

\end{thebibliography}

\end{document}